\documentclass[12pt]{article}
\usepackage{graphicx}
\usepackage{amsmath,amsthm,amssymb,enumerate}
\usepackage{euscript,mathrsfs}
\usepackage[left=2cm,right=2cm,top=3.5cm,bottom=3.5cm]{geometry}
\usepackage{color}
\usepackage{bm}
\usepackage{tikz}
\usetikzlibrary{matrix,arrows,decorations.pathmorphing,positioning,arrows.meta,quotes}
\usepackage{soul}
\usepackage{boldline}
\usepackage{dsfont}
\usepackage{multirow}
\usepackage{bbm}
\usepackage{subcaption}

\catcode`\@=11 \@addtoreset{equation}{section}

\catcode`\@=12

\allowdisplaybreaks

\usepackage{cancel}

\newtheorem{Theorem}{Theorem}[section]
\newtheorem{Proposition}[Theorem]{Proposition}
\newtheorem{Lemma}[Theorem]{Lemma}
\newtheorem{Corollary}[Theorem]{Corollary}

\theoremstyle{definition}
\newtheorem{Definition}[Theorem]{Definition}

\newtheorem{Remark}[Theorem]{Remark}

\newcommand{\bTheorem}[1]{
\begin{Theorem} \label{T#1} }
\newcommand{\eT}{\end{Theorem}}

\newcommand{\bProposition}[1]{
\begin{Proposition} \label{P#1}}
\newcommand{\eP}{\end{Proposition}}

\newcommand{\bLemma}[1]{
\begin{Lemma} \label{L#1} }
\newcommand{\eL}{\end{Lemma}}

\newcommand{\bCorollary}[1]{
\begin{Corollary} \label{C#1} }
\newcommand{\eC}{\end{Corollary}}

\newcommand{\bRemark}[1]{
\begin{Remark} \label{R#1} }
\newcommand{\eR}{\end{Remark}}

\newcommand{\Td}{{\mathcal{T}^d}}
\newcommand{\bDefinition}[1]{
\begin{Definition} \label{D#1} }
\newcommand{\eD}{\end{Definition}}

\newcommand{\toS}{\stackrel{(S)}{\longrightarrow}}
\newcommand{\toSC}{\stackrel{(S)}{\Longrightarrow}}

\newcommand{\Nu}{\mathcal{V}_{t,x}}

\newcommand{\rN}{\vr_{h, \mu_n}}
\newcommand{\mN}{\vm_{h, \mu_n}}
\newcommand{\rNN}{\vr_{h_n, \mu_n}}
\newcommand{\mNN}{\vm_{h_n, \mu_n}}

\newcommand{\eps}{\varepsilon}

\newcommand{\vc}[1]{{\bm #1}}

\newcommand{\jump}[1]{\left[ \!\left[ #1 \right]\! \right]}

\newcommand{\vrh}{\vr_h}

\newcommand{\vmh}{\vm_h}

\newcommand{\tvm}{\tilde{\vc{m}}}
\newcommand{\bfphi}{\boldsymbol{\varphi}}
\newcommand{\bfvarphi}{\boldsymbol{\varphi}}

\newcommand{\ds}{\,\mathrm{d}S_x}

\newcommand{\bFormula}[1]{
\begin{equation} \label{#1}}
\newcommand{\eF}{\end{equation}}

\newcommand{\facesint}{\mathcal{E}}
\newcommand{\grid}{\mathcal{T}}

\newcommand{\vuh}{\vu_h}

\newcommand{\intSh}[1] {\int_{\sigma} #1 \ds }

\newcommand{\Divh}{{\rm div}_h}
\newcommand{\Gradh}{\nabla_h}

\newcommand{\Ov}[1]{\overline{#1}}

\newcommand{\aleq}{\lesssim}

\newcommand{\vr}{\varrho}

\newcommand{\tvr}{\tilde \vr}

\newcommand{\vu}{\vc{u}}
\newcommand{\vm}{\vc{m}}

\newcommand{\vn}{\vc{n}}

\newcommand{\sumsfK}{\sum_{\sigma \in \pd K}}
\newcommand\Up{\mbox{\sl{Up}}}
\newcommand{\avs}[1]{\left\{\!\!\left\{ #1\right\}\!\!\right\}}

\newcommand{\muh}{h^\varepsilon}

\newcommand{\vvh}{\bm{v}_h}
\newcommand{\pd}{\partial}
\newcommand{\Divup}{{\rm div}_h^{\up}}
\newcommand{\up}{{\rm up}}

\newcommand{\gradd}{\nabla_{\mathcal D}}

\newcommand{\norm}[1]{\left\lVert#1\right\rVert}

\newcommand{\intTd}[1]{\int_{\Td} #1 \, \dx}

\newcommand{\vQh}{\vc{Q}_h}

\newcommand{\Div}{{\rm div}_x}
\newcommand{\Grad}{\nabla_x}

\newcommand{\dx}{\,{\rm d} {x}}

\newcommand{\dt}{\,{\rm d} t }

\newcommand{\intO}[1]{\int_{\Td} #1 \ \dx}

\newcommand{\intQ}[1]{\int_{Q} #1 \ {\rm d}y}

\newcommand{\vv}{\vc{v}}

\newcommand{\D}{{\rm d}}
\newcommand{\ep}{\varepsilon}

\def\softd{{\leavevmode\setbox1=\hbox{d}%
          \hbox to 1.05\wd1{d\kern-0.4ex{\char039}\hss}}}
\definecolor{Cgrey}{rgb}{0.85,0.85,0.85}

\newcommand\Cbox[2]{%
    \newbox\contentbox%
    \newbox\bkgdbox%
    \setbox\contentbox\hbox to \hsize{%
        \vtop{
            \kern\columnsep
            \hbox to \hsize{%
                \kern\columnsep%
                \advance\hsize by -2\columnsep%
                \setlength{\textwidth}{\hsize}%
                \vbox{
                    \parskip=\baselineskip
                    \parindent=0bp
                    #2
                }%
                \kern\columnsep%
            }%
            \kern\columnsep%
        }%
    }%
    \setbox\bkgdbox\vbox{
        \color{#1}
        \hrule width  \wd\contentbox %
               height \ht\contentbox %
               depth  \dp\contentbox
        \color{black}
    }%
    \wd\bkgdbox=0bp%
    \vbox{\hbox to \hsize{\box\bkgdbox\box\contentbox}}%
    \vskip\baselineskip%
}

\date{}

\begin{document}


\title{Approximating viscosity solutions of the Euler system}

\author{Eduard Feireisl\thanks{The research of E.F., B.S. leading to these results has received funding from the
Czech Sciences Foundation (GA\v CR), Grant Agreement 21-02411S. The Institute of Mathematics of the Academy of Sciences of
the Czech Republic is supported by RVO:67985840.\newline
\hspace*{1em} $^\spadesuit$M.L. has been funded by the Deutsche Forschungsgemeinschaft (DFG, German Research Foundation) - Project number 233630050 - TRR 146 as well as by  TRR 165 Waves to Weather. She is grateful to the Gutenberg Research College for supporting her research.
The research of S.Sch. was funded by Mainz Institute of Multiscale Modelling.} $^{, \clubsuit}$
\and M\' aria Luk\' a\v cov\' a -- Medvi\softd ov\' a$^{\spadesuit}$
\and Simon Schneider$^{\spadesuit}$ \and Bangwei She$^{*}$
}


\maketitle

\bigskip

\centerline{$^*$ Institute of Mathematics of the Czech Academy of Sciences}
\centerline{\v Zitn\' a 25, CZ-115 67 Praha 1, Czech Republic}
\centerline{feireisl@math.cas.cz, she@math.cas.cz}

\bigskip

\centerline{$^\clubsuit$ Institute of Mathematics, TU Berlin}
\centerline{Strasse des 17. Juni, Berlin, Germany}

\bigskip
\centerline{$^\spadesuit$ Institute of Mathematics, Johannes Gutenberg-University Mainz}
\centerline{Staudingerweg 9, 55 128 Mainz, Germany}
\centerline{lukacova@uni-mainz.de, sschne15@uni-mainz.de}

\begin{abstract}
Applying the concept of S-convergence, based on averaging in the spirit of Strong Law of Large Numbers, the vanishing viscosity solutions
of the Euler system are studied. We show how to efficiently compute a viscosity solution of the Euler system as the S-limit of numerical solutions obtained by the Viscosity Finite Volume method.
Theoretical results are illustrated by numerical simulations of the Kelvin--Helmholtz instability problem.

%

\end{abstract}



\section{Introduction}
\label{i}

The method of convex integration, adapted to problems in fluid mechanics by Buckmaster, De Lellis, Isett, Sz\' ekelyhidi or Vicol
\cite{BuDeSzVi, DelSze13, Ise}
to name only a few, produced a large piece of evidence that the Euler system in fluid mechanics is ill posed, see also the survey paper \cite{BucVic} and the references cited therein. Another argument supporting ill--posedness of the incompressible Euler system
was presented recently by Bressan and Murray \cite{BresMur}. Although one may still hope that the \emph{incompressible} Euler
system is well posed in the class of strong solutions, see however Elgindi and Jeong \cite{ElgJeo}, this is definitely not the case if compressibility of the fluid is taken into account. Indeed Chiodaroli, De Lellis, and Kreml \cite{ChiDelKre} provided an example of Lipschitz initial data for which the isentropic Euler system possesses infinitely many admissible weak solutions on a sufficiently long time lap. More recently, Chiodaroli et al. \cite{ChKrMaSwI} identified even smooth ($C^\infty$) initial data yielding similar results.

As suggested by the numerical experiments of Elling \cite{ELLI1}, the ``wild" solutions obtained through
the abstract approach of convex integration may be physically relevant. The carbuncle solutions described in \cite{ELLI1} departing
for the initial data given by an admissible stationary state but being non--stationary may well represent a branch of energy dissipating ``wild'' solutions of the compressible Euler system predicted by convex integration.

In view of these facts, the Euler system being a model of perfect (ideal) fluid should be viewed in a broader context as an asymptotic limit of more complex systems describing
\emph{real} fluids including the effect of viscosity. Following the original idea of DiPerna and
Majda \cite{DipMaj87,DipMaj87a,DIMA} we identify a \emph{viscosity solution} of the Euler system with a parametrized family of probability measures $\{ \mathcal{V}_{t,x} \}_{t \in [0,T], x \in \Omega}$ generated by
solutions of the Navier--Stokes system in the vanishing viscosity limit. Here $t$ is the time and $x$ the spatial coordinate in the
domain $\Omega \subset \mathds{R}^d$ occupied by the fluid.

This process involves eventually approximation of the initial data for the Euler system. Here we should keep in mind that the vanishing viscosity limit may sensitively depend on the relation between the rate of convergence of the viscosity coefficients and the choice of the initial data approximation. As the Navier--Stokes system admits
global in time weak solutions for any finite energy initial data relevant for the limit Euler system, this issue can be avoided
by considering the same initial data for both systems. Another possibility is to focus on smooth initial data, for which the convergence is unconditional even for data perturbations at least on a short time interval. Both alternatives will be discussed below.

To separate the properties of $\mathcal{V}_{t,x}$ inherited from the generating sequence from those intrinsic to the limit Euler system,
we identify a family of \emph{observables} -- functions of the state variables having the same expectation at any
set of the physical space $(0,T) \times \Omega$ for all dissipative (viscosity) solutions $\mathcal{V}_{t,x}$, generated by the vanishing viscosity limit.
There are several plausible scenarios of the vanishing viscosity limit:
\begin{itemize}
\item
{\bf Oscillatory (weak) limit.} The generating sequence of solutions of the viscous problem is merely bounded and converges weakly (in the sense of integral averages). The oscillations of the generating sequence can be described by a Young measure. In this case, the limit is not expected to be
a weak solution of the Euler system, cf. \cite{MarEd}.

\item {\bf Statistical (strong) limit.} There is a hidden regularizing effect acting in the vanishing viscosity limit so that
the generating sequence is precompact in the strong $L^p-$topology. This phenomenon is intimately related to the celebrated Kolmogorov
hypothesis advocated by Chen and Glimm \cite{CheGli} discussed in Section~\ref{K}. The generating sequence is precompact but still may admit a non--trivial set of accumulation points. The convergence is understood in a statistical sense and may be described
by a suitable measure sitting on the set of weak solutions of the Euler system.

\item {\bf Unconditional (strong) limit.}
The generating sequence converges strongly to a single limit. In {particular,} this is the case when the limit Euler system admits a (unique) strong solution.

\end{itemize}

Of course, the vanishing viscosity limit may exhibit a mixture of the first and second alternative as the case may be.
Our main goal is to propose a method how to approximate/compute efficiently
the viscosity solutions, in particular
the observable quantities in the case of oscillatory and/or statistical limit, that would be compatibly with the strong limit scenario
as soon as it takes place. Note that,
by virtue of the convergence result of Chen and Perepelitsa \cite{ChenPer1}, the \emph{oscillatory} scenario is excluded for problems in the spatial dimension $d=1$. On the other hand,
oscillatory/statistical solutions are observed in numerical approximations of
problems with unstable initial data, notably as perturbations of shock {wave} solutions, cf.~Elling \cite{ELLI2, ELLI3}
or emerging from the shear flow of the Kelvin--Helmholtz instability, cf.~\cite{FeiLukMizSheWa} or Fjordholm et al.~\cite{FjKaMiTa, FjMiTa1}.

While a viscosity solution
of the Euler system
can be identified with a Young measure generated by a sequence of solutions of the Navier--Stokes system, its numerical approximation
is calculated via the method of S--convergence proposed in \cite{Fei2020A}. Given a sequence of
approximate (vanishing viscosity or numerical) solutions $\{ \vc{U}_n \}_{n=1}^\infty$, we consider a family of probability measures
\begin{equation} \label{i1}
\mathcal{V}_N = \frac{1}{N} \sum_{n=1}^N \delta_{\vc{U}_n}, \ \delta_{\vc{U}_n} - \mbox{the Dirac mass at}\ \vc{U}_n,
\end{equation}
or, alternatively, the associated family of \emph{parametrized} probability measures
\begin{equation} \label{i2}
\left\{ \mathcal{V}_N \right\}_{(t,x) \in (0,T) \times \Omega} = \frac{1}{N} \sum_{n=1}^N \delta_{\vc{U}_n(t,x)}.
\end{equation}
Note carefully the subtle difference between \eqref{i1} and \eqref{i2}. The measure {$\mathcal{V}_N$} is a probability measure on
the \emph{infinite} dimensional space of trajectories -- solutions of the approximate problem. The family $\left\{ \mathcal{V}_N \right\}_{(t,x) \in (0,T) \times \Omega}$ consists of probability measures on the \emph{finite} dimensional phase space --
the range of the approximate solutions. In both cases, the limit for $N \to \infty$ fits in the framework of S--convergence
developed in \cite{Fei2020A}. If suitable uniform bounds are available for
$\{ \vc{U}_n \}_{n=1}^\infty$, then the sequence $\{ \mathcal{V}_N \}_{N=1}^\infty$ is tight and converges narrowly
modulo a suitable subsequence
to a limit $\mathcal{V}$ called
S--limit. Here again, there is a conceptual difference between \eqref{i1} and \eqref{i2}. The converging subsequence in \eqref{i1}
is obtained in terms of $N$ leaving the original sequence $\{ \vc{U}_n \}_{n=1}^\infty$ unchanged, while in \eqref{i2}, the
subsequence is selected from $\{ \vc{U}_n \}_{n=1}^\infty$ in the spirit of Banach--Saks theorem.

The bulk of this paper is to study the asymptotic behavior of the parametrized measures \eqref{i2}. We consider a more general version,
\begin{equation} \label{i3}
\left\{ \mathcal{V}_N \right\}_{(t,x) \in (0,T) \times \Omega} = \frac{1}{N} \sum_{n=1}^N s_{n,N} \delta_{\vc{U}_n(t,x)},
\end{equation}
where $\{ s_{n,N} \}$ is a suitable \emph{summation method}. In contrast with the conventional Young measure obtained
directly as a \emph{weak} limit of the generating sequence, the S--limit in \eqref{i2} is
strong (a.a.) with respect to the independent variables $(t,x)$. This makes
the limit object ``visible'' in numerical experiments. A similar construction based on the Ces\` aro averages ($s_{n,N} \equiv 1$) was used by Balder \cite{Bald} to approximate conventional Young measures.

Our goal is to visualize -- compute effectively the limit measure in \eqref{i3}. To this end,
we propose a numerical method called \emph{viscosity finite volume} (VFV) method of \emph{hybrid} type that is not purely Euler oriented but
involves the ghost effect of physical viscosity inherited from the sequence generating the viscosity solution. The original VFV method was introduced in
\cite{ FeiLMMiz, FLM18_brenner} and can be seen as a discretization (with vanishing viscosity) of the model proposed in a series of papers
by Brenner \cite{BREN2, BREN,BREN1}, see also Guermond and Popov \cite{GuePop}.  For the case of barotropic Euler system the method involves
three vanishing, mutually interrelated parameters: the numerical step $h$, the shear viscosity coefficient $\mu$, and the bulk
viscosity coefficient $\lambda$.
We show that if $\mu = \mu(h)$, $\lambda = \lambda(h)$ are chosen such that
\[
0< h << \mu(h),\ \mu(h) \to 0,  \lambda(h) \to 0 \ \mbox{as}\ h \to 0,
\]
then the VFV method produces a sequence of approximate solutions that admits an S--limit $\mathcal{V}$ that coincides with a viscosity solution of the Euler system. The latter will be precisely described in Definition~\ref{PD3}. The corresponding limiting processes are depicted in the flow chart below. In particular, the observables are uniformly approximated in the strong topology of the Lebesgue space $L^1((0,T) \times \Omega)$.

\bigskip

%
%

\begin{figure*}[!h]
\centering
\begin{tikzpicture}[scale=1,    box/.style = {draw, rounded corners,
                 minimum width=22mm, minimum height=5mm, align=center},
            > = {Latex[width=2mm,length=3mm]}]

            \draw[fill=Cgrey,Cgrey] (0,0.6) rectangle (\textwidth,-2.9);
\node[draw,  thick] at (0.12\textwidth, 0)   (n4)  [box]{{\bf VFV scheme}};
\node[draw,  thick] at (0.48\textwidth, -2.4)   (n5)  [box]  {{\bf Navier--Stokes system}};
\node[draw,  thick] at (0.88\textwidth, 0)   (n6)  [box]   {{\bf Euler system}};

\path[thick,->]
            (n4)    edge  node[sloped, anchor=center, above] {strong limit}          (n5)
            (n4)    edge  node[sloped, anchor=center, below] { $h \to 0$}          (n5)
            (n4)    edge  node[sloped, anchor=center, above] { S--limit}     (n6)
            (n4)    edge  node[sloped, anchor=center, below] { $h \to 0; \mu(h) \to 0; \lambda(h) \to 0$}     (n6)
            (n5)    edge  node[sloped, above]{\hspace{-0.8cm}vanishing viscosity limit  } (n6)
            (n5)    edge  node[sloped, anchor=center, below]{$\mu \to 0 ; \lambda  \to 0$ } (n6);
\end{tikzpicture}
\end{figure*}
\bigskip

The following are the main topics discussed in the present paper:
\begin{itemize}

\item In Section \ref{P}, we recall certain properties of the compressible {(barotropic)} Navier--Stokes system, in particular in the vanishing viscosity regime. Then we introduce the concept of viscosity solution to the Euler system.

\item In Section \ref{s}, we revisit the theory of S--convergence introduced in \cite{Fei2020A} and discuss
its relation to strong and weak convergence. We also show the \emph{subsequence principle} claiming that any bounded sequence admits
an S--convergent subsequence.

\item In Section \ref{e}, we apply the abstract results to the barotropic Euler system. We identify the viscosity solutions
with a Young measure generated in the vanishing viscosity limit of the Navier--Stokes problem. We  also introduce the concept of
\emph{generating sequence} for a viscosity solution to be approximated later by a numerical scheme.

\item In Section \ref{N}, we introduce the VFV method, discuss its structure preserving properties and show convergence to a viscosity solution of the Euler system. Theoretical results are illustrated by numerical experiments at the end of Section~\ref{N}.

\item Finally, the impact of Kolmogorov hypothesis is discussed in Section~\ref{K}.

\end{itemize}

\section{Euler and Navier--Stokes systems, viscosity solution}
\label{P}

Consider a bounded regular domain $\Omega \subset \mathds{R}^d$, $d = 2,3$ occupied by a fluid of a mass density
$\vr = \vr(t,x)$ moving with the velocity $\vu = \vu(t,x)$, where $x \in \Omega$ and $t \in [0,T]$ is the time.
We suppose that the fluid is \emph{perfect} and ignore thermal effects.
Accordingly, the time evolution of the system is governed by the barotropic \emph{Euler system}:
\begin{equation} \label{P1}
\begin{split}
\partial_t \vr + \Div \vm &= 0,\ \vm \equiv \vr \vu,\\
\partial_t \vm + \Div \left( \frac{\vm \otimes \vm}{\vr} \right) + \Grad p(\vr) &= 0,
\end{split}
\end{equation}
that may be
supplemented with the \emph{impermeability boundary condition}
\begin{equation} \label{P2}
\vm \cdot \vc{n}|_{\partial \Omega} = 0,\ \vc{n}- \mbox{the outer normal vector.}
\end{equation}
For the sake of simplicity, we focus on the isentropic pressure--density equation of state,
$p(\vr) \equiv a \vr^\gamma,\ a > 0, \ \gamma > 1$.

Real fluids are \emph{viscous}, and, in accordance with the Second law of thermodynamics, they dissipate mechanical energy. Supposing
the simplest linear (Newtonian) relation between the viscous stress and the symmetric velocity gradient,
we consider the \emph{Navier--Stokes {\rm (NS)} system}:
\begin{equation} \label{P3}
\begin{split}
\partial_t \vr + \Div (\vr \vu) &= 0,\\
\partial_t (\vr \vu) + \Div \left( \vr \vu \otimes \vu \right) + \Grad p(\vr) &= \Div \mathbb{S} (\Grad \vu),\\
\mathbb{S}(\Grad \vu) &= \mu \left( \Grad \vu + \Grad \vu^t - \frac{2}{d} \Div \vu \mathbb{I} \right) +
\lambda \Div \vu \mathbb{I},\ \mu > 0,\ \lambda \geq 0.
\end{split}
\end{equation}

As mentioned above, our approach is based on identifying suitable solutions of the Euler system \eqref{P1} as limits
of \eqref{P3} in the regime of vanishing viscosity coefficients $\mu$ and $\lambda$. As is well known, see e.g. the survey by E \cite{E1},
this involves the problem of a boundary layer that may be created depending on the choice of the boundary conditions for the
Navier--Stokes system, notably in the popular case of \emph{no--slip}
\[
\vu|_{\partial \Omega} = 0.
\]
To avoid this difficulty, the \emph{complete slip} conditions
\begin{equation} \label{P4}
\vu \cdot \vc{n}|_{\partial \Omega} = 0,\ [ \mathbb{S}(\Grad \vu) \cdot \vc{n} ] \times \vc{n}|_{\partial \Omega} = 0
\end{equation}
can be imposed. In view of future numerical implementation, we simplify even more by considering the \emph{space periodic
boundary conditions} for both the Euler and the Navier--Stokes system. In other words, we identify the spatial domain $\Omega$ with
the flat torus,
\begin{equation} \label{P5}
\Omega \equiv \Td = \left(  [0,1] |_{\{ 0, 1 \}} \right)^d,\ d = 2,3.
\end{equation}
Note that the slip condition \eqref{P4} can be equivalently reformulated in the periodic setting as soon as $\Omega$ is a cuboid, see
Ebin \cite{EB}.

\subsection{Weak solutions to the Navier--Stokes system}

The existence theory for the Navier--Stokes system \eqref{P3} has been developed by Lions \cite{LI4} and extended
in \cite{EF70} to accommodate a larger range of the adiabatic coefficient $\gamma > \frac{d}{2}$. The limit case $\gamma = 1$, $d=2$ has been finally settled by Plotnikov and Weigant \cite{PloWei}.

\begin{Definition}[{\bf Finite energy weak solution to NS system}] \label{PD1}
The functions $[\vr, \vu]$ are termed \emph{finite energy weak solution} to the Navier--Stokes system
\eqref{P3} in $(0,T) \times \Td$ with the initial conditions
\[
\vr(0, \cdot) = \vr_0, \ \vr \vu(0, \cdot) = \vm_0
\]
if the following holds:
\begin{itemize}
\item {\bf Integrability}
\[
\begin{split}
\vr &\in C_{\rm weak}([0,T]; L^\gamma(\Td)),\ \vr \geq 0,\\
\vu &\in L^2(0,T; W^{1,2}(\Td; \mathds{R}^d)),\
\vm = \vr \vu \in C_{\rm weak}([0,T]; L^{\frac{2 \gamma}{\gamma + 1}}
(\Td; \mathds{R}^d)).
\end{split}
\]
\item {\bf Equation of continuity}
\[
\int_0^T \intO{ \Big[ \vr \partial_t \varphi + \vr \vu \cdot \Grad \varphi \Big] } \dt = - \intO{\vr_0 \varphi}
\]
for any $\varphi \in C^1_c([0,T) \times \Td)$.
\item {\bf Momentum equation}
\[
\begin{split}
\int_0^T \intO{ \Big[ \vr \vu \cdot \partial_t \bfphi + \vr \vu \otimes \vu : \Grad \bfphi +
p(\vr) \Div \bfphi \Big] } \dt &= \int_0^T \intO{ \mathbb{S}(\Grad \vu) : \Grad \bfphi } \dt\\
&- \intO{\vm_0 \cdot \bfphi(0, \cdot)}
\end{split}
\]
for any $\bfphi \in C^1_c([0,T) \times \Td; \mathds{R}^d)$.
\item {\bf Energy inequality}
\[
\begin{split}
\intO{ \Big[ \frac 1 2 \frac{|\vm|^2}{\vr}  + \frac{a}{\gamma - 1} \vr^\gamma \Big] (\tau, \cdot) }
&+ \int_0^\tau \intO{ \mathbb{S}(\Grad \vu) : \Grad \vu } \dt \\
&\leq \intO{ \Big[ \frac{1}{2} \frac{|\vm_0|^2}{\vr_0} + \frac{a}{\gamma - 1} \vr_0^\gamma \Big] }
\end{split}
\]
for any $0 \leq \tau \leq T$.

\end{itemize}
\end{Definition}

The following result was {proven} in \cite[Theorem 7.1]{EF70}.

\begin{Theorem}[{\bf Global existence for (NS) system}] \label{PT1}
Let $\gamma > \frac{d}{2}$ and let
\[
\vr_0 \in L^\gamma(\Td), \ \vr_0 \geq 0,\ \intO{ \frac{|\vm_0|^2}{\vr_0} } < \infty.
\]

Then the Navier--Stokes system admits a finite energy weak solution $[\vr, \vu]$ in $(0,T) \times \Td$ in the sense of Definition
\ref{PD1} for any $T \leq \infty$.

\end{Theorem}

\subsection{Viscosity solutions to the Euler system}
\label{VSES}

We are ready to introduce the concept of viscosity solution to the Euler system. We consider finite energy measurable initial data $[\vr_0, \vm_0]$,
\[
\vr_0 \geq 0,\ \intO{ \left[ \frac{1}{2} \frac{|\vm_0|^2}{\vr_0} + \frac{a}{\gamma - 1} \vr_0^\gamma \right] } < \infty.
\]

\begin{Definition} [{\bf Regular data approximation}] \label{PD3a}

Let $[\vr_0, \vm_0]$ be finite energy data. We say that a sequence $[\vr_{0,n}, \vm_{0,n}]$ is {a} \emph{regular approximation}
of the data $[\vr_0, \vm_0]$ if
\[
\begin{split}
\vr_{0,n} \in W^{k,2}(\Td),\ \vm_{0,n} &\in W^{k,2}(\Td; \mathds{R}^d)\ \mbox{for some}\ k \geq 3,\ \inf_{\Td} \vr_{0,n} > 0,
\\
\vr_{0,n} &\to \vr_0 \ \mbox{weakly in}\ L^1(\Td),\
\vm_{0,n} \to \vm_0 \ \mbox{weakly in}\ L^1(\Td; \mathds{R}^d),\\
\intO{ \Big[ \frac{1}{2} \frac{|\vm_{0,n}|^2}{\vr_{0,n}} +
\frac{a}{\gamma - 1} \vr_{0,n}^\gamma \Big] } &\to \intO{ \Big[ \frac{1}{2} \frac{|\vm_{0}|^2}{\vr_{0}} +
\frac{a}{\gamma - 1} \vr_{0}^\gamma \Big] }.
\end{split}
\]

\end{Definition}

As is well--known, cf. e.g. Matsumura and Nishida \cite{MANI} or Valli and Zajaczkowski \cite{VAZA}, the Navier--Stokes
system \eqref{P3} admits a local in time classical solution for any initial data $[\vr_{0,n}, \vm_{0,n}]$ in the regularity
class specified in Definition \ref{PD3a}.

\begin{Definition}[{\bf Viscosity solution of Euler system}] \label{PD3}

A parametrized family of probability measures
\[
\mathcal{V}_{t,x}: (t,x) \in (0,T) \times \Td \mapsto \mathfrak{P}({\mathds{R}}^{d + 1}), \ \mathcal{V} \ \mbox{weakly measurable,}
\]
{where $\mathfrak{P}$ denotes the set of probability measures,}
is called \emph{viscosity solution} of the Euler system \eqref{P1} with the initial data $[\vr_0, \vm_0]$ if there exists a
regular approximation of the initial data $[\vr_{0,n}, \vm_{0,n}]$, sequences of viscosity coefficients
\[
\mu_n \searrow 0,\ \lambda_n \to 0,
\]
and a sequence of finite energy weak solutions $[\vr_n, \vm_n \equiv \vr_n \vu_n]_{n=1}^\infty$ of the Navier--Stokes
system \eqref{P3} starting from the initial data $[\vr_{0,n}, \vm_{0,n}]$ such that
\[
b(\vr_n, \vm_n) \to \Ov{b(\vr, \vm)} \ \mbox{weakly-(*) in}\ L^\infty((0,T) \times \Td)
\ \mbox{for any}\ b \in C_c ({\mathds{R}}^{d+1}),
\]
where
\[
\Ov{b(\vr, \vm)} (t,x) = \left< \mathcal{V}_{t,x} ; b (\tvr, \tvm) \right>
\ \mbox{for a.a.}\ (t,x) \in (0,T) \times \Td.
\]
\end{Definition}

Thus a viscosity solution of the Euler system is simply a Young measure generated by a sequence of solutions
of the Navier--Stokes system in the vanishing viscosity limit.
The fact that the Navier--Stokes system admits a classical local in time solution for any
regular approximation
$[\vr_{0,n}, \vm_{0,n}]$ of the initial data does not imply,
of course, the existence of a ``smooth'' viscosity solution not even on a short time interval, as the life span of a potential generating sequence may shrink
to zero for $n \to \infty$. However, in view of the existence result for the Navier--Stokes system stated in Theorem~\ref{PT1},
a viscosity solution of the Euler system always exists at least if $\gamma > \frac{d}{2}$.

A viscosity solution may depend on the choice of {\bf (i)} regular approximation of the initial data
$[\vr_{0,n}, \vm_{0,n} ]_{n=1}^\infty$, {\bf (ii)} the sequence of shear viscosity coefficients $\mu_n \searrow 0$,
{\bf (iii)}
the sequence of
bulk viscosity coefficients $\lambda_n \to 0$. To minimize the number of independent parameters we may omit the
approximation of the initial data and consider directly $[\vr_{0,n}, \vm_{0,n}] = [\vr_0, \vm_0]$ as the Navier--Stokes
admits global in time weak solutions for any finite energy data. Moreover, one can fix $\lambda$ as a function of $\mu$, in particular
we may set $\lambda_n = 0$, see Section \ref{e}.

As the vanishing viscosity limit is a consistent approximation of the
Euler system in the sense of \cite{FeiLMMiz} (see \cite{MarEd}), any viscosity solution specified in Definition \ref{PD3} is a dissipative measure valued (DMV)
solution of the Euler system in the sense of  \cite{FeiLMMiz}.

\begin{Definition}[{\bf DMV solution of Euler system}]  \label{DMV}

A parametrized family of probability measures
\[
\mathcal{V}_{t,x}: (t,x) \in (0,T) \times \Td \mapsto \mathfrak{P}({\mathds{R}}^{d + 1}), \ \mathcal{V} \ \mbox{weakly measurable,}
\]
is called \emph{dissipative measure valued (DMV) solution} of the Euler system \eqref{P1} with
the initial conditions $[\vr_0, \vm_0]$ if the following hold:
\begin{itemize}
\item {\bf Integrability}
\[
\left< \mathcal{V}; \tvr \right> \in C_{\rm weak}([0,T]; L^\gamma(\Td)),
\]
\[
\left< \mathcal{V}; \tvm \right> \in C_{\rm weak}([0,T]; L^{\frac{2\gamma}{\gamma + 1}}(\Td; \mathds{R}^d)).
\]
\item {\bf Equation of continuity}
\[
\int_0^T \intO{ \Big[ \left< \mathcal{V} ; \tvr \right> \partial_t \varphi + \left< \mathcal{V} ; \tvm \right> \cdot \Grad \varphi \Big] } \dt = - \intO{ \vr_0 \varphi (0, \cdot) }
\]
for any $\varphi \in C_c^{1}([0,T) \times \Td)$.

\item {\bf Momentum equation}
\[
\begin{split}
&\int_0^T \intO{ \left[ \left< \mathcal{V}; \tvm \right> \cdot \partial_t \bfvarphi + \left< \mathcal{V}; 1_{\tvr > 0}\frac{\tvm \otimes \tvm}{\tvr} \right>: \Grad \bfvarphi
+ \left< \mathcal{V} ; p(\tvr) \right> \Div \bfvarphi \right] } \dt \\ & = - \intO{ \vm_0 \cdot \bfphi }
- \int_0^T \int_{\Td} \Grad \bfvarphi : \D \mathfrak{R}(t) \dt
 \end{split}
\]
for any $\bfvarphi \in C^1_c([0,T) \times \Td; \mathds{R}^d)$, with the Reynolds defect $\mathfrak{R} \in L^\infty(0,T; \mathcal{M}^+(\Td; \mathds{R}^{d\times d}_{\rm {sym}}))$,
where ${\mathcal{M}^+(\Td; \mathds{R}^{d\times d}_{\rm {sym}})}$ denotes the space of {positive semi-definite, symmetric matrix-valued} measures.

\item {\bf Energy inequality}
\[
\intO{ \left< \mathcal{V}_{\tau,x} ; \frac{1}{2} \frac{|\tvm|^2}{\tvr} +
P(\tvr) \right> } + \int_{\Td} \D \mathfrak{E}(\tau)
\leq \intO{ \left[ \frac{1}{2} \frac{|\vm_0|^2}{\vr_0} + P(\vr_0) \right] }
\]
holds for a.a. $0 \leq \tau < T$, with the energy  defect
\[
\mathfrak{E} \in L^\infty(0,T; \mathcal{M}^+ (\Td)).
\]
\item {\bf Defect compatibility}
\[
\underline{d} \ \mathfrak{E} \leq {\rm tr}[ \mathfrak{R}] \leq \Ov{d}\
\mathfrak{E}\ \mbox{for some constants}\ 0 < \underline{d} \leq \Ov{d}.
\]

\end{itemize}

\end{Definition}

DMV solutions represent the {largest} class of objects
that can be identified as limits of consistent approximations of the Euler system. It is straightforward to show that
any convex combination of DMV solutions is a DMV solution. In addition, the class of DMV solutions
enjoys several remarkable properties, see \cite{FeiLMMiz}, among which:
\begin{itemize}

\item

{\bf Weak--strong uniqueness} Suppose that the Euler system
admits a Lipschitz solution $[\vr, \vm]$. Then
\[
\mathcal{V}_{t,x} = \delta_{\vr(t,x), \vm(t,x)} \ \mbox{for a.a.}\ (t,x) \in (0,T) \times \Td,
\ \mathfrak{R} = \mathfrak{E} = 0
\]
for any DMV solution emanating from the same initial data.

\item {\bf Compatibility} If
\[
(t,x) \mapsto \left< \mathcal{V}_{t,x} ;  (\tvr, \tvm) \right> \in C^1([0,T] \times \Td; \mathds{R}^{d + 1}),
\]
then
\[
\mathcal{V}_{t,x} = \delta_{\vr(t,x), \vm(t,x)} \ \mbox{for a.a.}\ (t,x) \in (0,T) \times \Td
\]
where $[\vr, \vm]$ is a classical solution of the Euler system.

\end{itemize}

Being aware of the fact that the viscosity solution $\mathcal{V}$ may inherit certain features that depend on its generating sequence, we introduce the set of \emph{observables}.

\begin{Definition}[{\bf Observables}] \label{PD4}
Let the initial data $[\vr_0, \vm_0]$ be given. We say that a function $B \in C(\mathds{R}^{d + 1})$
is
\emph{observable} for the Euler system if
\[
\left< \mathcal{V}^1_{t,x}; B (\widetilde{\vr}, \widetilde{\vm}) \right> = \left< \mathcal{V}^2_{t,x}; B (\widetilde{\vr}, \widetilde{\vm}) \right>  \ \mbox{a.a. in}\ (0,T) \times \Td
\]
whenever $\mathcal{V}^1$, $\mathcal{V}^2$ are viscosity solutions with the initial data $[\vr_0, \vm_0]$.

\end{Definition}

\begin{Remark} In general, the set of observables may also depend on the time $T$. In text below, we omit to mention explicitly this fact as $T$ is kept fixed.
	\end{Remark}

Clearly, the set of all observables
\[
\mathcal{O}[\vr_0, \vm_0] = \left\{ B \in C(\mathds{R}^d) \ \Big| \ B \ \mbox{is observable for the Euler system with the data}
\ [\vr_0, \vm_0]
\right\}
\]
is a closed linear subspace of $C(\mathds{R}^d)$ containing at least all constant functions.
Moreover,
\[
\mathcal{V}^1_{t,x} = \mathcal{V}^2_{t,x} \ \mbox{a.a. in}\ (0,T) \times \Td \text{ for all viscosity solutions }\mathcal{V}^1,\ \mathcal{V}^2
\ \Leftrightarrow \ C_c(\mathds{R}^{d + 1}) \subset \mathcal{O}[\vr_0, \vm_0].
\]
The observables are quantities that are independent of the generating sequence and as such represent an \emph{intrinsic property}
shared by all viscosity solutions starting from the same initial data $[\vr_0, \vm_0]$. In particular, the viscosity solution is unique
if and only if all functions in $C_c(\mathds{R}^{d + 1})$ are observables.

It is not excluded that the set of observables may depend on the initial data $[\vr_0, \vm_0]$. Indeed, if the Euler system admits
a $C^1$ solution, then $\mathcal{O}[\vr_0, \vm_0] = C(\mathds{R}^{d + 1})$, while it may not be the case if the barycenter of $\mathcal{V}$ is not smooth. It would be desirable though if at least the coordinates of the barycenter
\[
\vr = \left< \mathcal{V}; \tvr \right>,\ \vm = \left< \mathcal{V}; \tvm \right>
\]
were observables. In fact,
\[
\vm = \left< \mathcal{V}; \tvm \right> \ \mbox{observable} \ \Rightarrow \ \vr = \left< \mathcal{V}; \tvr \right>
\ \mbox{observable},
\]
which can be deduced from the fact that both quantities satisfy the continuity equation
\[
\partial_t \left< \mathcal{V}; \tvr \right> + \Div \left< \mathcal{V}; \tvm \right> = 0
\]
in the sense of distributions, with given initial data $\vr_0$.

Our main goal in this paper is to show that
a viscosity solution and, in particular, the
observables can be effectively computed by means of a suitable numerical method.

\subsection{Uniqueness of viscosity solutions}

Given a sequence of solutions $[\vr_n, \vm_n ]_{n=1}^\infty$ to the Navier--Stokes system
with the viscosity coefficients $\mu_n \searrow 0$, $\lambda_n \to 0$, emanating from a regularized approximation of the initial data $[\vr_{0,n}, \vm_{0,n}]_{n=1}^\infty$, we may need a subsequence to generate a Young measure that represents a viscosity solution of the
Euler system in the sense of Definition \ref{PD3}. If this is the case,
the viscosity solution is obviously not unique.
Although the numerical experiments performed and discussed in a series of
papers by Elling \cite{ELLI2, ELLI3, ELLI1} suggest this may be indeed possible, these results are based on approximating
\emph{unstable} initial data. In other words, the limit is sensitive to the choice of the approximate sequence $[\vr_{0,n}, \vm_{0,n} ]_{n=1}^\infty$. In view of these arguments, a more robust approximation of a viscosity solution will be introduced
based on the concept of S--convergence discussed in the next section.

\section{S--convergence and approximate solutions}
\label{s}

The concept of S--convergence was introduced in \cite{Fei2020A}.
The definition used in this paper is slightly different imposing more restrictions on the generating sequence.

\subsection{S--convergence}

An infinite matrix
\[
\{ s_{n,N} \}_{n=1, N=1}^\infty
\]
is called \emph{regular summation method} if the
following properties are satisfied:

\begin{gather} \label{P6}
\begin{aligned}
&\bullet \ \ 0 \leq s_{n,N} \leq \Ov{s}  \ \mbox{for any}\ n,N,  \mbox{ where } \Ov{s} = \mbox{const.} \mbox{ independent of } n \mbox{ and } N,\\[2mm]
&\bullet \ \ s_{n,N} = 0 \ \mbox{whenever}\ n > N,\\
&\bullet \ \ \sum_{n = 1}^N s_{n,N} = N \ \mbox{for any}\ N = 1,2,\dots
\end{aligned}
\end{gather}

\begin{Remark} \label{PRr1}

Note that the first condition is  usually replaced by a weaker stipulation
\[
\sup_{n \leq N} \frac{s_{n,N}}{N} \to 0 \ \mbox{as}\ N \to \infty
\]
in the literature.

\end{Remark}

\begin{Definition}[{\bf S--convergence}] \label{PD2}
\

{\bf (i)}
Let
\[
\vc{U}_n : Q \subset \mathds{R}^D \to \mathds{R}^m ,  D \mbox{ is  an integer,}\ n=1,2,\dots
\]
be a sequence of measurable (vector valued) functions. We say that $\{ \vc{U}_n \}_{n=1}^\infty$ is
\emph{{\rm S}--convergent} to a parameterized family of probability measures $\{ \mathcal{V}_y \}_{y \in Q}$,
$\mathcal{V}_y \in \mathfrak{P}(\mathds{R}^m)$ for a.a. $y \in Q$,
\[
\vc{U}_n \toS \mathcal{V} \ \mbox{with respect to a summation method}\ \{ s_{n,N} \}
\]
if
\begin{equation} \label{P7}
\intQ{
w \left[ \frac{1}{N} \sum_{n=1}^N s_{n,N} \delta_{\vc{U}_n(y)} ; \mathcal{V}_y \right] } \to 0\
\   \mbox{as}\ N \to \infty
\end{equation}
where $w$ denotes the $1-$Wasserstein distance in $\mathfrak{P}(\mathds{R}^m)$, cf. Villani~\cite{Villa}.

{\bf (ii)} We say that $\vc{U}_n$ is \emph{completely {\rm S}--convergent} to $\{ \mathcal{V}_y \}_{y \in Q}$,
\[
\vc{U}_n \toSC \mathcal{V},
\]
if
\eqref{P7} holds for any regular summation method and any subsequence $\{ \vc{U}_{n_j} \}_{j=1}^\infty$ of the original sequence.
\end{Definition}

\begin{Remark} \label{PR1}

Note that convergence in \eqref{P7}  requires a uniform $L^1$-bound on the first moments of the family
\[
\frac{1}{N} \sum_{n=1}^N s_{n,N} \delta_{\vc{U}_n(y)}.
\]
In particular, the barycenter
\[
\left< \mathcal{V} ; \widetilde{\vc{U}} \right> =
\lim_{N \to \infty} \frac{1}{N} \sum_{n=1}^N s_{n,N} \vc{U}_n
\]
is well--defined for a.a. $y \in Q$.

\end{Remark}

We proceed by stating a sufficient condition for a sequence of measurable vector valued functions to be S--convergent.

\begin{Lemma} \label{sL1}
Let
\[
\vc{U}_n : Q \subset \mathds{R}^D \to \mathds{R}^m , \ n=1,2,\dots
\]
be a sequence of measurable (vector valued) functions defined on a bounded domain $Q \subset \mathds{R}^D$ that enjoys the following properties:
\begin{itemize}
\item $\{ |\vc{U}_n | \}_{n=1}^\infty$ is equi--integrable in $Q$;
\item For any $b \in C_c(\mathds{R}^m)$ and a regular {summation} method  ${\{ s_{n,N} \}_{n=1, N=1}^\infty}$
\begin{equation} \label{s1}
\frac{1}{N} \sum_{n=1}^N s_{n,N} b(\vc{U}_n) \to \Ov{b(\vc{U})} \ \mbox{in measure in}\ Q\  \mbox{as}\ N \to \infty.
\end{equation}

Then $\{ \vc{U}_n \}_{n=1}^\infty$ is S--convergent with respect to ${\{ s_{n,N} \}_{n=1, N=1}^\infty}$, and
\[
\vc{U}_n \toS \mathcal{V},\ \mbox{where}\ \left< \mathcal{V}_y; b (\widetilde{\vc{U}}) \right> =
\Ov{b(\vc{U})}(y) \ \mbox{for a.a.}\ y \in Q.
\]

\end{itemize}

\end{Lemma}

\begin{proof}

As $\{ |\vc{U}_n | \}_{n=1}^\infty$ is equi--integrable we can by consider a sequence $b_k\in C_c(\mathds{R}^m)$ with $b_k(\vc{U}) =b(\vc{U})$ for $|\vc{U}|<k$ and $|b_k|<|b|$ to see that validity of \eqref{s1} can be extended to any function
\[
b \in C(\mathds{R}^m),\ |b(\vc{U})| \aleq 1 + |\vc{U}|.
\]

Next, we introduce a parametrized measure $\mathcal{V}$ as
\[
\left< \mathcal{V}_y , b (\widetilde{\vc{U}}) \right> = \Ov{b(\vc{U})} (y) \ \mbox{for a.a.}\ y \in Q.
\]
The convergence in \eqref{s1} is equivalent to
\begin{equation} \label{s2}
d_{W^*} \left[ \frac{1}{N} \sum_{n=1}^N s_{n,N} \delta_{\vc{U}_n(y)}; \mathcal{V}_y \right] \to 0
\ \mbox{in measure for}\ y \in Q,
\end{equation}
where $d_{W^*}$ is the dual metric,
\[
\begin{split}
d_{W^*} [{\mathcal{V}_1, \mathcal{V}_2}] &= \sum_{k=1}^\infty \frac{1}{2^k} \left| \left<{\mathcal{V}_1}; b_k(\tilde{U})\right> -
\left<{\mathcal{V}_2}; b_{k}(\tilde{U}) \right> \right|, \\ \{ b_k \}_{k=1}^\infty &\subset C_c(\mathds{R}^m)
\ \mbox{a dense set in the unit ball in}\ C_0(\mathds{R}^m).
\end{split}
\]
Moreover, we also have convergence of the first moments,
\begin{equation} \label{s3}
\frac{1}{N} \sum_{n=1}^N s_{n,N}|\vc{U}_n (y)| \to \left< \mathcal{V}_y; \widetilde{ |\vc{U}|} \right>
\ \mbox{in measure for}\ y \in Q.
\end{equation}

As the sequence $\{ \vc{U}_n \}_{n=1}^\infty$ is equi--integrable, the function
\[
y \mapsto w \left[ \frac{1}{N} \sum_{n=1}^N s_{n,N} \delta_{\vc{U}_n(y)}; \mathcal{V}_y \right]
\]
is equi--integrable as well. Consequently, to deduce the desired conclusion \eqref{P7}, it is enough to observe that
\[
y \mapsto w \left[ \frac{1}{N} \sum_{n=1}^N s_{n,N} \delta_{\vc{U}_n(y)}; \mathcal{V}_y \right] \to 0
\ \mbox{in measure for}\ y \in Q.
\]
Assuming the contrary, we get $\ep, \delta > 0$ and a subsequence $N_k \to \infty$ such that
\begin{equation} \label{s4}
\left| y \in Q \ \Big| \ w \left[ \frac{1}{N_k} \sum_{n=1}^{N_k} s_{n,N_k} \delta_{\vc{U}_n(y)}; \mathcal{V}_y \right] > \ep \right|
> \delta \  \  \mbox{for }\ k \to \infty.
\end{equation}
On the other hand,  in accordance with \eqref{s2}, \eqref{s3} and
Villani \cite[Theorem 6.9]{Villa} , $\{ N_k \}_{k=1}^\infty$ contains a subsequence (not relabeled) such that
\[
w \left[ \frac{1}{N_k} \sum_{n=1}^{N_k} s_{n,N_k} \delta_{\vc{U}_n(y)}; \mathcal{V}_y \right] \to 0
\ \mbox{as}\ k \to \infty \ \mbox{for a.a.}\ y \in Q
\]
in contrast with \eqref{s4}.

\end{proof}

Next we focus on sufficient conditions that would guarantee \eqref{s1}. As the composition $b(\vc{U}_n)$ is bounded,
it is enough to consider uniformly bounded sequences of measurable functions. We say that a set
$\mathcal{S} \subset \mathds{N} \times \mathds{N}$ is \emph{statistically significant} if
\[
\frac{ \# \{ (n,m) \in \mathcal{S} \ \Big| \ n,m \leq N \} }{N^2} \to 1 \
\ \mbox{as}\ N \to \infty.
\]

\begin{Lemma} \label{sL2}
Let $Q \subset \mathds{R}^D$ be a bounded domain and let
\[
{V}_n : Q \to \mathds{R},\ \|{V}_n \|_{L^\infty(Q)} \aleq 1 \ \mbox{uniformly for}\ n =1,2,\dots
\]
be a bounded sequence of measurable functions. Suppose there exists $V \in L^\infty(Q)$ such that the set
\[
\mathcal{S}(\ep)  = \left\{ (n,m) \ \Big|\  s_{n, M} s_{m, M} \left|\intQ{ (V_n - V)(V_m - V) } \right| < \ep \; \mbox{ for all } M\in\mathds{N} \right\}
\]
is statistically significant for any $\ep > 0$.

Then
\begin{equation} \label{s5}
\frac{1}{N} \sum_{n=1}^N s_{n,N} {V}_n \to V \ \mbox{in}\ L^q(Q) \ \mbox{as}\ N \to \infty,\ 1 \leq q < \infty.
\end{equation}

\end{Lemma}

\begin{proof}

As $\{ V_n \}_{n=1}^\infty$ is uniformly bounded, it is enough to show
\[
\frac{1}{N} \sum_{n=1}^N s_{n,N} {V}_n \to V \ \mbox{in}\ L^2(Q)
\]
in other words
\[
\intQ{ \left| \frac{1}{N} \sum_{n=1}^N s_{n,N} ( V_n - V) \right|^2 } \to 0.
\]
Denote
\[
K_{N}(\ep) = \left\{ n,m \leq N \Big| \ s_{n,N} s_{m,N} \left|\intQ{ (V_n - {V})(V_m - {V}) } \right|  \geq \ep   \right\}.
\]
We have
\[
\begin{split}
&\intQ{ \left| \frac{1}{N} \sum_{n=1}^N s_{n,N} ( V_n - V) \right|^2 } \leq
\frac{1}{N^2}  \sum_{n,m=1}^N s_{n,N} s_{m,N} \left| \intQ{ (V_n - V) (V_m - V) } \right| \\
\\ &\leq \frac{1}{N^2}  \sum_{n,m=1, (n,m) \in K_N(\ep) }^N \Ov{s}^2 \left| \intQ{ ({V}_n - {V}) ({V}_m - {V}) } \right|
+ \ep\\ &\leq 4 \Ov{s}^2 {|Q| \max\left\{\| V \|_{L^\infty(Q)},\ \sup_{n \geq 1} \| V_n \|_{L^\infty(Q)}\right\}^2} \frac{ \# K_N(\ep) }{N^2} + \ep,
\end{split}
\]
where, in accordance with our hypothesis,
\[
\frac{ \# K_N(\ep) }{N^2} \to 0 \ \mbox{as}\ N \to \infty.
\]
As $\ep > 0$ was arbitrary, the desired result follows.
\end{proof}

\begin{Corollary} \label{sC1}
Let $Q \subset \mathds{R}^D$ be a bounded domain and let
\[
{V}_n : Q \to  \mathds{R},\ \|{V}_n \|_{L^\infty(Q)} \aleq 1 \ \mbox{uniformly for}\ n =1,2,\dots
\]
be a bounded sequence of measurable functions. Suppose there exists $V \in L^\infty(Q)$ such that the set
\[
\mathcal{S}_\ep = \left\{ (n,m) \ \Big|\  \left|\intQ{ (V_n - V)(V_m - V) } \right| < \ep \right\}
\]
is statistically significant for any $\ep > 0$.

Then
\[
\frac{1}{N} \sum_{n=1}^N s_{n,N} {V}_n \to V \ \mbox{in}\ L^q(Q) \ \mbox{as}\ N \to \infty,\ 1 \leq q < \infty.
\]
for any regular summation method ${\{ s_{n,N} \}_{n=1, N=1}^\infty}$.

\end{Corollary}

\subsection{Asymptotic stationarity}

The disadvantage of Lemma \ref{sL2} is that the sufficient condition for convergence depends on the {\it a priori} unknown limit
$V$. To remove the problem we revisit the concept of asymptotic stationarity introduced in \cite{Fei2020A}.

\begin{Proposition} \label{sP2}
Let $Q \subset \mathds{R}^D$ be a bounded domain and let $\{ V_n \}_{n=1}^\infty$
be a bounded sequence of measurable functions,
\[
V_n \to V \ \mbox{weakly-(*) in}\ L^\infty (Q).
\]
In addition, suppose that for any $\ep > 0$, $K \geq 1$, there exists $k = k(\ep, K) \geq K$ such that
\begin{equation} \label{s6}
\begin{split}
&\left\{ (n, m) \Big|\
\left| \intQ{ \Big[ V_{n} V_m - V_{k+|n-m|} V_k \Big] } \right| < \ep \right\}
\\ &\mbox{is statistically significant.}
\end{split}
\end{equation}

Then
\[
\frac{1}{N} \sum_{n=1}^N s_{n,N} {V}_n \to V \ \mbox{in}\ L^q(Q) \ \mbox{as}\ N \to \infty,\ 1 \leq q < \infty
\]
for any regular summation method $\{ s_{n,N} \}_{n,N = 1}^\infty$.
\end{Proposition}

\begin{proof}

First observe that it is enough to assume ${V} = 0$. Indeed replacing ${V_n} \approx (V_n - V)$ we compute
\begin{align}
&\intQ{ \Big[ (V_{n} - V)(V_m - V) - (V_{k+|n-m|} - V)(V_k - V) \Big] } =
\intQ{ \Big[ V_{n} V_m - V_{k+|n-m|} V_k \Big] }\nonumber \\
&+ \intQ{ V \Big( V_{k} + V_{k+|n-m|} - V_m - V_n) \Big) }.
\label{Blb1}
\end{align}
Next, by the weak-(*) convergence of $V_n$, we have the following property. For any $\ep > 0$  there exists $K = K(\ep) \geq 1$,   such that
\begin{equation} \label{Blb2}
\left| \intQ{ V {\Big( V_{k} + V_{k+|n-m|} - V_m - V_n \Big)} } \right| < \frac{\ep}{2}
\end{equation}
whenever $m,n,k \geq K$.  Now, hypothesis \eqref{s6}, together with \eqref{Blb1}, \eqref{Blb2}, imply that the set
\[
\left\{ (n, m) \Big|\ n,m \geq K,
\left| \intQ{ \Big[ (V_{n} - V) (V_m - V) - (V_{k+|n-m|} - V)( V_k - V) \Big] } \right| < \frac{\ep}{2} \right\} 
\]
is still statistically significant for $k = k\left( \frac{\ep}{2}, K \right)$; whence we may suppose $V \equiv 0$.

Now, fix $\ep > 0$ with the associated $k( \frac{ \ep}{2})$. It follows from hypothesis \eqref{s6} that
\begin{equation} \label{s7}
\left\{ (n, m) \Big|\ |n - m| \geq L,
\left| \intQ{ \Big[ V_{n} V_m - V_{k+|n-m|} V_k \Big] } \right| <  \frac{\ep}{2} \right\}
\end{equation}
is statistically significant for any $L \geq 1$. As $V_n \to 0$ weakly-(*) in $L^\infty(Q)$, we can fix $L$ so that
\[
\left| \intQ{ V_k V_{k + |n-m|} } \right| <  \frac{\ep}{2} \ \mbox{whenever}\ |n - m| \geq L.
\]
 Keep in mind that $k=k(\frac{\eps}{2})$ is fixed at this stage.

Finally, it follows from \eqref{s7} that
\[
\left\{ (n, m) \Big|\
\left| \intQ{  V_{n} V_m } \right| < \ep \right\}
\]
is statistically significant for any $\ep > 0$ which, in view of Corollary \ref{sC1}, yields the desired conclusion.

\end{proof}

\subsection{S--convergence vs. strong convergence}

If
\[
\vc{U}_n \to \vc{U} \ \mbox{in}\ L^1(Q; \mathds{R}^D),
\]
then it is easy to see, cf. \cite[Corollary 3.4]{Fei2020A}, that
\[
\vc{U}_n \toS \delta_{\vc{U}}\ \mbox{for any regular summation method.}
\]
Moreover, as the (strong) limit is unique, we get
\[
\vc{U}_n \toSC \delta_{\vc{U}}.
\]

A converse statement reads:

\begin{Lemma} \label{sL8}

Let
\[
\vc{U}_n \to \vc{U} \ \mbox{weakly in}\ L^1(Q; \mathds{R}^D),
\]
and
\[
\frac{1}{N} \sum_{n=1}^N s_{n,N} b(\vc{U}_n) \to b(\vc{U}) \ \mbox{weakly-(*) in}\ L^\infty(Q)
\]
for any $b \in C_c(\mathds{R}^D)$ and some regular summation method $\{ s_{n,N} \}$.

Then there is a subsequence $\{ \vc{U}_{n_j} \}_{j=1}^\infty$ such that
\[
\vc{U}_{n_j} \to \vc{U} \ \mbox{in}\ L^1(Q; \mathds{R}^D).
\]

\end{Lemma}

\begin{proof}

Given $b \in C_c(\mathds{R}^D)$, we have
\[
\frac{1}{N} \sum_{n=1}^N s_{n,N} b(\vc{U}_n) \to b(\vc{U}),\
\frac{1}{N} \sum_{n=1}^N s_{n,N} {|b(\vc{U}_n)|^2} \to |b(\vc{U})|^2 \ \mbox{weakly-(*) in}\ L^\infty(Q),
\]
in particular
\[
\frac{1}{N} \sum_{n=1}^N s_{n,N} \intQ{ |b(\vc{U}_n) - b(\vc{U})|^2 } \to 0 \ \mbox{as}\ N \to \infty.
\]
Consequently, there is a subsequence $n_k \to \infty$ such that
\[
b(\vc{U}_{n_k}) \to b(\vc{U}) \ \mbox{in}\ L^2(Q).
\]

Repeating the same argument {for} a family $\{ b_m \}_{m=1}^\infty$ of functions $b$ dense in ${C_c(\mathds{R}^D)}$ we obtain a subsequence
$n_j$ such that
\[
b(\vc{U}_{n_j}) \to b(\vc{U}) \ \mbox{in}\ L^2(Q) \ \mbox{for any}\ b \in C_c(\mathds{R}^D).
\]
Due to the Dunford-Pettis theorem the sequence $\{\vc{U}_{n}\}_{n=1}^\infty$ is equi--integrable, which yields the desired conclusion.

\end{proof}

\subsection{S--convergent subsequence principle}

The following result was essentially proven by Balder \cite{Bald} in the case of
the Ces\` aro summation method $s_{n,N} = 1$ for $n \leq N$. In contrast with \cite{Bald},  our approach is based on the Banach-Saks theorem, while \cite{Bald}  uses the Koml\'os theorem not available for arbitrary summation method.

\begin{Proposition}[{\bf Subsequence principle}] \label{sP3}

Let
\[
\vc{U}_n: Q \subset \mathds{R}^D \mapsto \mathds{R}^m ,\ Q \subset \mathds{R}^D \ \mbox{a bounded domain}
\]
be  an equi--integrable sequence of $L^1$ functions.

Then there are a subsequence $\{ \vc{U}_{n_j} \}_{j=1}^\infty$ and
a parametrized family of probability measures $\{ \mathcal{V}_y \}_{y \in Q}\ \subset \mathfrak{P}(\mathds{R}^m)$ such that
\[
\vc{U}_{n_j} \toSC \mathcal{V}.
\]

\end{Proposition}

\begin{proof}

Repeating the nowadays standard procedure leading to the construction of a Young measure we may extract a suitable subsequence
(not relabeled) such that
\[
b(\vc{U}_n) \to \Ov{ b(\vc{U}) } \ \mbox{weakly-(*) in}\ L^\infty(Q) \ \mbox{as}\ n \to \infty
\]
for any $b \in C_c(\mathds{R}^m)$. We set
\[
\left< \mathcal{V}_{y}, b(\tilde{U}) \right> = \Ov{b(\vc{U})} (y) \ \mbox{for a.a.}\ y \in Q.
\]

In accordance with Lemma \ref{sL1}, we have to show \eqref{s1} for any regular method of summation. As $b$ is bounded, convergence in \eqref{s1} is equivalent to
\begin{equation} \label{e1}
\frac{1}{N} \sum_{n=1}^N s_{n,N} b(\vc{U}_n) \to \Ov{b(\vc{U})} \ \mbox{in}\
L^2(Q) \ \mbox{as}\ N \to \infty
\end{equation}
for any $b \in C_c(\mathds{R}^m)$. The Hilbert space $L^2(Q)$ enjoys the Banach--Saks property. Specifically, any {bounded sequence}
$\{v_n \}_{n=1}^\infty$ {in $L^2(Q)$} contains a subsequence $\{ v_{n_j } \}_{j=1}^\infty$ such that
\[
\frac{1}{N} \sum_{j=1}^N v_{n_j} \to v \ \mbox{in}\ L^2(Q).
\]
Thus, given $b \in C_c(\mathds{R}^m)$, we may apply the alternative shown by Rosenthal \cite{Rosen} to conclude that
there is a subsequence of indexes $\{ n_j \}_{j = 1}^\infty$ such that for any of its subsequences
$\{ n_l \}_{l=1}^\infty \subset \{ n_j \}_{j = 1}^\infty$ there holds
\[
\frac{1}{N} \sum_{l=1}^N s_{{l},N} b(\vc{U}_{n_l} ) \to \Ov{b(\vc{U})} \ \mbox{in}\ L^2(Q)
\]
for \emph{any} regular summation method $\{ s_{n,N} \}$.

Repeating this argument successively for a family of functions $\{ b_m \}_{m=1}^\infty$ dense in $C_0(\mathds{R}^m)$
we obtain \eqref{e1} for a suitable subsequence.

\end{proof}

\section{Applications to the Euler system}
\label{e}

We apply the abstract results obtained in Section \ref{s} to the Euler system.
We start by introducing the concept of generating sequence.

\begin{Definition}[{\bf Generating sequence}] \label{eD1}

Let $\mathcal{V}$ be a viscosity solution of the Euler system with the initial data $[\vr_0, \vm_0]$.

We say that a sequence $[\vr_n, \vm_n]_{n=1}^\infty$ of (weak) solutions to the Navier--Stokes system with
$\mu_n \searrow 0$, $\lambda_n \to 0$ starting from regular approximation $[\vr_{0,n}, \vm_{0,n}]_{n=1}^\infty$
is a \emph{generating sequence} for $\mathcal{V}$ if
\[
[\vr_n, \vm_n] \toSC \mathcal{V}.
\]

\end{Definition}

\subsection{Existence of viscosity solutions}

The following result is a direct consequence of Theorem \ref{PT1} on global
existence for the Navier--Stokes system and the subsequence principle stated in Proposition \ref{sP3}.

\begin{Theorem}[{\bf Existence of viscosity solution}] \label{eT1}

Let $\gamma > \frac{d}{2}$ and let $[\vr_0, \vm_0]$ be given finite energy initial data,
\[
\vr_0 \geq 0,\ \intO{ \left[ \frac{1}{2} \frac{|\vc{m_0}|^2}{\vr_0} + \frac{a}{\gamma - 1} \right] } < \infty.
\]

Then

{\bf(i)}
the Euler system admits a viscosity solution $\{ \mathcal{V}_{t,x} \}_{(t,x) \in (0,T) \times \Omega}$  in the sense of Definition~\ref{PD3};

{\bf (ii)} there exists a regular approximation $[\vr_{0,n}, \vm_{0,n}]$ of $[\vr_0, \vm_0]$, and $\mu_n \searrow 0$,
$\lambda_n \to 0$, such that the associated sequence of weak solutions
$[\vr_n, \vm_n]_{n=1}^\infty$ generates $\mathcal{V}$, specifically
\[
[\vr_n , \vm_n] \toSC \mathcal{V}.
\]

\end{Theorem}

\begin{proof}

It is a routine matter to construct {a} regular approximation $[\vr_{0,n}, \vm_{0,n}]$ of the initial data
$[\vr_0, \vm_0]$ specified in Definition \ref{PD3a}.
Let $\mu_n \searrow 0$, $\lambda_n \to 0$ be a sequence of viscosity coefficients. Let
$[\vr_n, \vm_n = \vr_n \vu_n]_{n=1}^\infty$ be a sequence of finite energy weak solutions to the Navier--Stokes system
with the initial data $[\vr_{0,n}, \vm_{0,n}]$, the existence of which is guaranteed by Theorem \ref{PT1}.

As the energy
\[
\intO{ \left[ \frac{1}{2} \frac{|\vm_n|^2}{\vr_n} + \frac{a}{\gamma - 1} \vr_n^\gamma \right] }
\ \mbox{is uniformly bounded for}\ t \in [0,T],
\]
we deduce
\[
\sup_{t \in [0,T]} \| \vr_n (t, \cdot) \|_{L^\gamma(\Td)} +
\sup_{t \in [0,T]} \| \vm_n (t, \cdot) \|_{L^{\frac{2 \gamma}{\gamma + 1}}(\Td; \mathds{R}^d))} \leq c,
\]
in particular $[\vr_n, \vm_n]_{n=1}^\infty$ {is} equi--integrable in $(0,T) \times \Td$. Consequently,
by virtue of Proposition~\ref{sP3}, there is a subsequence
$[\vr_{n_j}, \vm_{n_j} ]_{j=1}^\infty$ generating the viscosity solution $\mathcal{V}$ in the sense of Definition~\ref{eD1}.

\end{proof}

Of course, the existence result remains valid as long as the Navier--Stokes system admits {a} global--in--time finite energy weak solution
for {the} given initial data. Due to the result of Plotnikov and Weigant \cite{PloWei}, the existence still holds if $d=2$ and $\gamma \geq 1$. Similar extensions to a larger class of pressure--density state equations can be obtained via the existence theory developed in \cite[Chapter 7]{EF70}.

\subsection{Observables}

We show that the observables - invariants of the class of viscosity solutions - can be identified with
weak limits of solutions of the Navier--Stokes system.

\begin{Proposition}[{\bf Convergence of observables}] \label{eP1}

Let $B \in C(\mathds{R}^{d+1})$ satisfy
\[
|B (\vr, \vm)| \aleq (1 + \vr + |\vm|).
\]

Then the following is equivalent:
\begin{itemize}
\item
 $B$ is observable
in the sense of Definition \ref{PD4} for the initial data $[\vr_0, \vm_0]$;
\item
for any sequence of finite energy weak solutions $[\vr_n, \vm_n]_{n=1}^\infty$
of the Navier--Stokes system with viscosity coefficients $\mu_n \searrow 0$, $\lambda_n  \to  0$,
starting from regular approximation
$[\vr_{0,n}, \vm_{0,n}]_{n=1}^\infty$ of the initial data $[\vr_0, \vm_0]$, there holds
\begin{equation} \label{e2}
B(\vr_n, \vm_n) \to \Ov{B(\vr, \vm)} \ \mbox{weakly in}\ {L^r((0,T) \times \Td)}, \ r = \frac{2\gamma}{\gamma + 1}.
\end{equation}
\end{itemize}

\end{Proposition}

\begin{proof}

Suppose that $B \in \mathcal{O}[\vr_0, \vm_0]$. As the energy of $[\vr_{ n}, \vm_n]_{n=1}^\infty$ is bounded, there is a subsequence
such that
\begin{equation} \label{e11}
B(\vr_{n_j}, \vm_{n_j} ) \to \mathcal{B} \ \mbox{weakly in}\ L^r((0,T) \times \Td),\ r = \frac{2 \gamma}{\gamma + 1} \leq \gamma.
\end{equation}

By virtue of the subsequence principle stated in Proposition \ref{sP3}, there is yet another subsequence
$\{ n_l \}_{l=1}^\infty \subset \{ n_j \}_{j=1}^\infty$ such that
\[
[\vr_{n_l}, \vm_{n_l} ] {\toSC} \mathcal{V},
\]
where $\mathcal{V}$ is a viscosity solution of the Euler system with the initial data $[\vr_0, \vm_0]$. Since
$B$ is observable, convergence in {\eqref{P7}} yields
\[
\mathcal{B} = \left< \mathcal{V} ; B (\tilde{\vr}, \tilde{\vm} ) \right> \equiv \Ov{B(\vr, \vm)}.
\]
As the same argument can be applied to any subsequence satisfying  \eqref{e11} we get unconditional convergence
claimed \eqref{e2}.

The opposite implication can be shown by the same argument yielding a unique value
\[
\left< \mathcal{V} ; B (\tilde{\vr}, \tilde{\vm} ) \right> = \Ov{B(\vr, \vm)}
\]
for any viscosity solution $\mathcal{V}$, meaning $B \in \mathcal{O}[\vr_0, \vm_0]$.

\end{proof}

Thus the value of observables is uniquely determined as the weak limit of their composition with arbitrary generating sequence. In particular,
any such sequence generates a (unique) Young measure $\mathcal{V}$ if $C_c(R^{d + 1}) \subset \mathcal{O}[\vr_0, \vm_0]$. If this is the case, the measure $\mathcal{V}$ is the unique viscosity solution of the Euler system.

\section{Numerical method -- approximate solutions}
\label{N}

As shown in the preceding section, observables are independent of the particular choice of the vanishing
viscosity coefficients $\mu_n$, $\lambda_n$ and the approximation of the initial data $(\vr_{0,n}, \vm_{0,n})$ and as such reflect \emph{intrinsic} properties
of the limit Euler system.
Our principal objective is to show how to \emph{compute}
the observables by means of a suitable numerical scheme.
We distinguish three basic parameters characterizing the approximation process: the numerical step $h$, typically a mesh size, the artificial viscosity coefficients
$\mu$ and $\lambda$. In view of various restrictions imposed in particular by the consistency estimates, we focus
on parameters $(h, \mu, \lambda)$ ranging in a set $\mathcal{R}$,
\[
[h, \mu, \lambda] \in \mathcal{R} \subset (0,1] \times (0,1] \times [0,1], \ (0,0,0) \in \Ov{\mathcal{R}}.
\]
Typically,
\begin{equation}
\label{R}
\mathcal{R} = \left\{ [h, \mu, \lambda] \ \Big|\ 0 < h^\alpha < c \mu ;  \mbox{ either } \lambda = 0 \mbox{ or } 0 < h^\alpha \leq c \lambda   \right\}
\ \mbox{for some}\ \alpha > 0.
\end{equation}
The region $\mathcal{R}$ is determined by the range of parameters for which a particular numerical method converges,
cf. Theorem \ref{thm_dmvs} below.
In general, we consider the situation where the numerical step $h$ is largely dominated by the (artificial) viscosity, $h << \min(\mu, \lambda)$ if $\mu, \lambda >0$. This
is in line with our philosophy that the problem should be close to the generating sequence of the Navier--Stokes solutions
with low viscosity.

In what follows we introduce a particular numerical method, the so-called vanishing viscosity finite volume method. It is based on the model proposed in \cite{BREN, BREN1} by H.~Brenner as an alternative to the conventional Navier--Stokes  system, see also \cite{FLM18_brenner}.

\subsection{Viscosity Finite Volume (VFV) method}

We start by introducing a discretization of the computational domain and the discrete differential operators.
The physical domain $\Omega \equiv \mathcal{T}^d $, $d=2,3$, is decomposed into compact elements,
\[
 \mathcal{T}^d  = \bigcup_{K \in \grid_h} K.
\]
The elements $K$  are chosen to be  rectangular/cuboid. The mesh $\grid _h$  satisfies standard regularity properties.
The set of all faces $\sigma \in \partial K,$  $K \in \grid_h$ is denoted by $\Sigma.$
We suppose
$
|K| \approx h^d, \ |\sigma| \approx h^{d-1}
\ \mbox{ for any }\ K \in \grid_h, \mbox{ and } \sigma \in \Sigma,
$
where the parameter  $h\in(0,1)$  denotes the size of the mesh $\grid_h.$

We denote ${Q}_h$ the space of  functions  that are constant on each element $K \in \grid_h$, with the associated
projection:
\[
\Pi_h: L^1(\mathcal{T}^d ) \to {Q}_h,\ \Pi_h v = \sum_{K \in \grid_h} 1_K \frac{1}{|K|} \int_K
v \dx.
\]

For the average and jump of $v \in Q_h$ on a face $\sigma \in \Sigma$ we have the following discrete operators, respectively,
\[
\avs{v}(x) = \frac{v^{\rm in}(x) + v^{\rm out}(x) }{2},\ \ \
\jump{ v }  = v^{\rm out}(x) - v^{\rm in}(x).
\]
Here $v^{\rm out}, v^{\rm in}$ are respectively the outward, inward limits with respect to a given normal $\vn$ belonging to $\sigma \in \Sigma.$  We define the following discrete differential operators
for piecewise constant functions $r_h \in Q_h,$ $\vvh \in \vQh \equiv Q_h^d$:
\begin{equation*}
\begin{aligned}
&\Gradh r_h  = \sum_{K \in \grid_h}  (\Gradh r_h)_K 1_K,  \quad
(\Gradh r_h)_K = \sum_{\sigma\in \pd K} \frac{|\sigma|}{|K|} \avs{r_h} \vn,
\\
& \Delta_h r_h = \sum_{K\in \grid_h} (\Delta_h r_h)_K 1_K , \quad
 (\Delta_h r_h)_K = \sum_{\sigma \in \pd K}  \frac{|\sigma|}{|K|} \frac{\jump{r_h}}{h} ,
\\
& \gradd r_h =  \sum_{\sigma\in\pd K} \left(\gradd r_h \right)_{\sigma} 1_\sigma,  \quad \left(\gradd r_h\right) _{\sigma} =\frac{\jump{r_h} }{h} \vc{n}
\\
&\Divh \vvh  = \sum_{K \in \grid_h}  (\Divh  \vvh)_K 1_K, \quad
(\Divh \vvh)_K = \sum_{\sigma\in \pd K} \frac{|\sigma|}{|K|} \avs{\vvh} \cdot \vn
\\
&\Divup (r_h , \vvh) = \sum_{K \in \grid_h}  1_K \Divup(r_h , \vvh)_K, \quad
 \Divup(r_h , \vvh)_K = \sumsfK \frac{|\sigma|}{|K|} F_h(r_h,\vvh).
 \end{aligned}
\end{equation*}
Here $F_h$ denotes the classical upwind flux ${\Up}[r_h, \vvh]$  with an additional numerical diffusion, specifically
\begin{eqnarray*}
\Up [r_h, \vv_h]  &=&
\avs{r_h} \ \avs{\vv_h} \cdot \vc{n} - \frac{1}{2} |\avs{\vv_h} \cdot \vc{n}| \jump{ r_h } \\
F_h(r_h,\vvh)
&=& {\Up}[r_h, \vvh] - \muh \jump{ r_h } = \avs{r_h} \ \avs{\vv_h} \cdot \vc{n}
- \left( \muh + \frac{1}{2} |\avs{\vv_h} \cdot \vc{n}| \right)\jump{ r_h },
\ -1 < \varepsilon .
\end{eqnarray*}
 Moreover, we denote ${\bf F}_h({\bf r}_h,\vvh) = (F_h(r_{1,h},\vvh), \cdots, F_h(r_{d,h},\vvh) )^T$ for ${\bf r}_h = (r_{1,h},\cdots, r_{d,h})^T$.

In order to discretize the time evolution in $[0,T]$ we introduce a time step $\Delta t > 0,$ $\Delta t \approx h,$ and denote
\[
t_k = k \Delta t,\ k=1,2,\dots, N_T.
\]
Furthermore,
\begin{align*}
v^k(x) = v(t^k,x)  \ \mbox{ for all } \  x\in \Omega,\ t^k=k\,\Delta t \ \mbox{ for } k=0,1, \ldots, N_T.
\end{align*}
The time derivative $\frac{\partial {v} }{\partial t}$  is approximated by the backward Euler finite difference
\[
 \frac{\partial {v} }{\partial t}  \Big|_{t^k} \approx D_t {v}^k \equiv \frac{ {v}^k - {v}^{k-1} }{\Delta t}.
\]
Finally,
we introduce a piecewise constant interpolation in time of the discrete values $v^k$,
\begin{align}\label{NM_TD}
v_h(t,\cdot) = v_0 \mbox{ for } t<\Delta t,\ &
v_h(t,\cdot)=v^k \mbox{ for } t\in [k \Delta t,(k+1) \Delta t),\ k=1,2,\ldots,N_T.
\end{align}

\noindent
Let
$(\vrh^0, \vuh^0) = (\Pi_h\vr_0, \Pi_h\vu_0) \in Q_h\times \vQh,$  $\vmh^0 \equiv \vrh^0 \vuh^0 $ be the discrete initial data.

\begin{Definition}[{\bf VFV numerical scheme}] \label{DD1}

A pair $(\vrh, \vmh \equiv \vrh \vuh)$ of piecewise constant functions (in space and time) is a numerical approximation of the Euler system \eqref{P1} by  \emph{viscosity finite volume (VFV) method} if the following system of discrete equations holds:
\begin{eqnarray}
D_t \vrh  + \Divup(\vrh , \vuh ) &=&0,
  \nonumber \\
D_t (\vrh\vuh)  + \Divup(\vrh\vuh,  \vuh ) +\Gradh p_h    &=&  \mu(h) \Delta_h \vuh + \nu(h) \Gradh \Divh \vuh , \label{flm}
\end{eqnarray}
where $\nu(h) =  \frac  1 3 \mu(h) +\lambda(h)$ and the discrete pressure is $p_h= a \vrh^\gamma .$
\end{Definition}

It is convenient to rewrite \eqref{flm}  in the weak form
\begin{subequations}\label{scheme}
\begin{align}
&\intO{ D_t \vrh \varphi_h } - \sum_{ \sigma \in \facesint } \intSh{  F_h(\vrh,\vuh)
\jump{\varphi_h}   } = 0 \quad \mbox{for all } \varphi_h \in {Q}_h,\label{scheme_den}\\
&\intO{ D_t  (\vrh \vuh) \cdot \bfphi_h } - \sum_{ \sigma \in \facesint } \intSh{ {\bf F}_h(\vrh \vuh,\vuh)
\cdot \jump{\bfphi_h}   }- \sum_{ \sigma \in \facesint } \intSh{  \avs{p(\vr_h)} \vc{n} \cdot \jump{ \bfphi_h }  } \nonumber \\
&= - \mu(h) \frac{1}{h} \sum_{ \sigma \in \facesint } \intSh{ \jump{\vuh}  \cdot
\jump{\bfphi_h}  } - \nu(h) \intO{ \Divh \vuh \Divh \bfphi_h },
\quad \mbox{for all }
\bfphi_h \in \vQh . \label{scheme_mom}
\end{align}
\end{subequations}

Note that the VFV method (\ref{flm}) or \eqref{scheme} mimics the  physical process of \emph{vanishing
viscosity limit} in the Navier--Stokes system discussed in Section \ref{e}. In particular, we expect to recover the observables as limits of the numerical solutions as long as the numerical step $h$ is largely dominated by the (vanishing) viscosity coefficients.

\subsection{Structure preserving properties and convergence of VFV method }

We recall some fundamental properties of the VFV method shown in \cite{FeiLukMizShe,FLM18_brenner}.

\begin{itemize}
\item {\bf Conservation of discrete mass}\\
\begin{equation*}
\intO{ \vrh (t, \cdot) } = \intO{ \vr_{0,h} } = M_0 > 0,\,
\ t \geq 0.
\end{equation*}
\item {\bf Positivity  of the discrete density}\\
$$
\vrh(t) >  0  \ \mbox{ for any } t > 0 \mbox{ provided }
\vrh(0) > 0.
$$
\item {\bf Discrete total energy dissipation}\\
\[
D_t \intTd{ \left(\frac{1}{2} \vrh |\vuh|^2  + \frac{ (\vrh)^\gamma}{\gamma-1} \right)  }
+  \mu(h)   \norm{ \gradd  \vuh}_{L^2}^2  + \nu(h) \norm{\Divh \vuh}_{L^2}^2 \leq 0.
\]
\end{itemize}

In \cite{FeiLMMiz,FLM18_brenner}, we have established
the following convergence result for the VFV method, see also \cite{FeLMMiSh} for further details.

\begin{Theorem}[{\bf Weak convergence of VFV method/Euler limit}]\label{thm_dmvs}
Let $\{\vrh,\vmh\}_{h \searrow 0}$ be a family of solutions
generated by the VFV method  \eqref{flm} with $\Delta t \approx h$. Let the initial data satisfy
\[
\vr_0 \in L^\infty (\Omega), \  \vc{m}_0 \in L^\infty(\Omega), \qquad \vr_0>0.
\]

Further, we suppose that $\mu(h)\geq h^\alpha,$ $\nu(h)\geq h^\alpha$, and $\mu(h) \to 0$, $\nu(h) \to 0,$ cf.~\eqref{R},
\begin{equation}\label{VE_alpha}
\begin{split}
& -1 < \eps 
\mbox{ and } \ 0 < \alpha <  2 - \frac{ d/3 + 1 + \eps }{\gamma} \quad \mbox{ for } \gamma\in(1,2),
\\
\mbox{or}\quad  & -1 < \eps \ \mbox{ and } \ 0 < \alpha <  2 - \frac{ d}{\gamma} \quad \mbox{ for } \gamma \geq 2.  \\
\end{split}
\end{equation}

Then at least for a subsequence the numerical solutions $\{\vrh,\vm_h\}_{h \searrow 0}$ of the VFV method \eqref{flm} give rise to a Young measure
$\{ {\Nu}\}_{(t,x)\in(0,T)\times\Omega} $ that represents a dissipative measure valued solution of the Euler system (\ref{P1}) in the sense of Definition~\ref{DMV}.\\[0.1cm]


\end{Theorem}


As suggested by Elling \cite{ELLI1} the  viscosity solution may depend on the relation between the viscosity coefficients $\mu, \, \lambda$. To avoid this ambiguity we consider the specific case $\lambda = 0,$ meaning $\nu = \frac 1 3 \mu.$  Note that this corresponds to the Stokes hypothesis for gases.

The VFV method is of hybrid type. On the one hand, it yields a solution to the Euler system provided the viscosity coefficients
go to zero along with the numerical step as in \eqref{VE_alpha}. On the other hand, it is
a converging numerical approximation of the Navier--Stokes system~\eqref{P1} as long as we set the viscosity coefficients $\mu(h) = \rm{const.},$ $\nu(h) = \rm{const.}$  Indeed, as shown in \cite{FeiLukMizShe},
\begin{equation}
\label{konv}
\begin{split}
\vrh  &\rightarrow \vr \mbox{ (strongly) in } L^{1}(0,T)\times \Td ),\\
\vuh  &\rightarrow \vu \mbox{ (strongly) in } L^1((0,T)\times \Td; \mathds{R}^d ),
\end{split}
\end{equation}
as soon as the Navier--Stokes system admits a strong solution $(\vr, \vu)$ on $(0,T)\times \Td.$

In particular, there exists $\Ov{\vr}$ such that
$$
\left| \left\{ \vr_{h} \geq \Ov{\vr} \right\} \right|  \to 0  \qquad  h \searrow  0,
$$
where $| \cdot |$ denotes the Lebesque measure on $(0,T) \times \Td.$
As $\vr_h$ is constant on each element, this is equivalent to
\begin{equation*}\label{podminka}
\#\left\{K \times [t_i, t_{i+1}); K \in \grid_h, i=1 \dots N_T \Big| \vr_{h} \geq \Ov{\vr} \right\} h^{d+1}  \to 0 \mbox{ for } h \searrow  0.
\end{equation*}
Interestingly, this condition is also sufficient for the convergence claimed in \eqref{konv}.

\begin{Theorem}[{\bf Strong convergence of VFV method/Navier--Stokes limit}]\label{thm_convergence}
Suppose $\mu(h)= {\rm{const.}}$, $\lambda(h) = \nu(h) - \frac 1 3 \mu(h) = 0.$
Let the initial data belong to the class
\begin{equation}
\label{init_data}
\vr_0 \in W^{3,2}(\Td), \ \vr_0 >0,\  \vm_0 = \vr_0 \vu_0, \ \vu_0 \in W^{3,2}(\Td; \mathds{R}^d).
\end{equation}
Let $\{\vr_h, \vu_h \}_{h \searrow 0}$ be a family of solutions of the VFV method \eqref{flm}.

Then the following is equivalent:
\begin{itemize}
\item[\textbf{i)}]
\begin{align} \label{SC}
\vrh & \rightarrow \vr \mbox{ (strongly) in } L^{1}((0,T)\times \Td ), \nonumber \\
\vuh &\rightarrow \vu \mbox{ (strongly) in } L^1((0,T)\times \Td; \mathds{R}^d ) \mbox{ for } h \searrow 0,
\end{align}
where
$[\varrho, \vu]$ is the classical solution to the Navier--Stokes system (\ref{P3}) with the initial data $[\vr_0, \vu_0]$.
\item[\textbf{ii)}]
There exists $h_n \searrow 0$  such {that} the {limits}
\begin{align} \label{SW}
\vr_{h_n} & \rightarrow \vr \mbox{ weakly-(*)} \mbox{ in } L^{\infty}(0,T; L^\gamma(\Td) ),\\
\vu_{h_n} &\rightarrow \vu \mbox{ weakly-(*)} \mbox{ in }  L^\infty(0,T; L^{\frac{2 \gamma}{\gamma + 1}}(\Td); \mathds{R}^d )  \mbox{ for } n\to \infty
\end{align}
{are} $C^1$ {functions}.
\item[\textbf{iii)}]
There exists $\Ov{\vr} > 0$  and a subsequence $h_n \searrow 0 $ such that
\begin{equation} \label{podminka1}
\#\left\{K \times [t_i, t_{i+1}); K \in \grid_h, i=1, \dots, N_T \Big| \vr_{h_n} \geq \Ov{\vr} \right\} h_n^{d+1}  \to 0 \mbox{ for } n \to \infty.
\end{equation}
\end{itemize}
\end{Theorem}

\begin{proof}
The equivalence of $\textbf{i)}$ and $\textbf{ii)}$ follows from \cite[Theorem~5.8]{FeLMMiSh}.
We proceed by showing the equivalence of $\textbf{i)}$ and $\textbf{iii)}.$ In view of the previous discussion it is enough to show that \eqref{podminka1} implies strong convergence to a classical solution \eqref{SC}.  By virtue of the standard results on the local existence, see, e.g.,  \cite{VAZA},  the Navier--Stokes system admits a classical solution defined on an interval $[0, T_{max}),$  $T_{max} > 0.$ Due to {\cite[Theorem 5.3]{FeiLukMizShe}},  we get  \eqref{SC} on  $[0, \tau]$ for any $\tau < T_{max}.$  Moreover,
$$
\vr_{h_n} \to \vr \mbox{ weakly-(*)} \mbox{ in }{ L_{\mathrm{weak}-(*)}^\infty (0, T; L^\gamma(\Td))}.
$$
We claim that \eqref{podminka1} implies that $0 < \vr \leq \Ov{\vr}.$  It follows from \eqref{podminka1} and H\"older's inequality that
$$
\int \int_{\vr_{h_n} \geq \Ov{\vr}} \vr_{h_n} \dx \mbox{d}t
\leq |\vr_{h_n} \geq \Ov{\vr}|^\frac{1}{\gamma'} \| \vr_{h_n}\|_{L^\gamma( (0,T) \times \Td)} \to 0  \qquad n\to \infty.
$$
Rewriting $\vr_{h_n}$ in the following way
$$
\vr_{h_n} = T_R(\vr_{h_n}) + \left( \vr_{h_n} -  T_R(\vr_{h_n}) \right),
$$
$T_R(r) \equiv \min\{r, R \}, $  $ R \geq \Ov{\vr}$
we have
\begin{eqnarray}
&& \| \vr_{h_n} -  T_R(\vr_{h_n}) \|_{L^1((0,T) \times \Td)} = \int_0^T \int_{\Td} \vr_{h_n} -  T_R(\vr_{h_n}) \dx \dt, \nonumber \\
&& \int \int_{\vr_{h_n} \geq R} \vr_{h_n} \dx \dt \leq  \int \int_{\vr_{h_n} \geq \Ov{\vr}} \vr_{h_n} \dx \dt \to 0.
\end{eqnarray}
Consequently,
$$
T_R(\vr_{h_n}) \to \vr \mbox{ weakly-(*)} \mbox{ in } { L_{\mathrm{weak}-(*)}^\infty (0, T; L^\gamma(\Td))}
$$
for any $R \geq \Ov{\vr},$ which yields that $\vr \leq \Ov{\vr}.$ Due to the conditional regularity criterion of Sun, Wang, Zhang \cite{SuWaZh} $T_{max} > T$ which concludes the proof.
\end{proof}

\medskip

We proceed by applying the previous convergence results to deduce S--convergence of the numerical solutions obtained by the VFV method. In particular, we show that the numerical solutions $\{\vrh, \vm_h\}_{h \searrow 0}$ converge to a viscosity solution of the Euler system \eqref{P3} in the sense of Definition~\ref{eT1}. Similarly to Theorem~\ref{thm_convergence} we assume here and hereafter that $\lambda(h) = 0.$

We denote $[\rN,\ \mN]$ the numerical solution obtained from the VFV method
starting from a regular approximation  $[\vr_{0,n}, \vm_{0,n}]$ of the initial data
$[\vr_0, \vm_0]$, with artificial viscosity $\mu_n $, and the numerical step $h > 0$.
Our principal hypothesis is that possible density concentrations may arise
only on a small set. Specifically, we assume that condition \eqref{podminka1} , i.e.
\begin{equation*}
\#\left\{K \times [t_i, t_{i+1}); K \in \grid_h, i=1, \dots, N_T \Big| \vr_{h_k} \geq \Ov{\vr} \right\} h_k^{d+1}  \to 0 \mbox{ for } k \to \infty.
\end{equation*}
holds for any fixed $\mu_n > 0,$  $[h_k, \mu_n, 0] \in \Ov{\mathcal{R}}.$
Note that this holds if numerical densities $\vr_{h}$ are uniformly bounded, an assumption quite common in the literature.

Our first result addresses the problem of convergence to a given viscosity solution.

\begin{Theorem}[{\bf Convergence of VFV method}] \label{NT1}

Let $\{ \mathcal{V}_{t,x} \}_{(t,x) \in (0,T) \times \Td}$ be a viscosity solution of the Euler system
with the initial data $[\vr_0, \vm_0].$
Assume that $\mathcal{V}$ admits a generating sequence of solutions of the Navier--Stokes system starting
from a regular data approximation $[\vr_{0,n}, \vm_{0,n}]$ and with
viscosity coefficients  $\mu_n \searrow 0$ in the sense of Definition~\ref{eD1}.
Let $[\rN,\ \mN]$ denote the numerical solution obtained from the VFV scheme with the
initial data $[\vr_{0,n}, \vm_{0,n}]$, the artificial viscosities $\mu_n $, and the numerical step $h> 0$,
$[h, \mu_n, 0] \in \Ov{\mathcal{R}}$.
Finally, suppose that the condition \eqref{podminka1} holds for any fixed $\mu_n$.

Then there exists $H_n \searrow 0$ such that
\begin{equation} \label{N4}
[\vr_{h_n, \mu_n} \, \vm_{h_n, \mu_n} ] \toSC \mathcal{V} \ \mbox{ as } n \to \infty \mbox{ whenever } \ 0 < h_n \leq H_n.
\end{equation}

\end{Theorem}

 \begin{Remark}
Although we  know that $H_n < \mu_n, $  an explicit formula relating $H_n$ to $\mu_n$ is not available.

In addition, in view of Theorem~\ref{thm_convergence},
hypothesis~\eqref{podminka1} is equivalent to the fact that the viscosity solution of the Euler system is generated by a sequence of classical solutions
of the Navier--Stokes system. The situation is unclear for the viscosity solution generated by a sequence of weak solutions.
\end{Remark}

\begin{proof}

In view of hypothesis \eqref{podminka1} we can apply Theorem~\ref{thm_convergence}
for any fixed $ n>0$
to conclude that
\begin{equation} \label{N5}
\left\| \vr_{h, \mu_n} - \vr_n \right\|_{L^1((0,T) \times \Td)} +
\left\| \vm_{h, \mu_n} - \vm_n \right\|_{L^1((0,T) \times \Td; \mathds{R}^d)} \to 0
\ \mbox{as}\ h \to 0
\end{equation}
where $\{ \vr_n, \vm_n \}_{n=1}^\infty$ is the generating sequence.

In view of \eqref{N5}, for any $\delta_n \searrow 0$ there is $H_n > 0$ such that
\begin{equation}
\label{N55}
\left\| \vr_{h, \mu_n} - \vr_n \right\|_{L^1((0,T) \times \Td)} +
\left\| \vm_{h, \mu_n} - \vm_n \right\|_{L^1((0,T) \times \Td; \mathds{R}^d)} \leq \delta_n
\end{equation}
as soon as $0 < h < H_n$. Consequently, if $0 < h_n < H_n$, the sequences
$[ \vr_{h_n, \mu_n} , \vm_{h_n, \mu_n}  ]$ and $[\vr_n, \vm_n]$ are statistically equivalent in the sense of \cite[Section 3, Definition 3.1]{Fei2020A}.
In particular, they generate the same (S)--limit, in other words
\[
[\vr_{h_n, \mu_n} , \vm_{h_n, \mu_n} ] \toS \mathcal{V},
\]
see \cite[Theorem 3.2, Remark 3.3]{Fei2020A}.
Moreover, since \eqref{N55} holds for any $n$, we can strengthen the above convergence
to
\[
[\vr_{h_n, \mu_n} , \vm_{h_n, \mu_n} ] \toSC \mathcal{V}.
\]

\end{proof}

The above result can be reformulated in terms of observables as follows.

\begin{Theorem}[{\bf Approximation of observables}] \label{NT1a}

Let $B \in C({\mathds{R}}^{d+1})$,
\[
|B(\vr, \vm)| \aleq (1 + \vr + |\vm| )
\]
be observable, $B \in \mathcal{O}[\vr_0, \vm_0]$. Let $[\vr_{0,n}, \vm_{0,n}]_{n=1}^\infty$ be a regular approximation
of the initial data $[\vr_0, \vm_0]$ in the sense of Definition \ref{PD3a}.
Denote $[\rN,\ \mN]$  the numerical solution obtained from the VFV method with the
initial data $[\vr_{0,n}, \vm_{0,n}]$, and
\[
[h, \mu_n,0] \in \Ov{\mathcal{R}}, \mu_n \searrow 0.
\]
Finally, suppose that the condition \eqref{podminka1} holds for each fixed $\mu_n$.

Then there exists $H_n \searrow 0$ such that
\begin{equation} \label{N4a}
\begin{split}
B(\vr_{h_n}, \vm_{h_n}) &\to \Ov{B(\vr, \vm)} \ \mbox{weakly in}\ L^1((0,T) \times \Omega)\\ \mbox{whenever} \ 0 &< h_n \leq H_n.
\end{split}
\end{equation}
In particular, if
\begin{equation} \label{N4b}
\frac{1}{N} \sum_{n=1}^N s_{n,N} B(\rNN, \mNN) \to \mathcal{B} \ \mbox{in}\ L^1((0,T) \times \Td),\ 0 < h_n < H_n
\end{equation}
for some regular summation method, then $\mathcal{B} = \Ov{B(\vr, \vm)}$.
\end{Theorem}

\begin{Remark} \label{NR1}

Note that \eqref{N4b} holds if the sequence of numerical solutions $[\rNN, \mNN ]_{n=1}^\infty$ is S--convergent
with respect to a regular summation method $\{ s_{n,N} \}_{n=1}^\infty$. We point out that $[\rNN, \mNN ]_{n=1}^\infty$ need not to generate a Young measure.

\end{Remark}

\begin{proof}

The proof basically copies the steps of the proof of Theorem \ref{NT1}. Thus, keeping
$n > 0$ fixed, we let $h\searrow 0$ in the sequence of numerical solutions
\[
[\vr_{h, \mu_n}, \vm_{h, \mu_n}].
\]
We obtain a sequence of classical solutions $[\vr_n, \vm_n]$ of the Navier--Stokes system with the viscosity coefficients $\mu_n$, $\lambda_n=0$ such that for any $\delta_n > 0$
\begin{equation} \label{N4c}
\left\| \vr_{h, \mu_n} - \vr_n \right\|_{L^q((0,T) \times \Td)} +
\left\| \vm_{h, \mu_n} - \vm_n \right\|_{L^q((0,T) \times \Td; \mathds{R}^d)} \leq \delta_n,\ 1 \leq q < \infty,
\end{equation}
whenever $0 < h < H_n$.

In view of Proposition \ref{eP1},
\[
B(\vr_n, \vm_n ) \to \Ov{B(\vr, \vm)} \ \mbox{weakly in}\ L^1((0,T) \times \Td);
\]
whence \eqref{N4a} follows from \eqref{N4c}.

Convergence~\eqref{N4b}  follows directly from \eqref{N4a}.

\end{proof}

Note the subtle difference between Theorem \ref{NT1} and Theorem \ref{NT1a}. In Theorem \ref{NT1}, the approximate sequence of numerical solutions starts with the same initial data as the generating S--convergent sequence of solutions of the Navier--Stokes system. The S--convergence is then inherited by the numerical solutions if the step $h_n$ is small enough. In Theorem \ref{NT1a}, the viscosity coefficients as well as the initial data are arbitrary. The limit value $\Ov{B(\vr, \vm)}$ of an observable $B$ can be recovered
as a \emph{strong} limit of weighted averages as long as the latter exists. In particular, this is the case as soon as the family of
numerical solutions is S--convergent. In both cases, the VFV scheme effectively computes the viscosity solution of the Euler system.

\subsection{Numerical experiments}

 In order to illustrate our theoretical results we have to consider a problem that is known to produce oscillatory solutions to the Euler system.
A prominent example is the celebrated Kelvin--Helmholtz problem~\cite{Helmhotz, Kelvin}, however typically studied for the complete Euler system. As observed, for example in \cite{paper1}, similar effects can be produced also for the barotropic Euler system driven by a potential volume force producing strong stratification.  Motivated by the prevailing amount of existing literature, we focus directly on the full Euler system,
cf.~\cite{FeiLukMizSheWa, FLM18_brenner} for the VFV method.

We choose the following initial data
\[
(\vr, u_1, u_2,p)(x)=\left\{
\begin{array}{ll}
(2, -0.5, 0, 2.5), & \text{if}\ I_1<x_2<I_2\\
(1, 0.5, 0, 2.5), & \text{otherwise.}
\end{array}
\right.
\]
Here the interface profiles
\[
I_j = I_j(x):=J_j+\epsilon Y_j(x), \quad j=1,2,
\]
are chosen to be small perturbations around the lower $J_1=0.25$ and the upper $J_2=0.75$ interface, respectively. Further,
\[
Y_j(x)=\sum_{n=1}^{m}a_j^n\cos(b_j^n+2n\pi x_1),\quad j=1,2,
\]
where $a_j^n\in[0,1]$ and $b_j^n \in[-\pi,\pi]$, $i=1,2$, $n=1,\ldots,m$ are arbitrary but fixed numbers. The coefficients $a_j^n$ have been normalized such that $\sum_{n=1}^{m}a_j^n=1$ to guarantee that $|I_j(x)-J_j|\leq \epsilon$ for $j=1,2$. We have set $m=10$ and $\epsilon=0.01.$

In what follows we present the numerical simulations obtained by the VFV method with $\alpha = 0.9,$ $\varepsilon = {0.5}$ at the final
time $T=2.$ It is the time when small-scales vortex sheets have been already formed at the interfaces.
Table~\ref{tab1} documents
a representative part of our extensive numerically simulations and presents experimental convergence study for different
regular summation methods corresponding to the following choices of the weight function $\omega$:
\begin{equation*}
\omega_{\text{equal}}(t) = 1, \;
\omega_{\text{quad}}(t) = t(1-t), \;
\omega_{\text{sin2}}(t) = \sin^2(\pi t), \;  \omega_{\text{exp}}(t) = \exp\left(\frac{-1}{t(1-t)}\right) \mathbf{1}_{(0,1)} \
\mbox{ for } t\in [0,1].
\end{equation*}
Recall that for every non-negative $\omega \in C[0,1]$ the infinite matrix $\{ s_{n,N} \}_{n=1, N=1}^\infty$
\begin{equation*}
	s_{n,N} =
	\begin{cases}
		1, & \text{for }1\leq n \leq N,
		\text{ if }\sum\limits_{m=1}^N \omega(\frac{m}{N}) = 0,\\[3mm]
		{N\omega(\frac{n}{N})}{\Big/}\left({\sum\limits_{m=1}^N \omega(\frac{m}{N})}\right), &\text{for }1\leq n \leq N,
		\text{ if }\sum\limits_{m=1}^N \omega(\frac{m}{N}) > 0, \\[3mm]
		0, &\text{otherwise,}
	\end{cases}
\end{equation*}
is a regular summation method.

We choose uniform meshes having $k \times k $  mesh cells and the mesh parameter $h = 1/k$. Here $k$  is taken from the set
$ \{32\ell \,{|}\,\ell \in \mathds{N},\, 1\leq \ell \leq 64\}.$

Table~\ref{tab1}  presents the experimental order of convergence of weighted averages of the density computed on $k \times k$ meshes
for all considered summation methods using $\omega_\text{equal}, \ \omega_\text{quad}, \ \omega_\text{sin2}, \ \omega_\text{exp}$.
The error is computed in the $L^1$-norm and the reference solution is chosen as the average over all computed solutions with the weight function $\omega$ of the respective column.
We present only the errors for averages with $k$ up to $1024$, since otherwise
the set of simulations used to compute the averages is already very close to the reference solution.


When considering different subsequences and reference solutions, the errors of all summation methods are typically within the same order of magnitude, though the convergence rate may differ.
Analogous convergence results using
the Ces\`aro average, i.e.~$\omega=\omega_{\text{equal}}$, as the reference solution for all summation methods are presented in Tables~\ref{tab1a}, \ref{tab2a}, cf.~Appendix.


\begin{table}
\caption{Convergence study in the $L^1$-norm for averages of the density at time $T = 2$ using different weight functions.}
\label{tab1}
\begin{center}
	\begin{tabular}{|cV{3}c|cV{3}c|cV{3}c|cV{3}c|c|}
		\hline
		$k$&
		\multicolumn{2}{cV{3}}{$\omega_{\text{equal}}$}& \multicolumn{2}{cV{3}}{$\omega_{\text{quad}}$}& \multicolumn{2}{cV{3}}{$\omega_{\text{sin2}}$}& \multicolumn{2}{c|}{$\omega_{\text{exp}}$}\\
		\cline{2-9}
		(up to)& error   & order& error   & order& error   & order& error   & order\\
		 \hline64 & 1.46e-01 & - & 2.03e-01 & - & 2.06e-01 & - & 2.08e-01 & - \\
		 \hline
		 96 & 1.18e-01 & 0.53 & 1.54e-01 & 0.67 & 1.58e-01 & 0.66 & 1.60e-01 & 0.65 \\
		 \hline
		 128 & 9.86e-02 & 0.61 & 1.24e-01 & 0.76 & 1.25e-01 & 0.82 & 1.22e-01 & 0.94 \\
		 \hline
		 160 & 8.48e-02 & 0.67 & 1.03e-01 & 0.82 & 1.02e-01 & 0.90 & 9.85e-02 & 0.96 \\
		 \hline
		 192 & 7.44e-02 & 0.72 & 8.80e-02 & 0.88 & 8.61e-02 & 0.93 & 8.37e-02 & 0.89 \\
		 \hline
		 224 & 6.62e-02 & 0.76 & 7.64e-02 & 0.92 & 7.47e-02 & 0.92 & 7.29e-02 & 0.90 \\
		 \hline
		 256 & 5.96e-02 & 0.79 & 6.75e-02 & 0.93 & 6.59e-02 & 0.93 & 6.45e-02 & 0.92 \\
		 \hline
		 288 & 5.42e-02 & 0.80 & 6.04e-02 & 0.94 & 5.90e-02 & 0.95 & 5.77e-02 & 0.94 \\
		 \hline
		 320 & 4.97e-02 & 0.81 & 5.47e-02 & 0.95 & 5.32e-02 & 0.97 & 5.21e-02 & 0.96 \\
		 \hline
		 352 & 4.60e-02 & 0.82 & 4.99e-02 & 0.97 & 4.85e-02 & 0.98 & 4.75e-02 & 0.98 \\
		 \hline
		 384 & 4.28e-02 & 0.82 & 4.58e-02 & 0.98 & 4.45e-02 & 0.99 & 4.36e-02 & 0.99 \\
		 \hline
		 416 & 4.01e-02 & 0.82 & 4.23e-02 & 0.98 & 4.11e-02 & 1.00 & 4.02e-02 & 1.00 \\
		 \hline
		 448 & 3.78e-02 & 0.82 & 3.94e-02 & 0.99 & 3.81e-02 & 1.01 & 3.73e-02 & 1.00 \\
		 \hline
		 480 & 3.57e-02 & 0.82 & 3.68e-02 & 0.99 & 3.55e-02 & 1.01 & 3.48e-02 & 1.01 \\
		 \hline
		 512 & 3.38e-02 & 0.82 & 3.45e-02 & 0.99 & 3.33e-02 & 1.02 & 3.26e-02 & 1.01 \\
		 \hline
		 544 & 3.22e-02 & 0.82 & 3.25e-02 & 0.99 & 3.13e-02 & 1.02 & 3.07e-02 & 1.00 \\
		 \hline
		 576 & 3.07e-02 & 0.82 & 3.07e-02 & 0.99 & 2.95e-02 & 1.02 & 2.90e-02 & 1.00 \\
		 \hline
		 608 & 2.94e-02 & 0.82 & 2.91e-02 & 0.99 & 2.79e-02 & 1.02 & 2.75e-02 & 1.00 \\
		 \hline
		 640 & 2.82e-02 & 0.82 & 2.77e-02 & 0.99 & 2.65e-02 & 1.01 & 2.61e-02 & 0.99 \\
		 \hline
		 672 & 2.71e-02 & 0.82 & 2.64e-02 & 0.99 & 2.52e-02 & 1.01 & 2.49e-02 & 0.99 \\
		 \hline
		 704 & 2.61e-02 & 0.82 & 2.52e-02 & 0.99 & 2.41e-02 & 1.01 & 2.38e-02 & 0.99 \\
		 \hline
		 736 & 2.51e-02 & 0.83 & 2.41e-02 & 0.99 & 2.30e-02 & 1.00 & 2.27e-02 & 0.99 \\
		 \hline
		 768 & 2.42e-02 & 0.84 & 2.31e-02 & 0.99 & 2.21e-02 & 1.00 & 2.18e-02 & 0.99 \\
		 \hline
		 800 & 2.34e-02 & 0.87 & 2.22e-02 & 0.99 & 2.12e-02 & 1.00 & 2.09e-02 & 0.99 \\
		 \hline
		 832 & 2.26e-02 & 0.88 & 2.13e-02 & 1.00 & 2.04e-02 & 1.00 & 2.01e-02 & 0.99 \\
		 \hline
		 864 & 2.19e-02 & 0.90 & 2.05e-02 & 1.01 & 1.96e-02 & 1.01 & 1.94e-02 & 1.00 \\
		 \hline
		 896 & 2.11e-02 & 0.92 & 1.98e-02 & 1.03 & 1.89e-02 & 1.02 & 1.87e-02 & 1.01 \\
		 \hline
		 928 & 2.04e-02 & 0.94 & 1.91e-02 & 1.05 & 1.82e-02 & 1.03 & 1.80e-02 & 1.01 \\
		 \hline
		 960 & 1.98e-02 & 0.98 & 1.84e-02 & 1.07 & 1.76e-02 & 1.05 & 1.74e-02 & 1.03 \\
		 \hline
		 992 & 1.91e-02 & 1.02 & 1.77e-02 & 1.09 & 1.70e-02 & 1.07 & 1.68e-02 & 1.05 \\
		 \hline
		 1024 & 1.85e-02 & 1.05 & 1.71e-02 & 1.13 & 1.64e-02 & 1.10 & 1.63e-02 & 1.07 \\
		 \hline
	\end{tabular}
\end{center}
\end{table}

In Figure~\ref{fig:Error_Plot_different_subsequences} we compare the  Ces\` aro averages of different subsequences of numerical solutions.
Specifically, we calculate the following errors
	\begin{align*}
		\text{sequence 1:}\quad &\left\| \sum_{\ell =1}^K\left(\frac{1}{3\ell}\sum_{m=1}^{3\ell} \varrho_{32m} -\frac{1}{\ell}\sum_{m=1}^{\ell} \varrho_{32(3m-2)}\right)\right\|_{L^1((0,T) \times \Td)}
		\\
		\text{sequence 2:}\quad &\left\| \sum_{\ell=1}^K\left(\frac{1}{3\ell}\sum_{m=1}^{3\ell} \varrho_{32m} -\frac{1}{\ell}\sum_{m=1}^{\ell} \varrho_{32(3m-1)}\right)\right\|_{L^1((0,T) \times \Td)}
		\\
		\text{sequence 3:}\quad &\left\| \sum_{\ell=1}^K\left(\frac{1}{3\ell}\sum_{m=1}^{3\ell} \varrho_{32m} -\frac{1}{\ell}\sum_{m=1}^{\ell} \varrho_{32(3m)}\right)\right\|_{L^1((0,T) \times \Td)}
	\end{align*}
	for $\ell=1,\,2,\,3,\,\dots,\, 21$.
Figure~\ref{fig:Error_Plot_different_subsequences} illustrates that the observed {convergence} does not depend on a chosen subsequence of numerical solutions. Together with Table~\ref{tab1} it indicates that the sequence of numerical solutions obtained by the VFV method {is S-convergent for all regular summation methods and the limit does not depend on the specific sequence of numerical solutions}.

\begin{figure}
	\centering
	\includegraphics[width=0.5\linewidth]{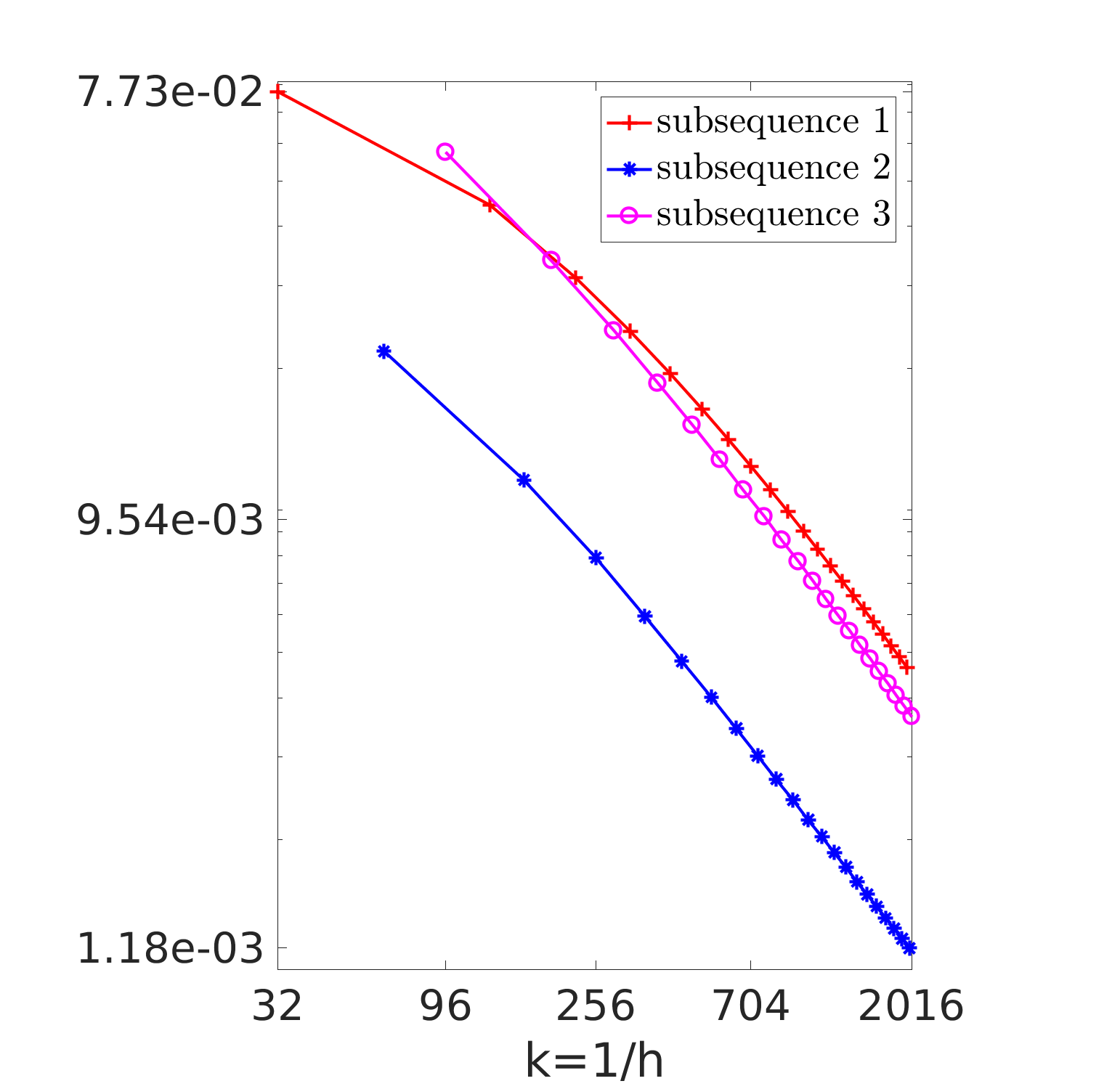}
	\caption{Convergence of different subsequences of Ces\`aro averages.}
	\label{fig:Error_Plot_different_subsequences}
\end{figure}

Figures~\ref{non_av_images}, 3 present numerical densities  computed by the VFV method  on  $k \times k$ meshes. Figure~\ref{av_images} shows the Ces\` aro averages of the density for various meshes.
Analogous pictures, not presented here, have been obtained also for other summation methods with the weight functions
$\omega_\text{quad}, \ \omega_\text{sin2}, \ \omega_\text{exp}$.
The first variance, that is another observable function, is presented for different meshes in Figure~\ref{var_images}.
More precisely, for the summation method with $\omega_\text{equal}$ the first variance is computed as
$\frac{1}{k} \sum_{n=1}^k \left | {\varrho_n} -  \frac{1}{k} \sum_{{m}=1}^k {\varrho_{m}} \right |. $
 In contrast to Figure~\ref{non_av_images} where no convergence of
single numerical solutions is observed,
Figures~\ref{av_images}, \ref{var_images} indicate the convergence for observables average and variance.

\begin{figure}
	\begin{subfigure}{0.24\linewidth}
		\centering
		\includegraphics[width=\linewidth]{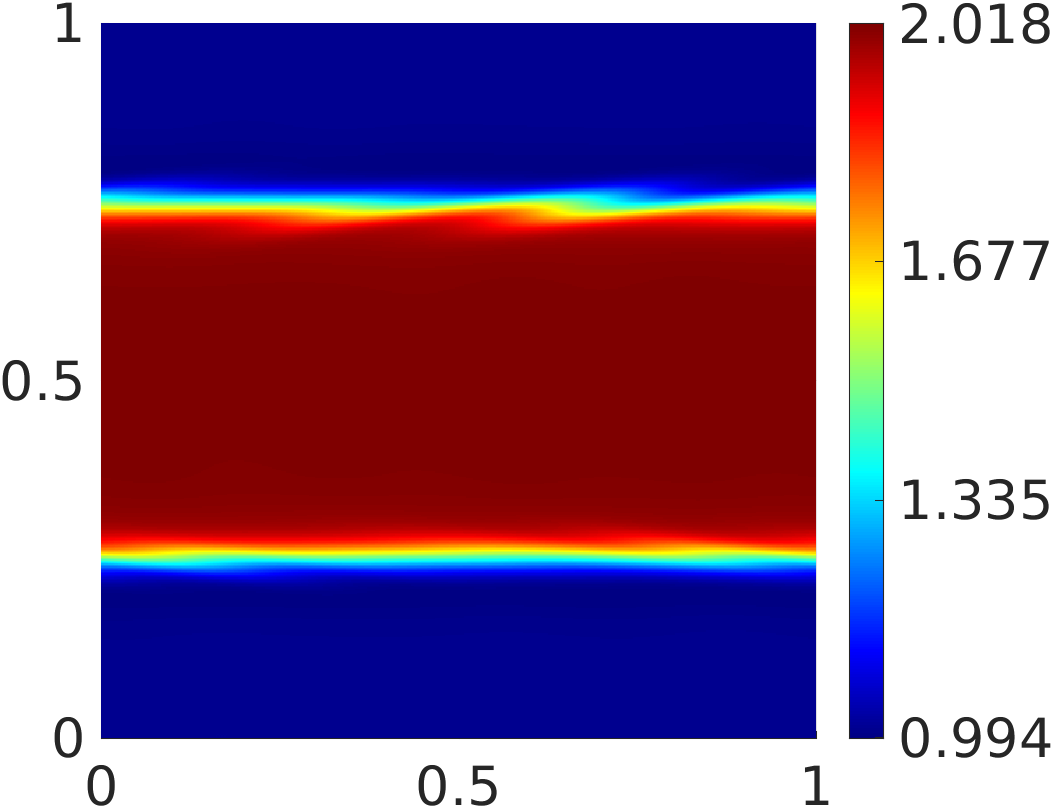}
		\caption{$k = 512$}
		\label{fig:SimulationN512}
	\end{subfigure}
	\hfill
	\begin{subfigure}{0.24\linewidth}
		\centering
		\includegraphics[width=\linewidth]{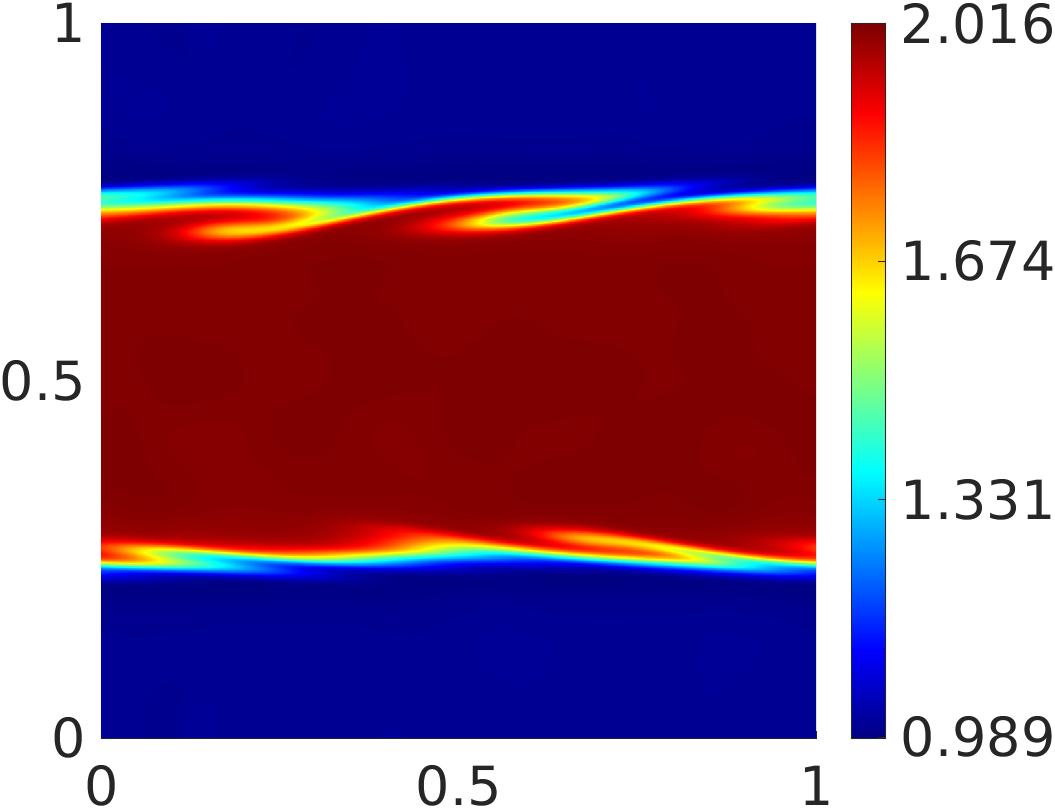}
		\caption{$k = 1024$}
		\label{fig:SimulationN1024}
	\end{subfigure}
	\hfill
	\begin{subfigure}{0.24\linewidth}
		\centering
		\includegraphics[width=\linewidth]{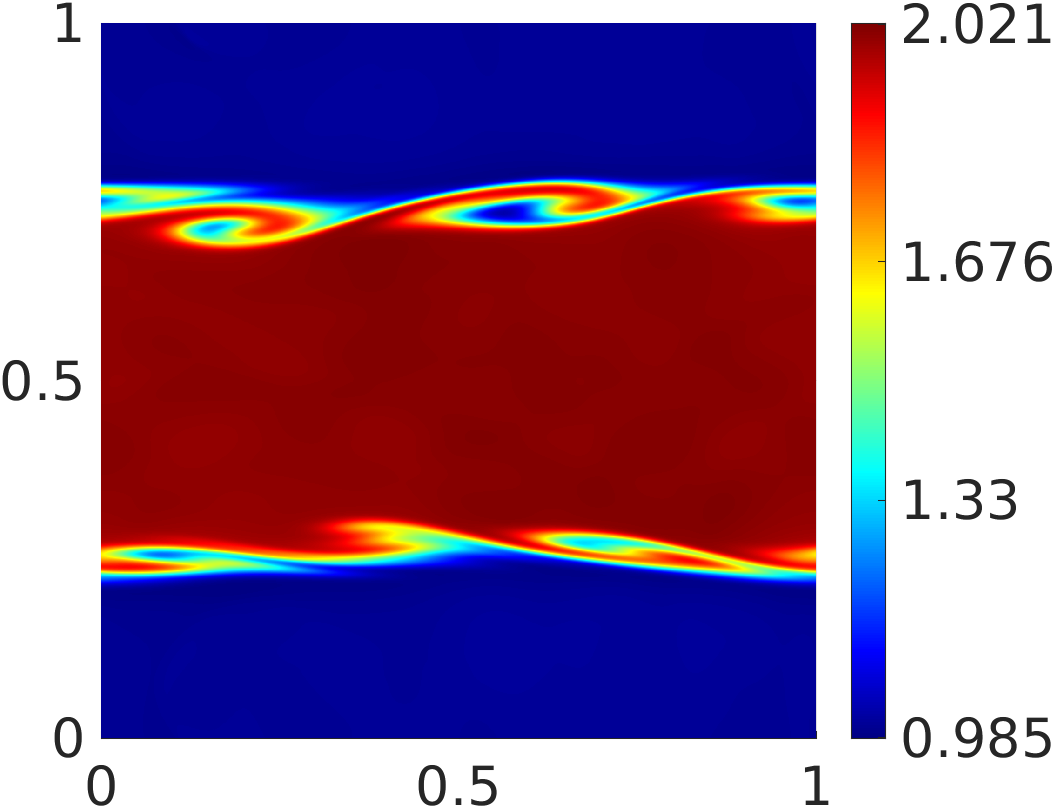}
		\caption{$k = 1536$}
		\label{fig:SimulationN1536}
	\end{subfigure}
	\hfill
	\begin{subfigure}{0.24\linewidth}
		\centering
		\includegraphics[width=\linewidth]{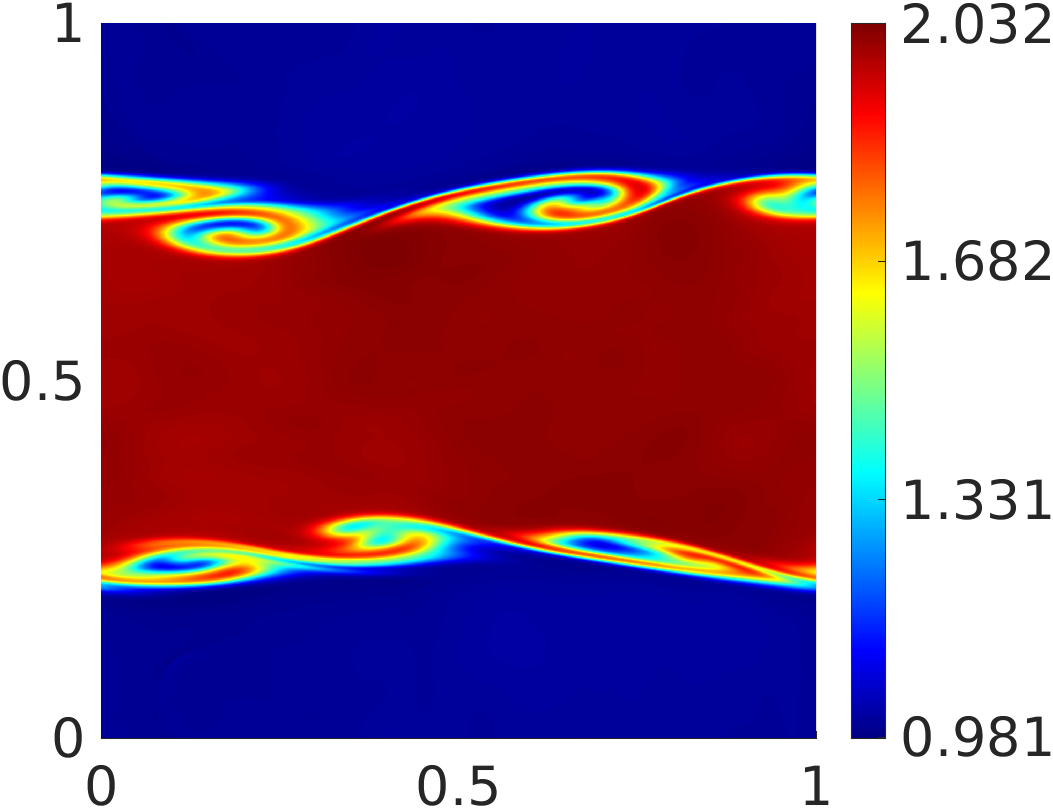}
		\caption{$k = 2048$}
		\label{fig:SimulationN2048}
\end{subfigure}
\caption{Density computed by the VFV scheme at $T= 2$ for the Kelvin-Helmholtz problem on a mesh with $k\times k$ cells.}
\label{non_av_images}
\end{figure}

\begin{figure}
\label{fig3}
	\begin{subfigure}{0.49\linewidth}
		\centering
		\includegraphics[width=\linewidth]{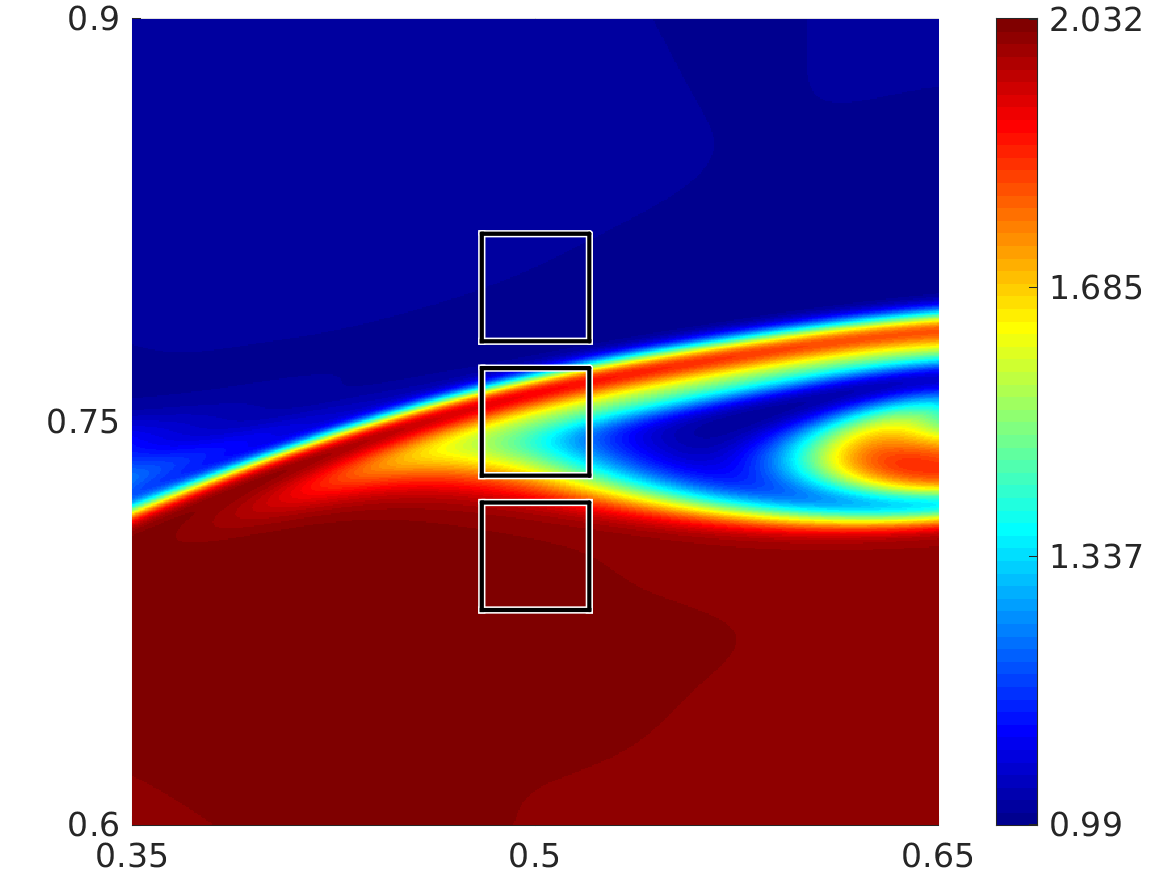}
		\caption{$k=2048$, zoomed in}
		\label{fig:SimulationN2048zoomedIn}
	\end{subfigure}
	\hfill
	\begin{subfigure}{0.49\linewidth}
		\centering
		\includegraphics[width=\linewidth]{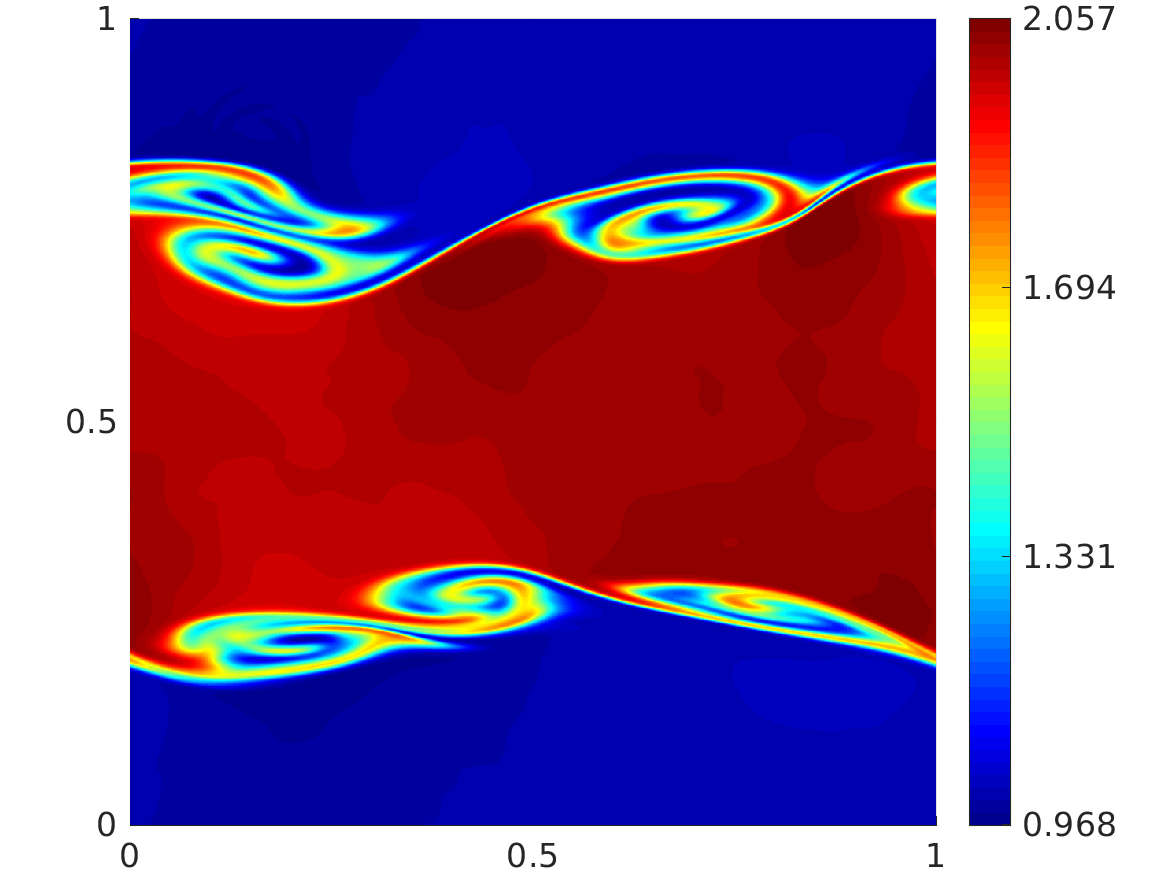}
		\caption{$k=3072$}
		\label{fig:SimulationN3072}
	\end{subfigure}
	\caption{Density computed by the VFV scheme at $T= 2$ for the Kelvin-Helmholtz problem on a mesh with $k\times k$ cells.}
\end{figure}

\begin{figure}
	\begin{subfigure}{0.24\linewidth}
		\centering
		\includegraphics[width=\linewidth]{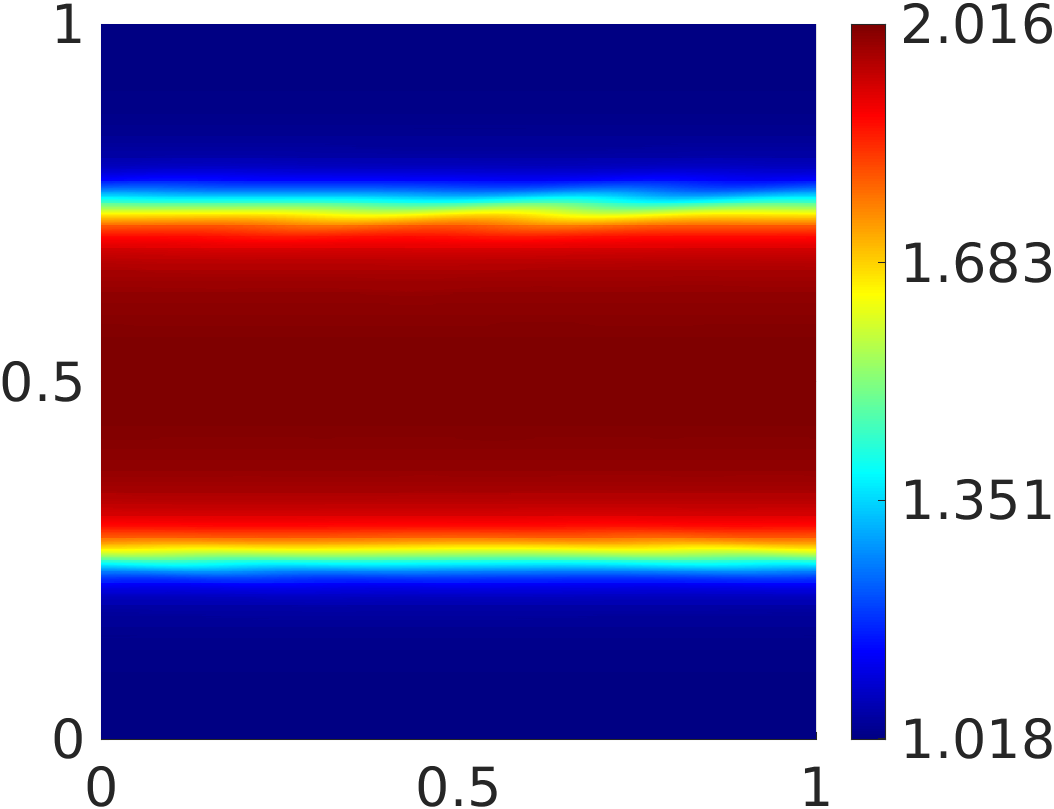}
		\caption{up to $k = 512$}
		\label{fig:Average_upTo512}
	\end{subfigure}
		\hfill
	\begin{subfigure}{0.24\linewidth}
		\centering
		\includegraphics[width=\linewidth]{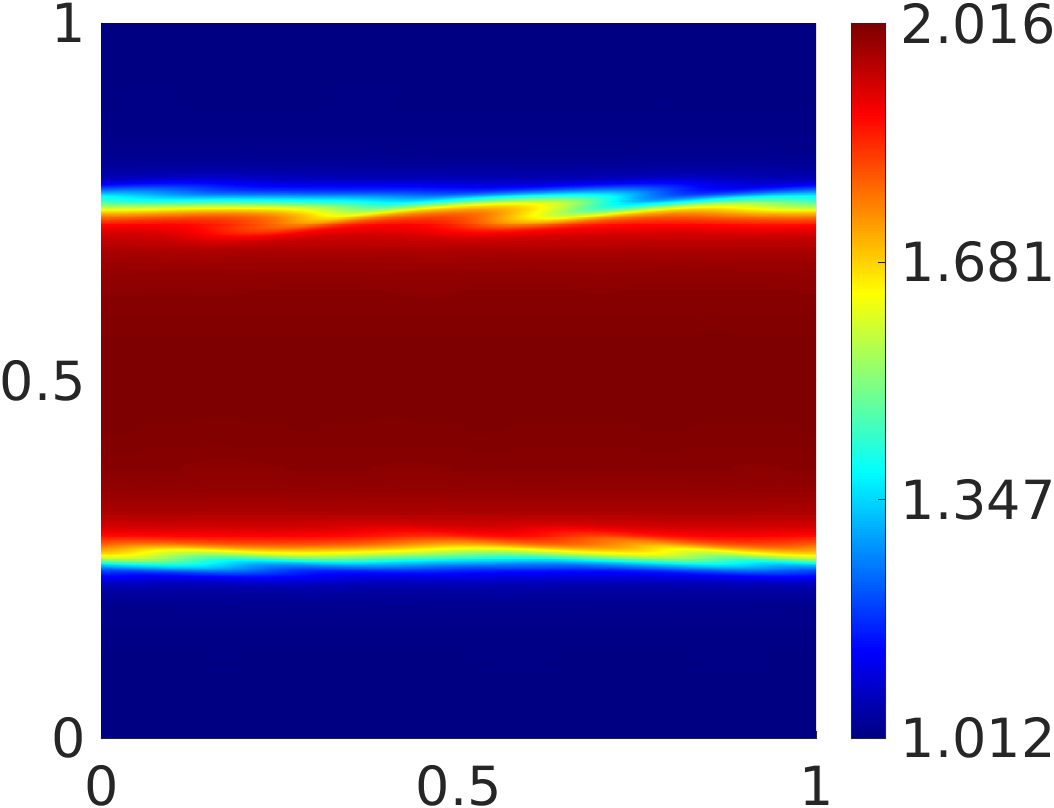}
		\caption{up to $k = 1024$}
		\label{fig:Average_upTo1024}
	\end{subfigure}
	\hfill
	\begin{subfigure}{0.24\linewidth}
		\centering
		\includegraphics[width=\linewidth]{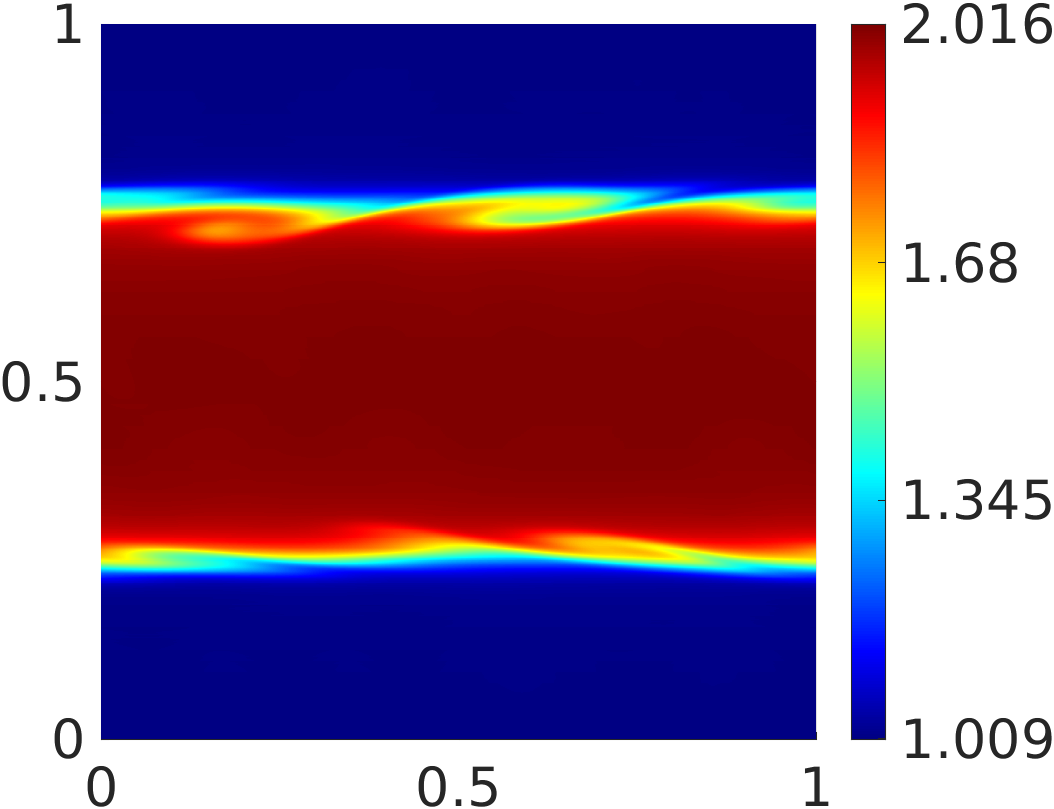}
		\caption{up to $k = 1536$}
		\label{fig:Average_upTo1536}
	\end{subfigure}
	\hfill
	\begin{subfigure}{0.24\linewidth}
		\centering
		\includegraphics[width=\linewidth]{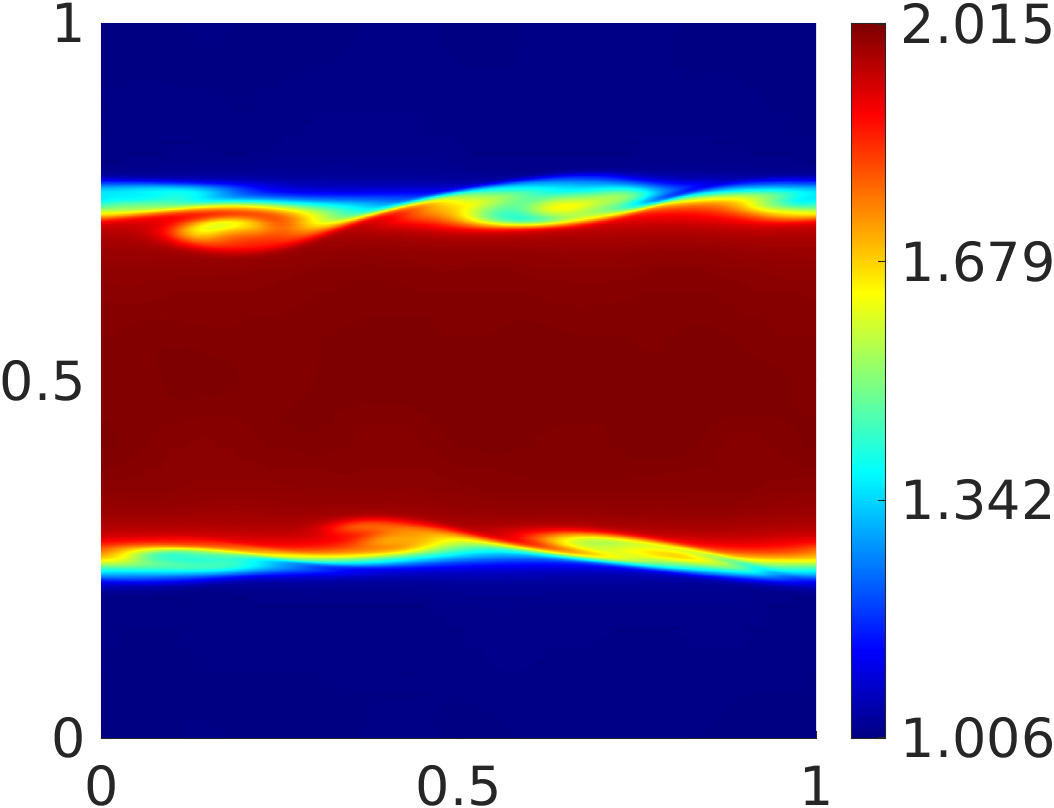}
		\caption{up to $k = 2048$}
		\label{fig:Average_upTo2048}
	\end{subfigure}
	\caption{Ces\`aro averages of the density computed by the VFV method on meshes with $j\times j$ cells, $j= 32,\;64,\;96,\dots,\;k$, for the Kelvin-Helmholtz problem.}
	\label{av_images}
\end{figure}

\begin{figure}
	\begin{subfigure}{0.24\linewidth}
		\centering
		\includegraphics[width=\linewidth]{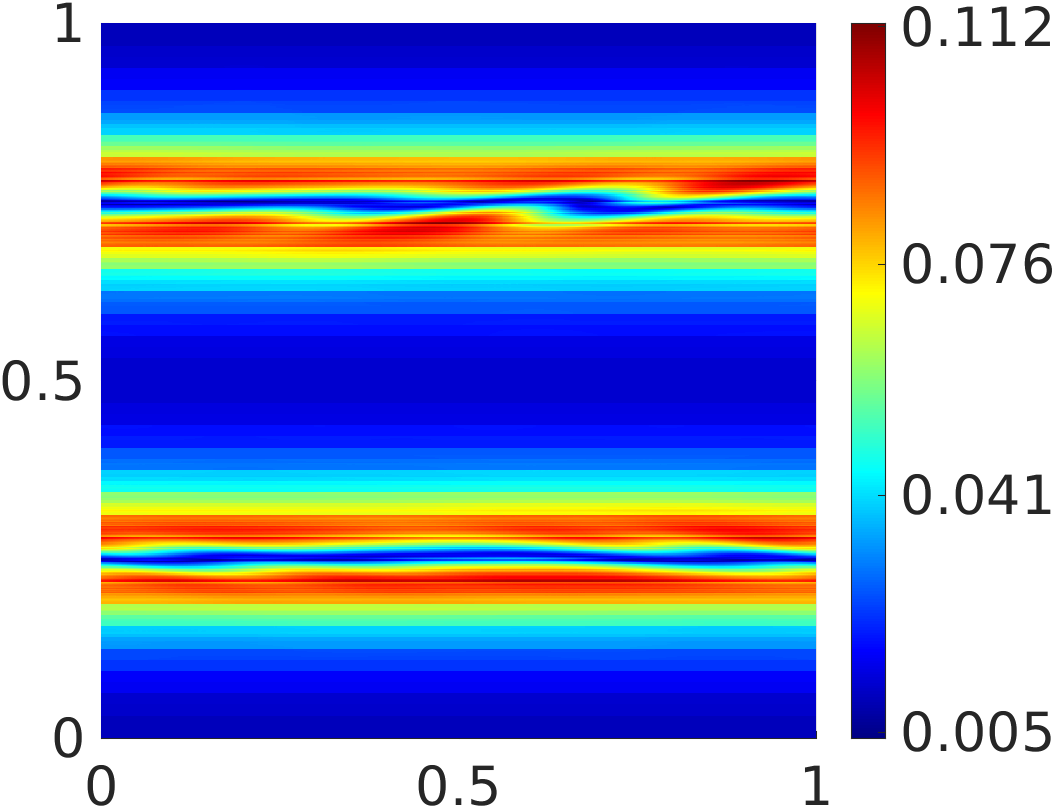}
		\caption{up to $k = 512$}
		\label{fig:firstVar_512}
	\end{subfigure}
	\hfill
	\begin{subfigure}{0.24\linewidth}
		\centering
		\includegraphics[width=\linewidth]{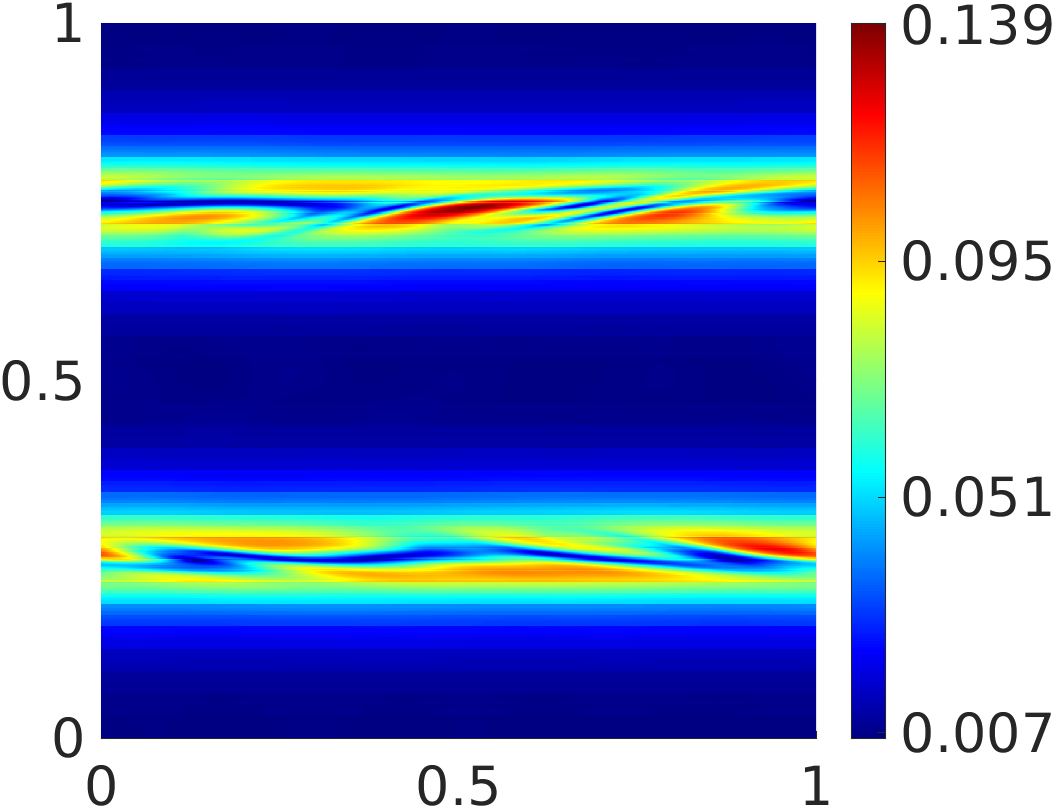}
		\caption{up to $k = 1024$}
		\label{fig:firstVar_1024}
	\end{subfigure}
	\hfill
	\begin{subfigure}{0.24\linewidth}
		\centering
		\includegraphics[width=\linewidth]{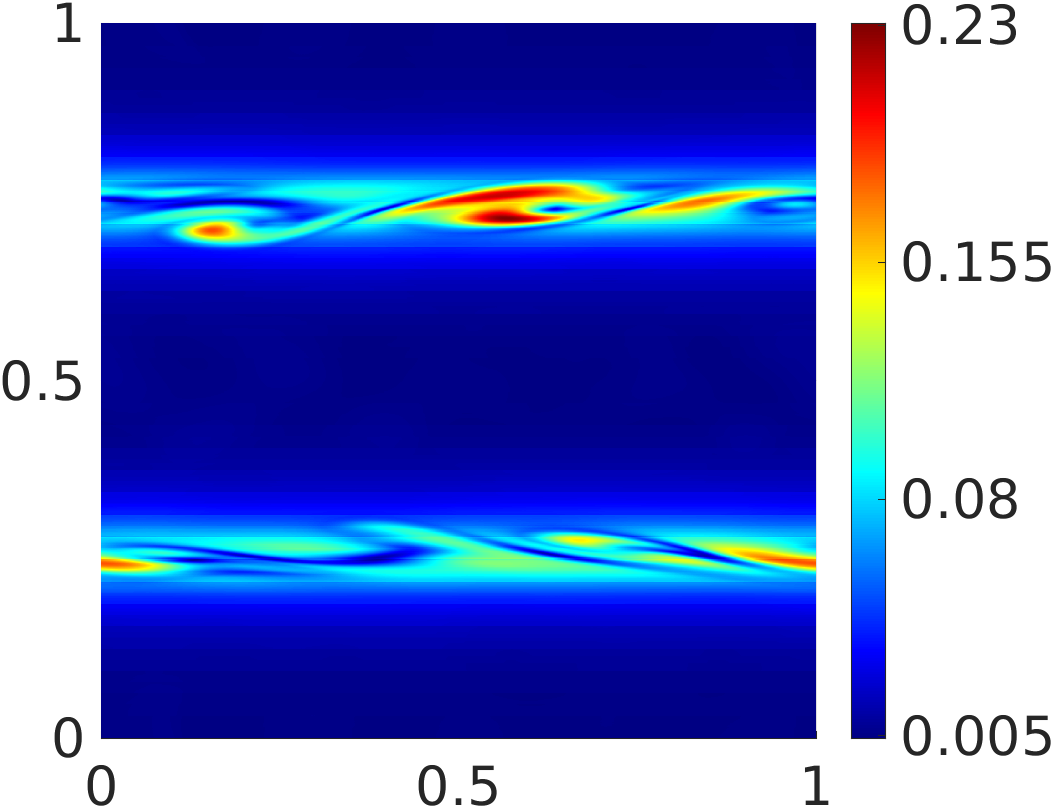}
		\caption{up to $k = 1536$}
		\label{fig:firstVar_1536}
	\end{subfigure}
	\hfill
	\begin{subfigure}{0.24\linewidth}
		\centering
		\includegraphics[width=\linewidth]{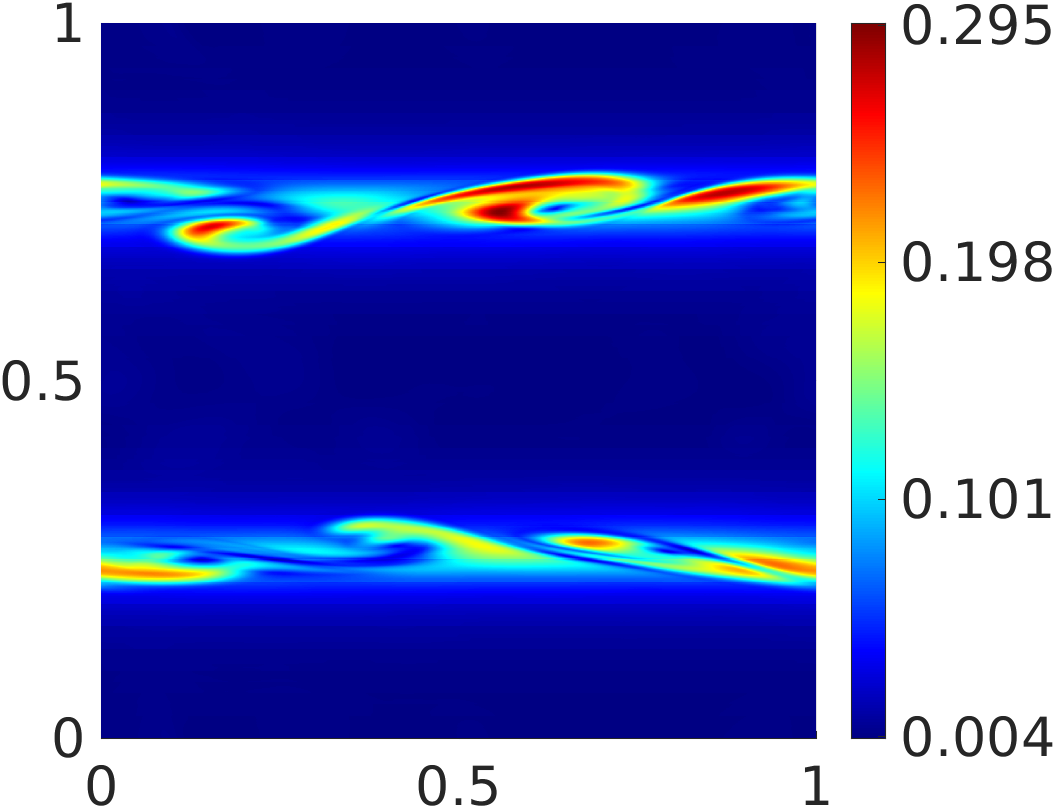}
		\caption{up to $k = 2048$}
		\label{fig:firstVar_2048}
	\end{subfigure}
	\caption{First variance of the density computed by the VFV method on meshes with $j\times j$ cells, $j= 32,\;64,\;96,\dots,\;k$, for the Kelvin-Helmholtz problem.}
	\label{var_images}
\end{figure}

\begin{table}
	\caption{Convergence study for different summation methods in the $L^1$-norm in space and the Wasserstein distance for measures
$ \sum_{n=1}^N \omega(n/N)\delta_{\vr_n(T,x)}/\sum_{m=1}^N \omega(m/N)$ at time $T = 2.$
}
	\label{tab3}
	\begin{center}
		\begin{tabular}{|cV{3}c|cV{3}c|cV{3}c|cV{3}c|c|}
		\hline
		$k$&
		\multicolumn{2}{cV{3}}{$\omega_{\text{equal}}$}& \multicolumn{2}{cV{3}}{$\omega_{\text{quad}}$}& \multicolumn{2}{cV{3}}{$\omega_{\text{sin2}}$}& \multicolumn{2}{c|}{$\omega_{\text{exp}}$}\\
		\cline{2-9}
		(up to)& error   & order& error   & order& error   & order& error   & order\\
		\hline
48	 & 1.27e-01 	 & - 	 & 1.79e-01 	 & - 	 & 1.61e-01 	 & - 	 & 1.75e-01 	 & - 	 \\
64	 & 1.05e-01 	 & 0.66 	 & 1.54e-01 	 & 0.52 	 & 1.22e-01 	 & 0.96 	 & 1.50e-01 	 & 0.54 	 \\
96	 & 8.49e-02 	 & 0.52 	 & 1.35e-01 	 & 0.32 	 & 9.51e-02 	 & 0.61 	 & 1.30e-01 	 & 0.35 	 \\
128	 & 6.90e-02 	 & 0.72 	 & 1.18e-01 	 & 0.47 	 & 7.51e-02 	 & 0.82 	 & 1.16e-01 	 & 0.40 	 \\
192	 & 5.49e-02 	 & 0.56 	 & 1.03e-01 	 & 0.34 	 & 6.12e-02 	 & 0.50 	 & 1.03e-01 	 & 0.29 	 \\
256	 & 4.33e-02 	 & 0.83 	 & 8.92e-02 	 & 0.50 	 & 5.02e-02 	 & 0.69 	 & 9.23e-02 	 & 0.38 	 \\
384	 & 3.31e-02 	 & 0.66 	 & 7.65e-02 	 & 0.38 	 & 4.11e-02 	 & 0.49 	 & 8.16e-02 	 & 0.30 	 \\
512	 & 2.47e-02 	 & 1.02 	 & 6.46e-02 	 & 0.59 	 & 3.23e-02 	 & 0.84 	 & 7.07e-02 	 & 0.50 	 \\
768	 & 1.77e-02 	 & 0.82 	 & 5.34e-02 	 & 0.47 	 & 2.37e-02 	 & 0.76 	 & 5.61e-02 	 & 0.57 	 \\
1024	 & 1.24e-02 	 & 1.24 	 & 4.37e-02 	 & 0.70 	 & 1.51e-02 	 & 1.57 	 & 3.32e-02 	 & 1.82 	 \\
\hline
	\end{tabular}
	\end{center}
\end{table}

To illustrate the probabilistic nature of the limiting solution
we present in Figures~\ref{hist_bot}, \ref{hist_mid} and \ref{hist_top} the approximations of the probability density at time $T=2$ of the Ces\` aro averages computed on meshes with $j \times j$ cells, $j \in \{32m\Big|\;m \in \mathds{N},\; 1 \leq m \leq k\}$, averaged in space on the domains $(0.48,0.52)\times(0.68,0.72)$, $(0.48,0.52)\times(0.73,0.77)$  and $(0.48,0.52)\times(0.78,0.82)$, respectively. These three regions are depicted in Figure~\ref{fig:SimulationN2048zoomedIn}. Note that these regions are chosen in such a way that they are either completely below, right on or completely above of the initial upper interface $J_2$.
These figures yields further evidence of S-convergence as documented in Table~\ref{tab3}. The latter presents S-convergence with respect to various summation methods using the weighted functions $\omega_\text{equal},\ \omega_\text{quad}$, $\omega_\text{sin2}$ and $\omega_\text{exp}$.
In particular, we compute  the experimental convergence of the weighted averages $ \sum_{n=1}^N \omega(n/N)\delta_{\vr_n(T,x)}/\sum_{m=1}^N \omega(m/N)$ in the Wasserstein distance.   The corresponding reference solutions were obtained using meshes with
$k = 32, 48, 64, 96, 128,192,256,384,512,768,1024, 1536,3072$ and the respective weight functions~$\omega.$

\begin{figure}
	\begin{subfigure}{0.19\linewidth}
		\centering
		\includegraphics[width=\linewidth]{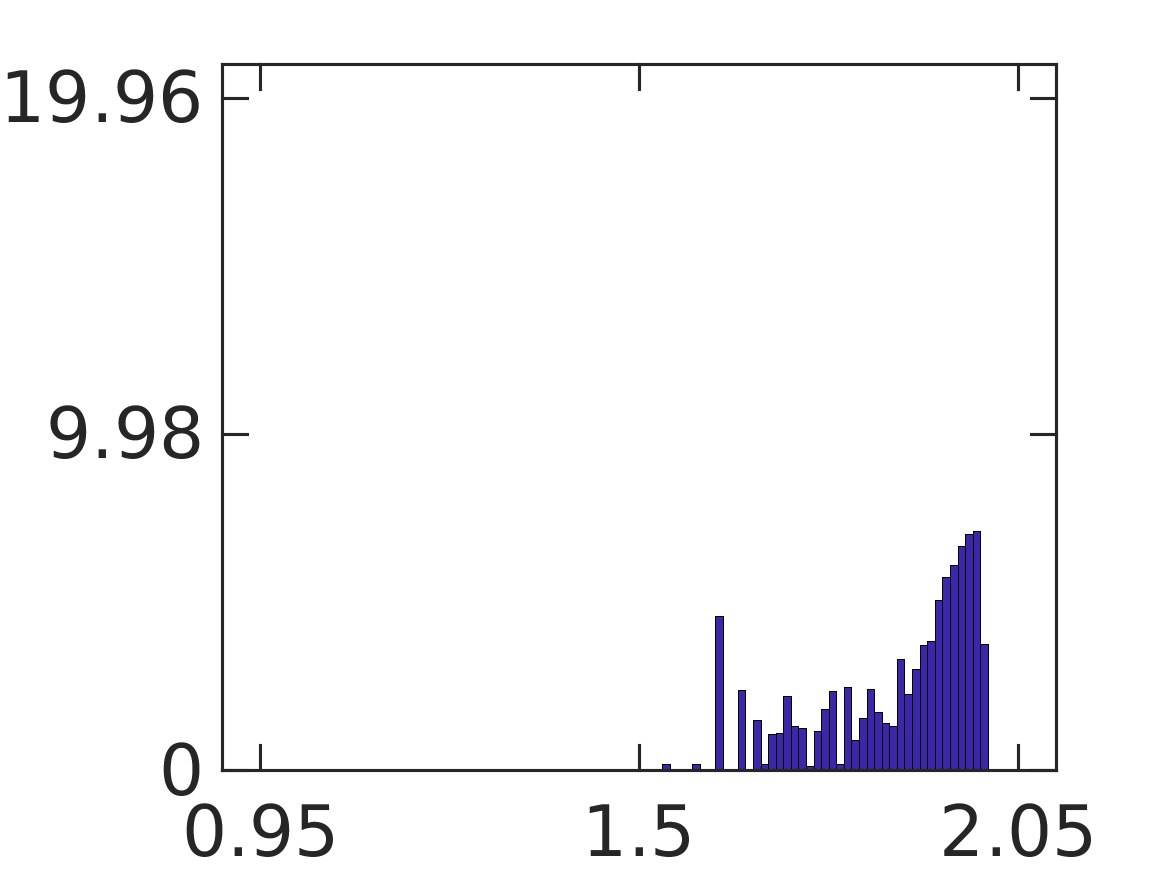}
		\caption{up to $k = 512$}
		\label{fig:histbot_upTo512}
	\end{subfigure}
	\hfill
	\begin{subfigure}{0.19\linewidth}
		\centering
		\includegraphics[width=\linewidth]{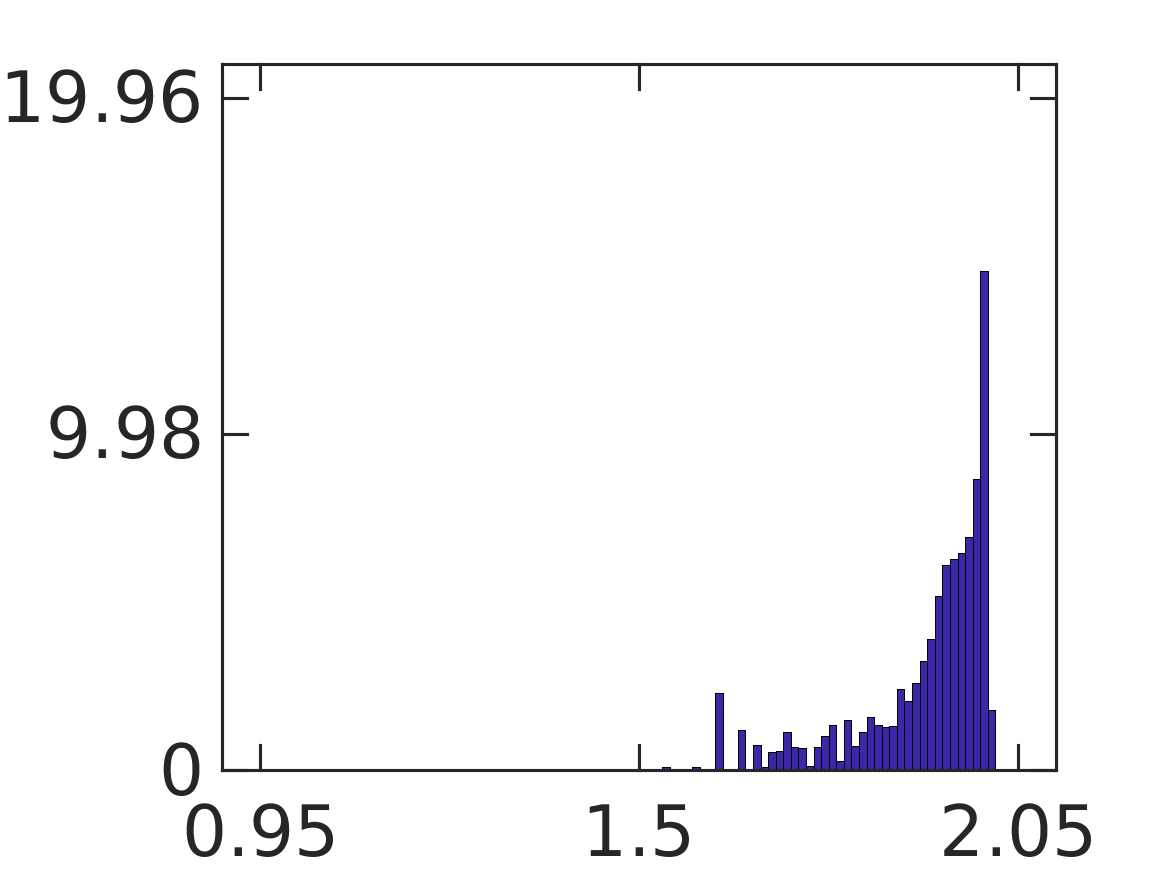}
		\caption{up to $k = 1024$}
		\label{fig:histbot_upTo1024}
	\end{subfigure}
	\hfill
	\begin{subfigure}{0.19\linewidth}
		\centering
		\includegraphics[width=\linewidth]{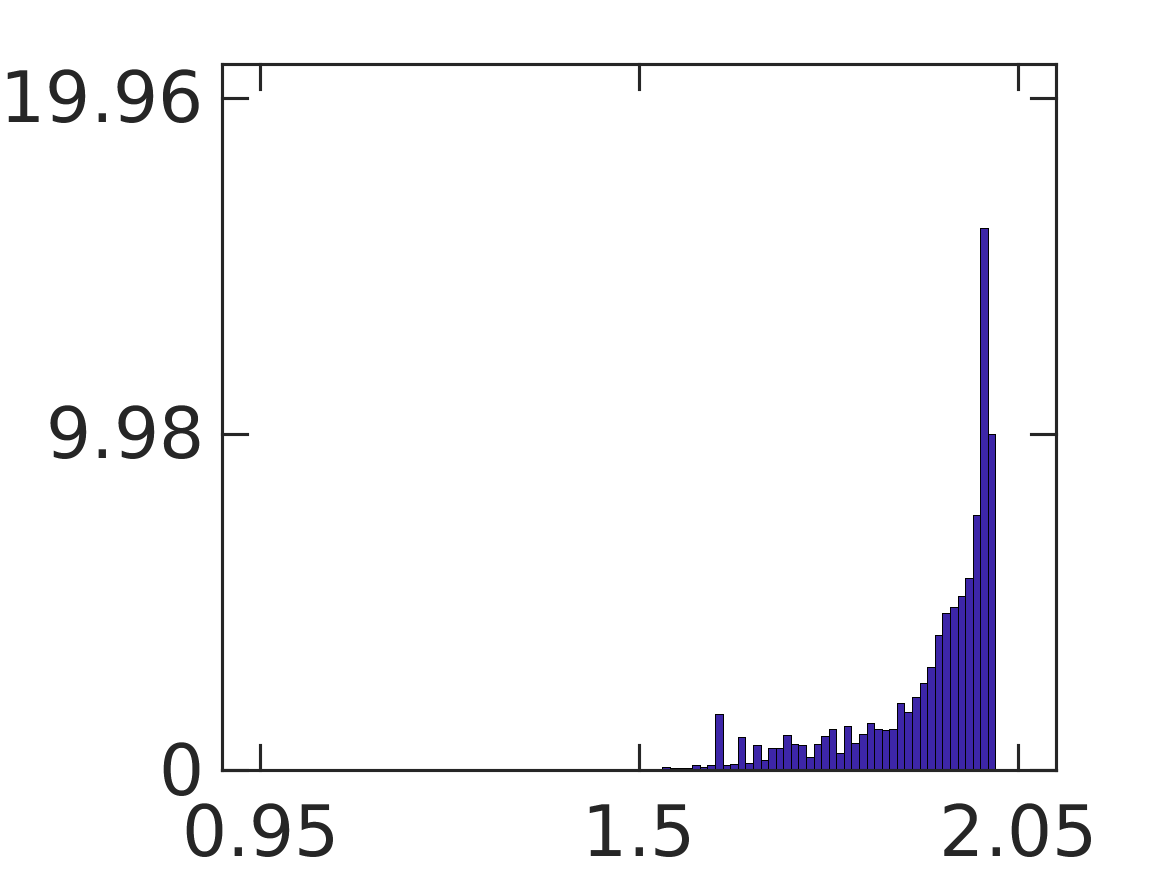}
		\caption{up to $k = 1536$}
		\label{fig:histbot_upTo1536}
	\end{subfigure}
	\hfill
	\begin{subfigure}{0.19\linewidth}
		\centering
		\includegraphics[width=\linewidth]{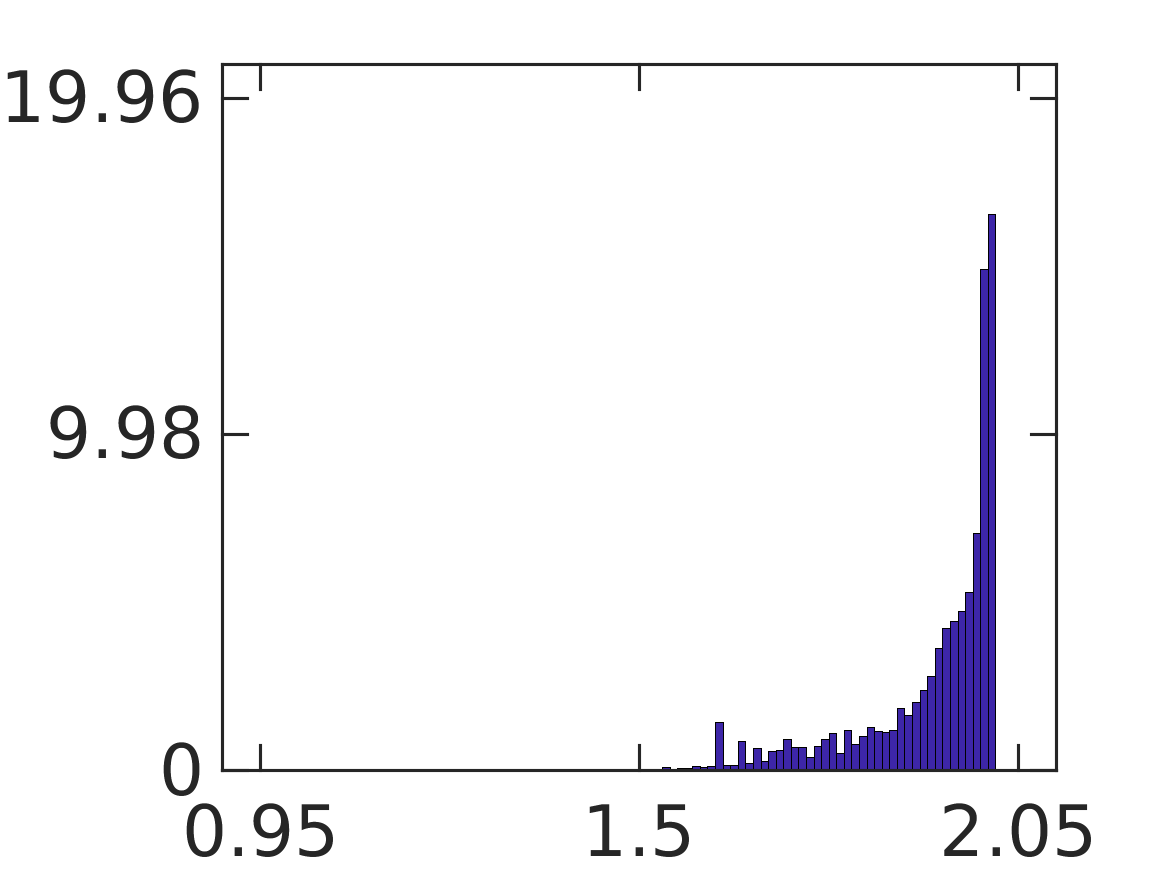}
		\caption{up to $k = 1792$}
		\label{fig:histbot_upTo1792}
	\end{subfigure}
	\hfill
	\begin{subfigure}{0.19\linewidth}
		\centering
		\includegraphics[width=\linewidth]{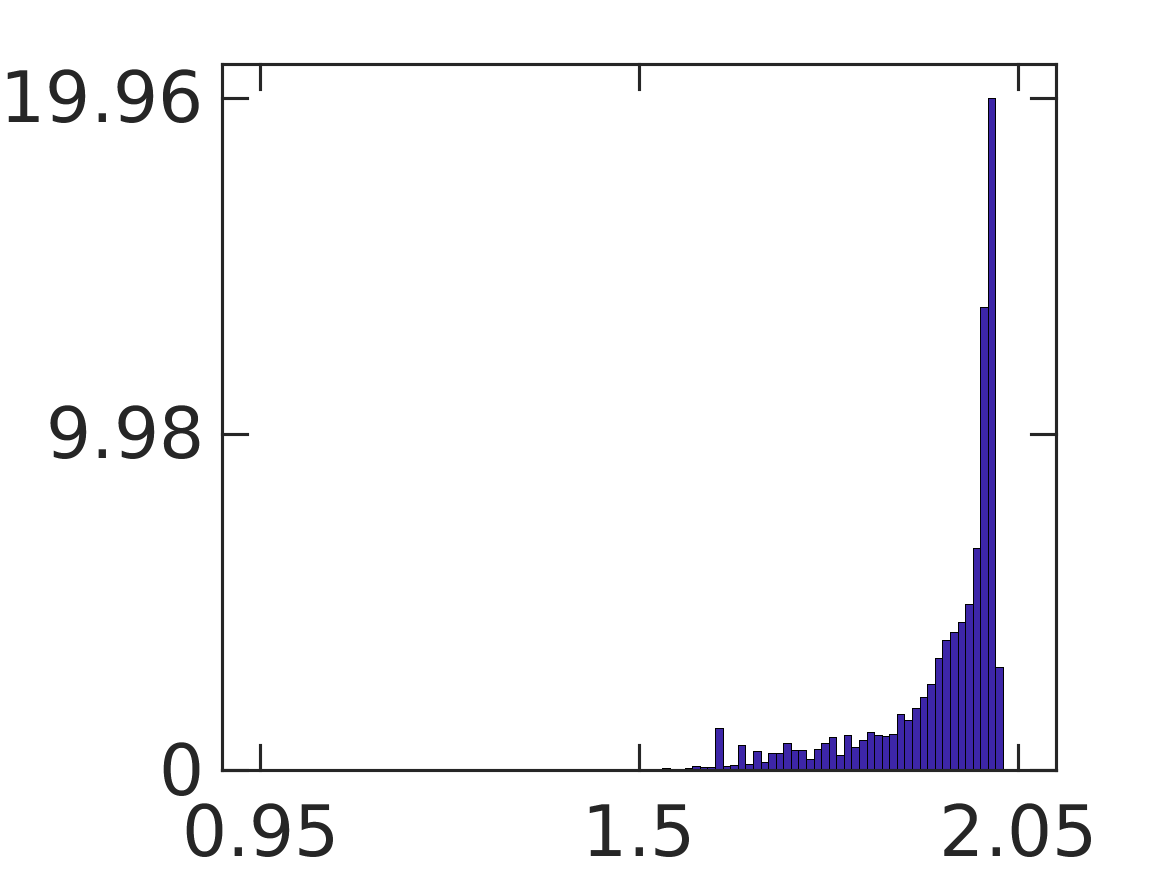}
		\caption{up to $k = 2048$}
		\label{fig:histbot_upTo2048}
	\end{subfigure}
	\caption{Probability density computed by the VFV method on meshes with $j\times j$ cells, $j= 32,\;64,\;96,\dots,\;k$, for the Kelvin-Helmholtz problem on the domain $(0.48,0.52)\times(0.68,0.72)$.}
	\label{hist_bot}
\end{figure}

\begin{figure}
	\begin{subfigure}{0.19\linewidth}
		\centering
		\includegraphics[width=\linewidth]{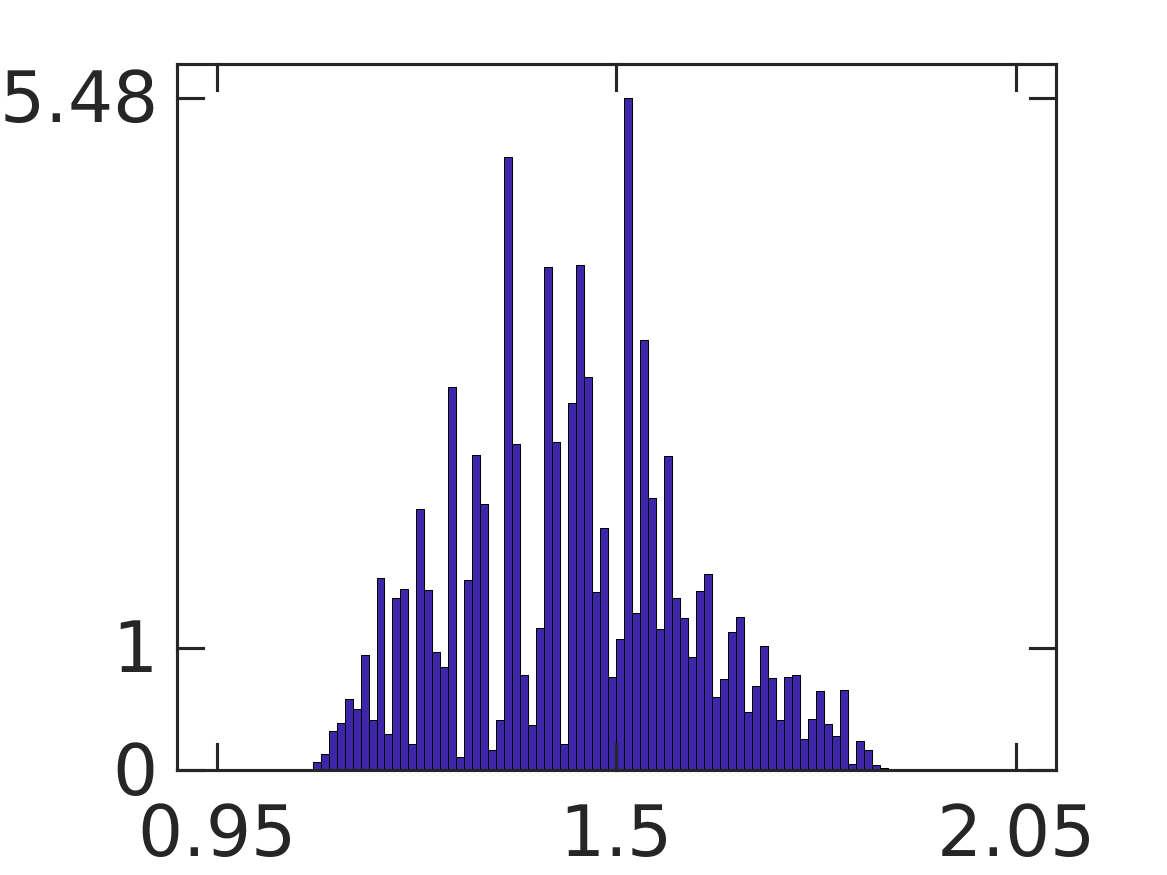}
		\caption{up to $k = 512$}
		\label{fig:histmid_upTo512}
	\end{subfigure}
	\hfill
	\begin{subfigure}{0.19\linewidth}
		\centering
		\includegraphics[width=\linewidth]{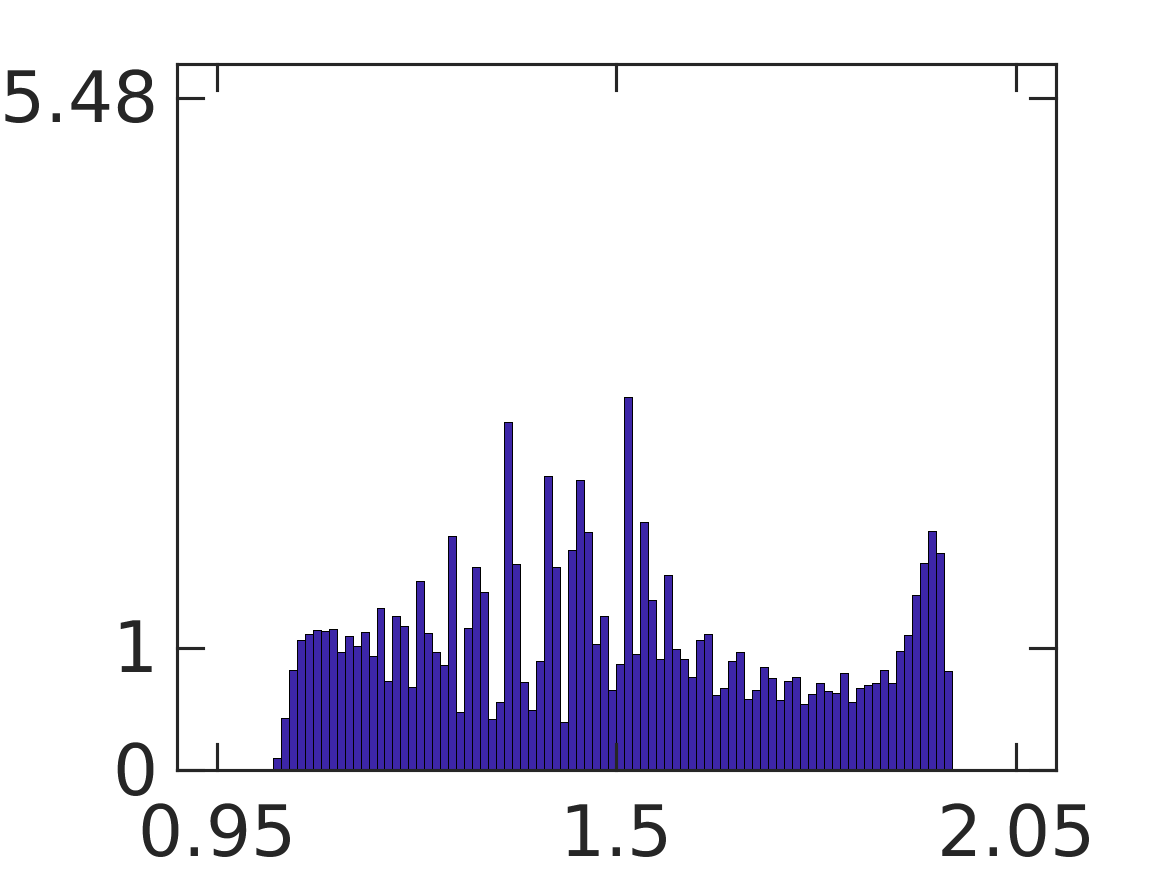}
		\caption{up to $k = 1024$}
		\label{fig:histmid_upTo1024}
	\end{subfigure}
	\hfill
	\begin{subfigure}{0.19\linewidth}
		\centering
		\includegraphics[width=\linewidth]{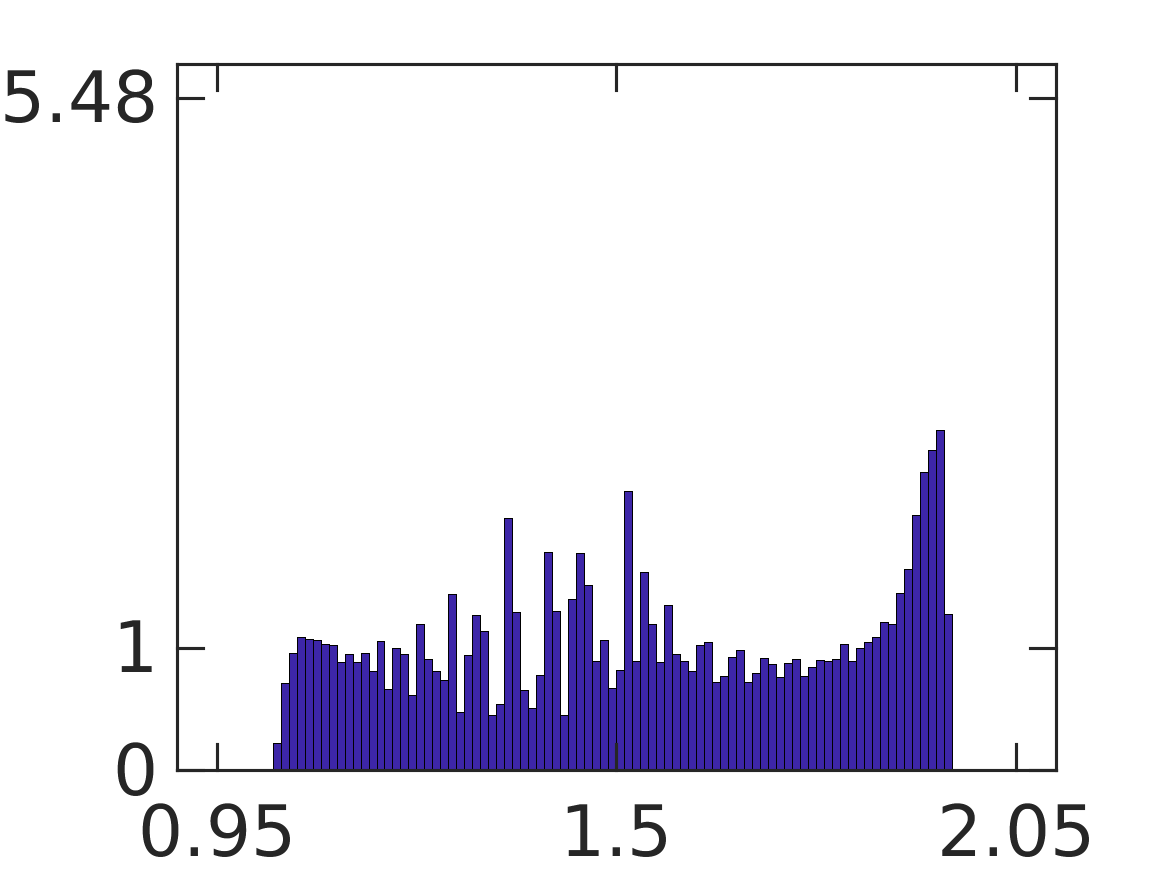}
		\caption{up to $k = 1536$}
		\label{fig:histmid_upTo1536}
	\end{subfigure}
	\hfill
	\begin{subfigure}{0.19\linewidth}
		\centering
		\includegraphics[width=\linewidth]{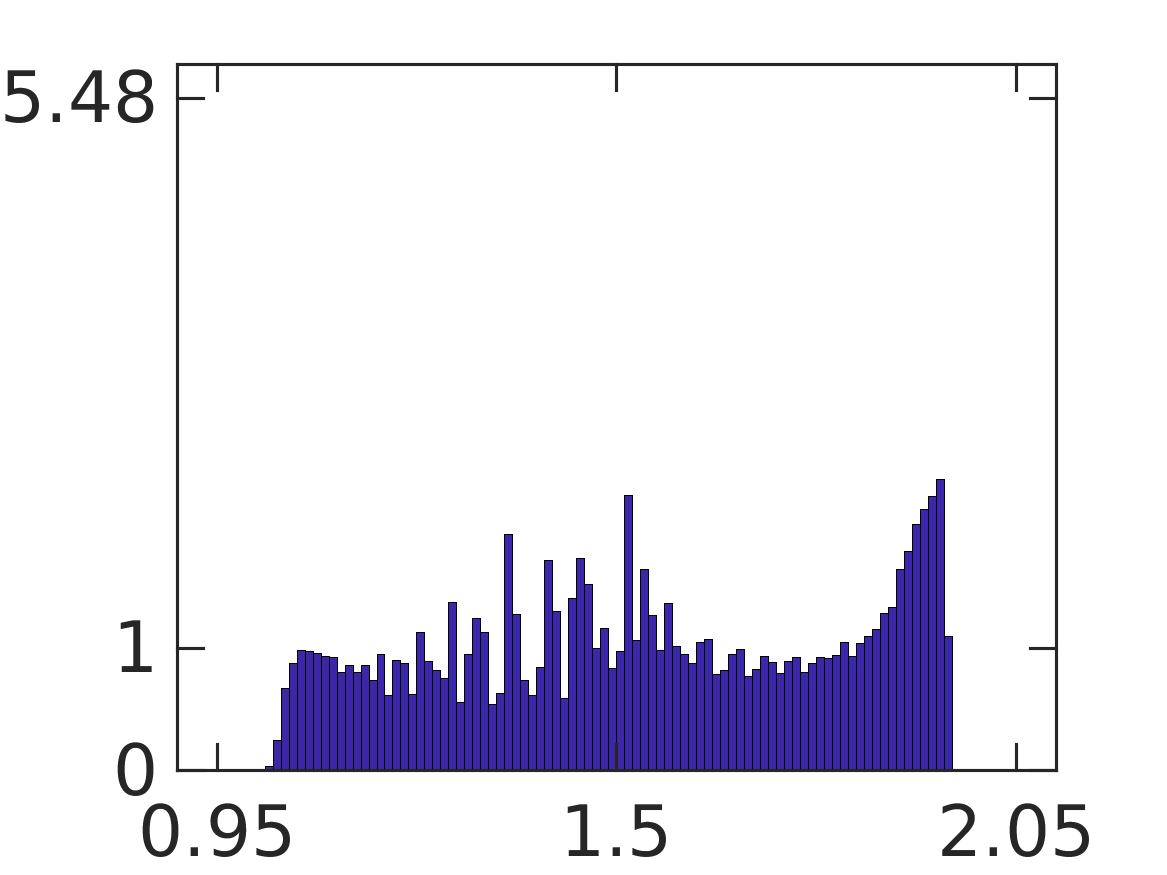}
		\caption{up to $k = 1792$}
		\label{fig:histmid_upTo1792}
	\end{subfigure}
	\hfill
	\begin{subfigure}{0.19\linewidth}
		\centering
		\includegraphics[width=\linewidth]{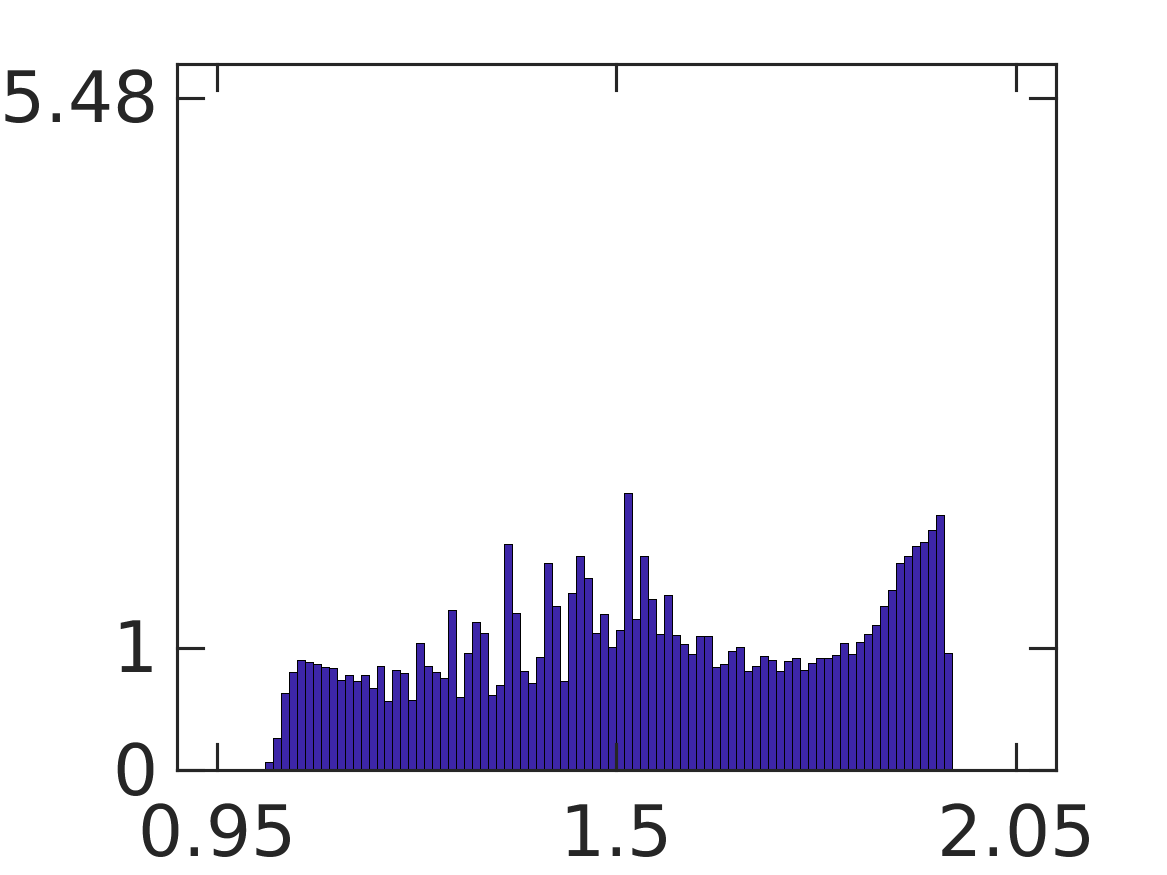}
		\caption{up to $k = 2048$}
		\label{fig:histmid_upTo2048}
	\end{subfigure}
	\caption{Probability density computed by the VFV method on meshes with $j\times j$ cells, $j= 32,\;64,\;96,\dots,\;k$, for the Kelvin-Helmholtz problem on the domain $(0.48,0.52)\times(0.73,{0.77})$.}
	\label{hist_mid}
\end{figure}

\begin{figure}
	\begin{subfigure}{0.19\linewidth}
		\centering
		\includegraphics[width=\linewidth]{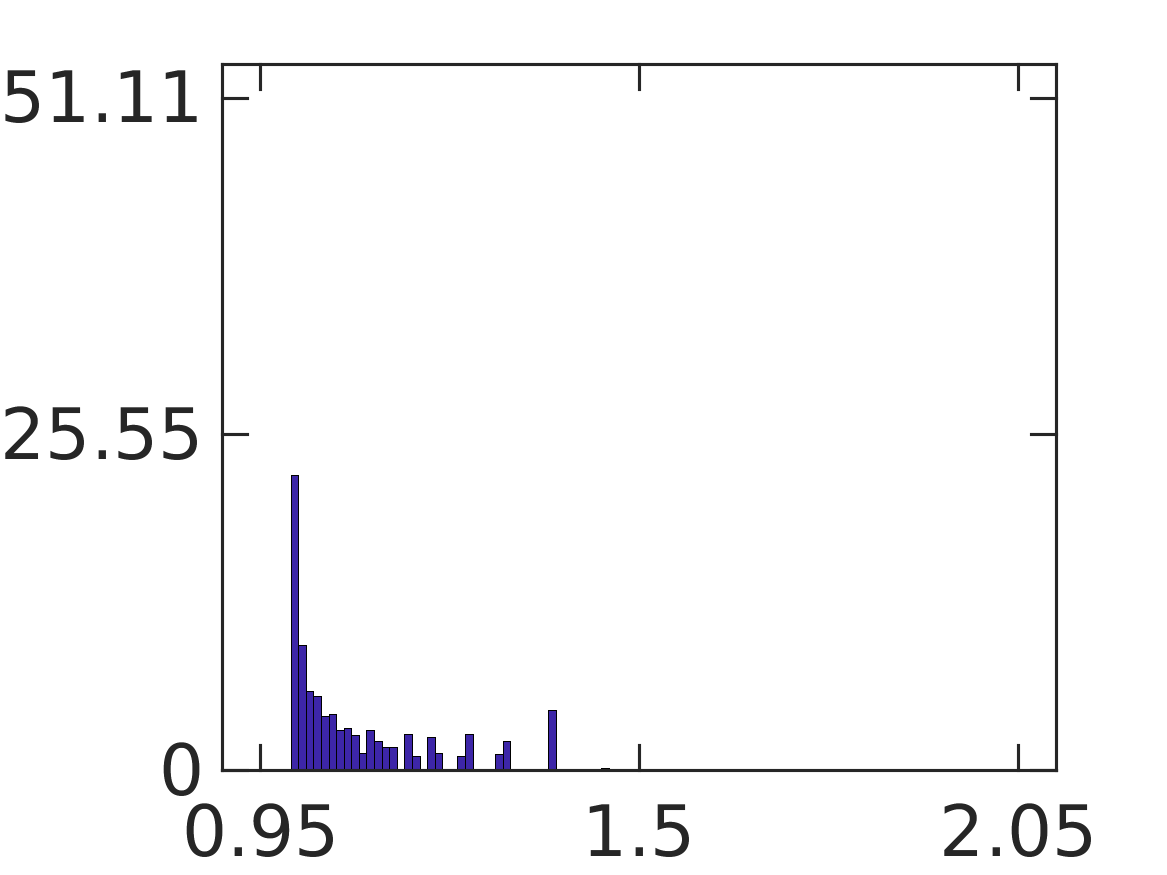}
		\caption{up to $k = 512$}
		\label{fig:histtop_upTo512}
	\end{subfigure}
	\hfill
	\begin{subfigure}{0.19\linewidth}
		\centering
		\includegraphics[width=\linewidth]{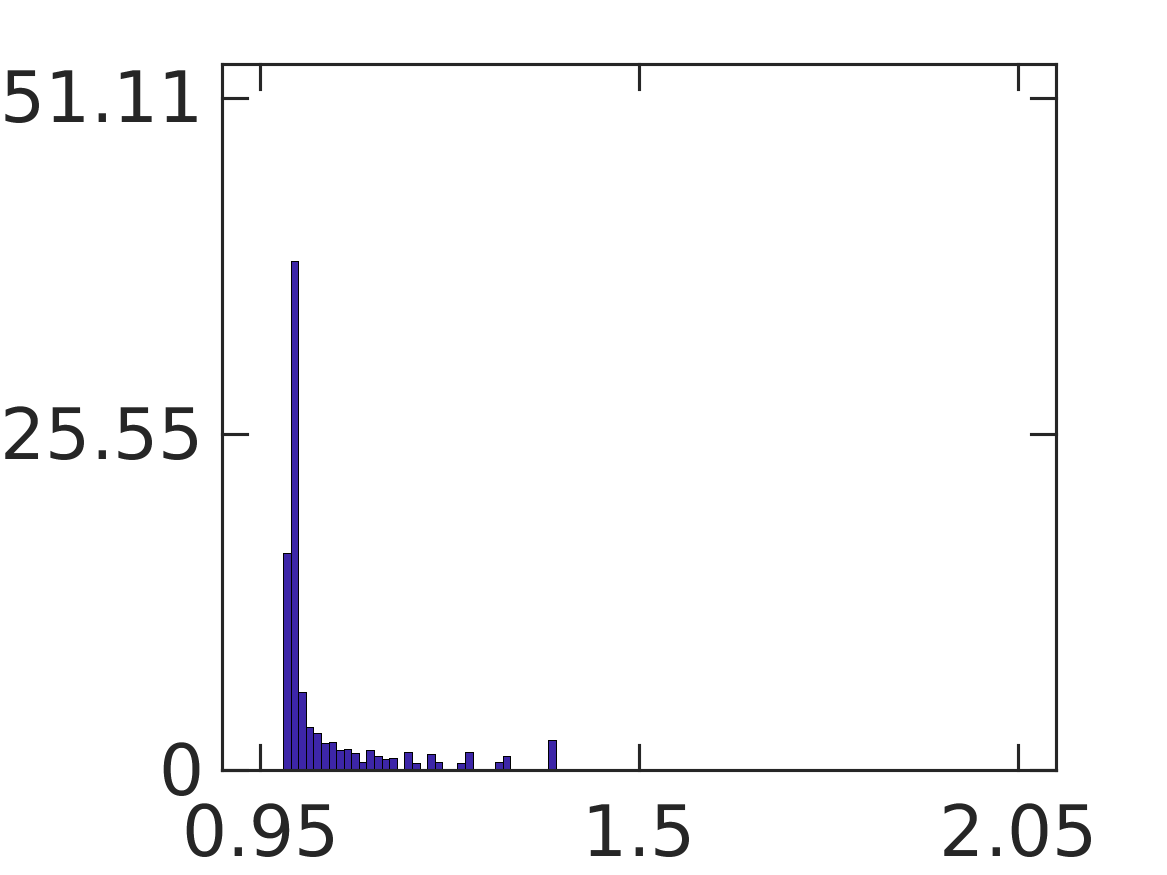}
		\caption{up to $k = 1024$}
		\label{fig:histtop_upTo1024}
	\end{subfigure}
	\hfill
	\begin{subfigure}{0.19\linewidth}
		\centering
		\includegraphics[width=\linewidth]{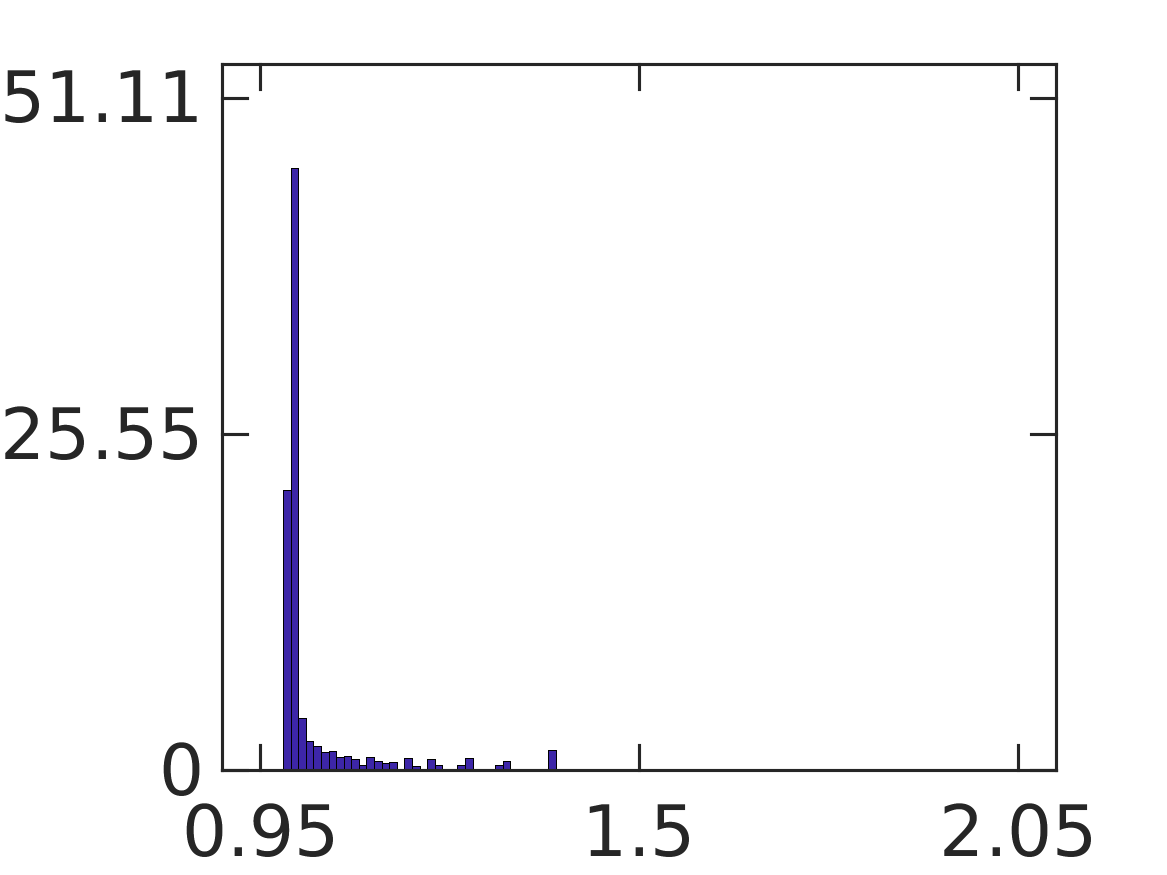}
		\caption{up to $k = 1536$}
		\label{fig:histtop_upTo1536}
	\end{subfigure}
	\hfill
	\begin{subfigure}{0.19\linewidth}
		\centering
		\includegraphics[width=\linewidth]{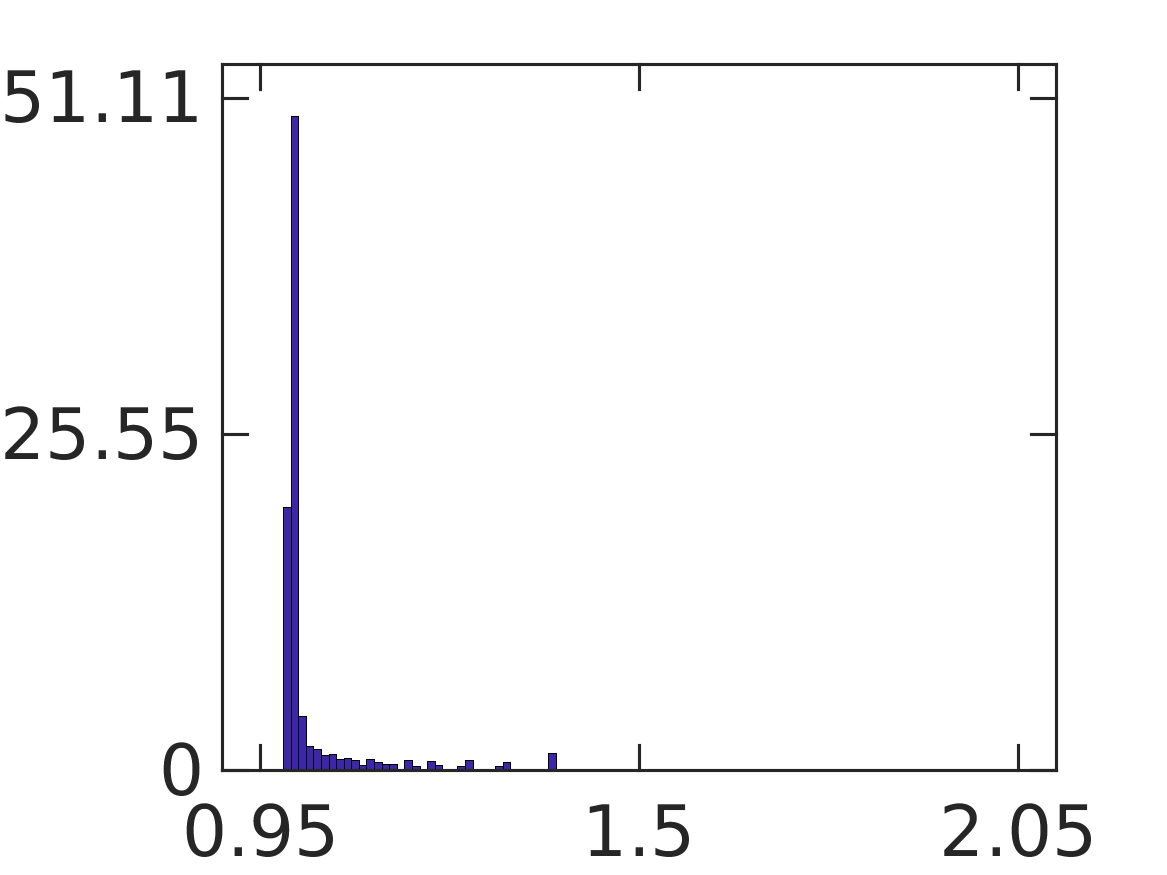}
		\caption{up to $k = 1792$}
		\label{fig:histtop_upTo1792}
	\end{subfigure}
	\hfill
	\begin{subfigure}{0.19\linewidth}
		\centering
		\includegraphics[width=\linewidth]{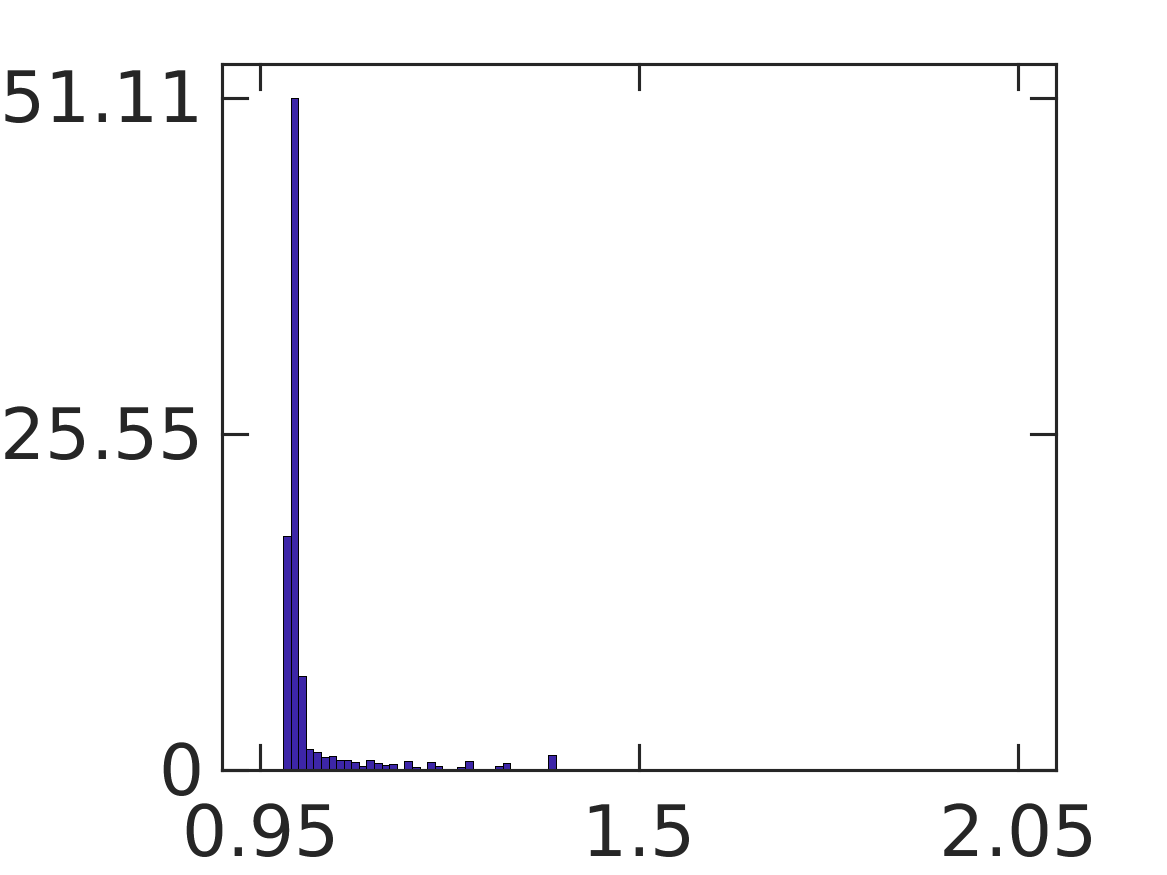}
		\caption{up to $k = 2048$}
		\label{fig:histtop_upTo2048}
	\end{subfigure}
	\caption{Probability density computed by the VFV method on meshes with $j\times j$ cells, $j= 32,\;64,\;96,\dots,\;k$, for the Kelvin-Helmholtz problem on the domain $(0.48,0.52)\times(0.78,0.82)$.}
	\label{hist_top}
\end{figure}

\section{Kolmogorov K41 hypothesis}
\label{K}

The celebrated Kolmogorov K41 hypothesis for incompressible flow has been extrapolated to the compressible setting by Chen and Glimm
\cite{CheGli}. In particular, it yields compactness of any family of (weak) solutions $[\vr_n, \vm_n]_{n=1}^\infty$ of the Navier--Stokes system in the zero viscosity regime $\mu_n \searrow 0$, $\lambda_n \to 0$. This {yields} strong convergence
\[
\vr_{n_k} \to \vr \ \mbox{in}\ L^1((0,T) \times \Td),\
\vm_{n_k} \to \vm \ \mbox{in}\ L^1((0,T) \times \Td; \mathds{R}^d)
\]
at least for a suitable subsequence. Thus if $[\vr, \vm]$, or at least $\vm$, are observables in the sense of Definition
\ref{PD4}, the limit is
the same for any subsequence. In other words, there is no need of subsequence and the convergence is strong unconditionally.
We call this hypothesis (KH). It is easy to check that (KH) yields uniqueness of the viscosity solution,
\[
{\rm (KH)}\ \Rightarrow  \ \mathcal{V}(t,x) = \delta_{[\vr, \vm](t,x)}.
\]
We remark that
although the convergence of the vanishing viscosity solutions is strong, the limit $[\vr, \vm]$ is not necessarily
a weak solution of the Euler system due to possible concentrations.

Now, an easy adaptation of Theorem \ref{NT1} yields:

\begin{Theorem} \label{KT1}

Under the hypothesis {\rm (KH)} let  $[\rN,\ \mN]$  denote the numerical solution obtained from the VFV method  with the
initial data $[\vr_{0,n}, \vm_{0,n}]$ a regular approximation of $[\vr_0, \vm_0]$, the artificial viscosity $\mu_n \searrow 0$, and the numerical step $h$, where
$[h, \mu_n, 0] \in \Ov{\mathcal{R}}$. Finally, suppose that the condition \eqref{podminka1} holds for any fixed $\mu_n.$

Then there exists $H_n \searrow 0$ such that
\begin{equation} \label{K1}
[\rNN, \mNN ] \to [\vr, \vm] \ \mbox{in}\ L^1((0,T) \times \Td; \mathds{R}^{d+1})\ \mbox{whenever} \ 0 < h_n \leq H_n,
\end{equation}
where $\delta_{[\vr, \vm]}$ is the unique viscosity solution of the Euler system.

\end{Theorem}

The proof is a combination of Theorem \ref{NT1} for $\mathcal{V}=\delta_{[\vr, \vm]}$ with Lemma \ref{sL8}. As a matter of fact, Lemma \ref{sL8} yields \eqref{K1}
only for a suitable subsequence, however, the convergence here is unconditional as the limit is unique.

\begin{Remark}
We point out that uniqueness of the limit and its independence of the
		choice of the sequence of the vanishing viscosity coefficients are imposed in Theorem \ref{K1}
		through the (KH) hypothesis.
	\end{Remark}

\begin{Remark}
Note that convergence is unconditional, meaning \eqref{podminka1} is not required and \eqref{K1} holds for any $H_n \to 0$ as soon as  the limit $[\vr,\vm] \in C^1.$ This follows from Theorem~\ref{thm_dmvs} and \cite[Theorem~5.9]{FeLMMiSh}.
\end{Remark}

Unfortunately, the strong convergence to a \emph{unique} solution is not always observed/expected for compressible fluid flows in the
zero viscosity (turbulent) regime. The appearance of the so--called carbuncles observed by Elling \cite{ELLI2} or the turbulent
wake areas in the obstacle problem are examples of numerous phenomena when, apparently, the convergence to a possible limit is not strong. Such a scenario is compatible with Kolmogorov K41 hypothesis only if the viscosity approximation admits a non--trivial
set of accumulation points. In particular, there are \emph{different} limits for \emph{different sequences $\mu_n \searrow 0$}.
Note that this is in sharp contrast with the situation anticipated in Section~\ref{VSES}, namely the existence of a \emph{unique} limit
of the approximate sequence in the \emph{weak} topology. The possibility of several limits would obviously invalidate the results of
any numerical simulation as the latter sensitively depends on the choice of artificial viscosity
and the associated numerical step. Even if S--convergence is applied, the sums
\[
\frac{1}{N} \sum_{n=1}^N \delta_{(\vr_n, \vm_n)(t,x)}
\]
may
depend on the choice of the approximate sequence.

The only piece of information that could be retained and computed is therefore a \emph{statistical} limit of the viscous
approximation proposed in \eqref{i1}. In contrast with the Young measure, the statistical limit is interpreted as a probability
measure on the space of solution trajectories, here
\[
X = \left\{ \vr, \vm \Big| \vr \in C_{\rm weak}([0,T]; L^\gamma(\Omega)),\
\vm \in C_{\rm weak}([0,T]; L^{\frac{2 \gamma}{\gamma + 1}}(\Omega; \mathds{R}^d)) \right\}.
\]
The ``statistical" convergence then can be stated as
\[
\frac{1}{N} \sum_{n=1}^N \delta_{(\vr_n, \vm_n)} \to \mathcal{V} \ \mbox{narrowly in}\ \mathfrak{P}[X],
\]
which is formally identical with the S--convergence. The standard tools of probability theory, notably the
Prokhorov theorem, will provide an analogue of the Subsequence principle stated for
S--convergence in Proposition \ref{sP3}, namely
\[
\frac{1}{N_k} \sum_{n=1}^{N_k} \delta_{(\vr_n, \vm_n)} \to \mathcal{V} \ \mbox{for a suitable subsequence}\ N_k \to \infty.
\]
Note carefully that here, in contrast with Proposition \ref{sP3}, the \emph{whole} sequence $\{ \vr_n, \vm_n \}_{n=1}^\infty$ is relevant
for the asymptotic limit.

As the statistical limit is associated with a specific choice of the approximate sequence of viscosities and/or initial data,
its approximation by the VFV method is possible under the hypotheses of Theorem \ref{NT1}. Specifically, we would have to \emph{fix}
the relevant sequence of the data $[\vr_{0,n}, \vm_{0,n}]$ together with the viscosities $\mu_n \searrow 0$ and to adjust the numerical
step $h$ so that $[h, \mu, 0] \in \overline{\mathcal{R}}$. This can be  CPU demanding since calculation and averaging of large amount of
numerical approximations must be realized to obtain a statistically reliable picture.

\section{Conclusion}

Anticipating the vanishing viscosity limit of the Navier--Stokes system as a physically relevant solution of the Euler system
we have identified a viscosity solution of the Euler system with a parametrized family of probability measures generated by solutions of the Navier–Stokes system in the vanishing viscosity limit.
We have introduced the observables as the quantities that are independent of a specific approximate sequence
and proposed a numerical scheme to compute efficiently the values of observables by means of a summation method. Here efficiently means that the approximate sequence converges strongly (in the $L^1$--topology).

Numerical solutions were obtained by the Viscosity Finite Volume method that is based on the standard finite volume method supplemented with vanishing numerical viscosity
in the spirit of the model proposed by H.~Brenner \cite{BREN2, BREN, BREN1}.
We have shown that VFV method identifies the
viscosity solution and/or the observables at least if:
\begin{itemize}
\item the numerical step $h$ is tuned, in fact considerably smaller than the artificial viscosity $\mu$;
\item the VFV method provides a sequence of approximate solutions that admits S--limit if both the numerical step $h$ and
the artificial viscosity $\mu$ approach zero.

\end{itemize}

In future our aim is to study S--convergence for the full Euler system of gas dynamics. A challenging question is to
investigate a connection of S--convergence and compressible turbulence.

\section*{Appendix}

Tables~\ref{tab1a} and~\ref{tab2a} present the convergence of weighted averages of the density computed on $k \times k$ meshes, with $k=32, \dots, N$, $N \in \{64,\, 96,\, 128,\,\dots,\, 2048\}.$ The reference
solution is the same for all summation methods and computed as  the Ces\`aro average of the solutions computed on $k \times k$ meshes with
$k \in \{32\ell \,{|}\,\ell \in \mathds{N},\, 1\leq \ell \leq 64\}\cup \{2304, 2560, 2816, 3072\}.$

 Table~\ref{tab2a} presents the convergence results for all weights $\omega$, except $\omega = \omega_{\text{equal}}$, since
the computed Ces\` aro averages are already very close to the reference solution and formally computed experimental order of convergence is not representative.
\begin{table}
\caption{Convergence study in the $L^1$-norm for averages of the density at time $T = 2$ using different weight functions and the
Ces\`aro average as a reference solution.}
\label{tab1a}
\begin{center}
	\begin{tabular}{|cV{3}c|cV{3}c|cV{3}c|cV{3}c|c|}
		\hline
		$k$&
		\multicolumn{2}{cV{3}}{$\omega_{\text{equal}}$}& \multicolumn{2}{cV{3}}{$\omega_{\text{quad}}$}& \multicolumn{2}{cV{3}}{$\omega_{\text{sin2}}$}& \multicolumn{2}{c|}{$\omega_{\text{exp}}$}\\
		\cline{2-9}
		(up to)& error   & order& error   & order& error   & order& error   & order\\
		\hline64 & 1.44e-01 & - & 1.92e-01 & - & 1.92e-01 & - & 1.92e-01 & - \\
		\hline
		96 & 1.16e-01 & 0.54 & 1.44e-01 & 0.71 & 1.44e-01 & 0.71 & 1.44e-01 & 0.71 \\
		\hline
		128 & 9.72e-02 & 0.62 & 1.14e-01 & 0.81 & 1.11e-01 & 0.91 & 1.07e-01 & 1.05 \\
		\hline
		160 & 8.35e-02 & 0.68 & 9.34e-02 & 0.90 & 8.87e-02 & 1.01 & 8.33e-02 & 1.10 \\
		\hline
		192 & 7.31e-02 & 0.73 & 7.83e-02 & 0.97 & 7.33e-02 & 1.05 & 6.89e-02 & 1.04 \\
		\hline
		224 & 6.50e-02 & 0.76 & 6.70e-02 & 1.01 & 6.23e-02 & 1.05 & 5.85e-02 & 1.05 \\
		\hline
		256 & 5.86e-02 & 0.78 & 5.85e-02 & 1.02 & 5.41e-02 & 1.06 & 5.07e-02 & 1.08 \\
		\hline
		288 & 5.33e-02 & 0.80 & 5.19e-02 & 1.02 & 4.77e-02 & 1.07 & 4.46e-02 & 1.09 \\
		\hline
		320 & 4.90e-02 & 0.81 & 4.66e-02 & 1.02 & 4.25e-02 & 1.08 & 3.98e-02 & 1.08 \\
		\hline
		352 & 4.53e-02 & 0.81 & 4.23e-02 & 1.03 & 3.84e-02 & 1.07 & 3.60e-02 & 1.04 \\
		\hline
		384 & 4.22e-02 & 0.81 & 3.87e-02 & 1.02 & 3.50e-02 & 1.06 & 3.30e-02 & 0.98 \\
		\hline
		416 & 3.96e-02 & 0.80 & 3.56e-02 & 1.02 & 3.22e-02 & 1.03 & 3.07e-02 & 0.92 \\
		\hline
		448 & 3.73e-02 & 0.80 & 3.31e-02 & 1.01 & 3.00e-02 & 0.99 & 2.88e-02 & 0.86 \\
		\hline
		480 & 3.53e-02 & 0.80 & 3.09e-02 & 0.99 & 2.81e-02 & 0.94 & 2.73e-02 & 0.79 \\
		\hline
		512 & 3.35e-02 & 0.79 & 2.90e-02 & 0.97 & 2.66e-02 & 0.87 & 2.60e-02 & 0.71 \\
		\hline
		544 & 3.20e-02 & 0.79 & 2.74e-02 & 0.94 & 2.53e-02 & 0.80 & 2.51e-02 & 0.62 \\
		\hline
		576 & 3.06e-02 & 0.78 & 2.60e-02 & 0.92 & 2.43e-02 & 0.71 & 2.43e-02 & 0.53 \\
		\hline
		608 & 2.93e-02 & 0.78 & 2.48e-02 & 0.88 & 2.35e-02 & 0.62 & 2.37e-02 & 0.46 \\
		\hline
		640 & 2.82e-02 & 0.78 & 2.38e-02 & 0.83 & 2.28e-02 & 0.54 & 2.33e-02 & 0.40 \\
		\hline
		672 & 2.71e-02 & 0.77 & 2.29e-02 & 0.78 & 2.23e-02 & 0.47 & 2.29e-02 & 0.35 \\
		\hline
		704 & 2.62e-02 & 0.76 & 2.21e-02 & 0.73 & 2.19e-02 & 0.41 & 2.26e-02 & 0.30 \\
		\hline
		736 & 2.53e-02 & 0.76 & 2.15e-02 & 0.68 & 2.16e-02 & 0.36 & 2.23e-02 & 0.25 \\
		\hline
		768 & 2.45e-02 & 0.77 & 2.09e-02 & 0.63 & 2.13e-02 & 0.32 & 2.21e-02 & 0.22 \\
		\hline
		800 & 2.37e-02 & 0.79 & 2.04e-02 & 0.58 & 2.10e-02 & 0.28 & 2.19e-02 & 0.19 \\
		\hline
		832 & 2.30e-02 & 0.78 & 2.00e-02 & 0.54 & 2.08e-02 & 0.25 & 2.18e-02 & 0.16 \\
		\hline
		864 & 2.23e-02 & 0.80 & 1.96e-02 & 0.51 & 2.06e-02 & 0.22 & 2.17e-02 & 0.13 \\
		\hline
		896 & 2.17e-02 & 0.80 & 1.93e-02 & 0.47 & 2.05e-02 & 0.20 & 2.16e-02 & 0.10 \\
		\hline
		928 & 2.11e-02 & 0.82 & 1.90e-02 & 0.44 & 2.04e-02 & 0.18 & 2.16e-02 & 0.07 \\
		\hline
		960 & 2.05e-02 & 0.84 & 1.87e-02 & 0.41 & 2.03e-02 & 0.16 & 2.15e-02 & 0.05 \\
		\hline
		992 & 1.99e-02 & 0.87 & 1.85e-02 & 0.39 & 2.02e-02 & 0.15 & 2.15e-02 & 0.04 \\
		\hline
		1024 & 1.94e-02 & 0.89 & 1.83e-02 & 0.38 & 2.01e-02 & 0.14 & 2.15e-02 & 0.03 \\
		\hline
		1056 & 1.88e-02 & 0.90 & 1.81e-02 & 0.37 & 2.00e-02 & 0.13 & 2.15e-02 & 0.02 \\
		\hline
	\end{tabular}
\end{center}
\end{table}
\begin{table}
	\caption{Continuation: Convergence study in the $L^1$-norm for averages of the density at time $T = 2$ using different weight functions
and the Ces\`aro average as a reference solution.}
	\label{tab2a}
	\begin{center}
		\begin{tabular}{|cV{3}c|cV{3}c|cV{3}c|c|}
			\hline
			$k$ &
			\multicolumn{2}{cV{3}}{$\omega_{\text{quad}}$}& \multicolumn{2}{cV{3}}{$\omega_{\text{sin2}}$}& \multicolumn{2}{c|}{$\omega_{\text{exp}}$}\\
			\cline{2-7}
			(up to)& error   & order& error   & order& error   & order\\
			\hline
			1088 &  1.79e-02 & 0.36 & 1.99e-02 & 0.13 & 2.14e-02 & 0.02 \\
			\hline
			1120 &1.77e-02 & 0.36 & 1.98e-02 & 0.12 & 2.14e-02 & 0.03 \\
			\hline
			1152 & 1.75e-02 & 0.36 & 1.98e-02 & 0.12 & 2.14e-02 & 0.03 \\
			\hline
			1184 & 1.73e-02 & 0.37 & 1.97e-02 & 0.12 & 2.14e-02 & 0.04 \\
			\hline
			1216 & 1.71e-02 & 0.37 & 1.96e-02 & 0.13 & 2.13e-02 & 0.06 \\
			\hline
			1248 & 1.70e-02 & 0.38 & 1.96e-02 & 0.14 & 2.13e-02 & 0.07 \\
			\hline
			1280 & 1.68e-02 & 0.39 & 1.95e-02 & 0.15 & 2.13e-02 & 0.09 \\
			\hline
			1312 & 1.66e-02 & 0.40 & 1.94e-02 & 0.17 & 2.12e-02 & 0.11 \\
			\hline
			1344 & 1.65e-02 & 0.41 & 1.93e-02 & 0.19 & 2.11e-02 & 0.13 \\
			\hline
			1376 & 1.63e-02 & 0.43 & 1.92e-02 & 0.21 & 2.11e-02 & 0.15 \\
			\hline
			1408 & 1.61e-02 & 0.45 & 1.91e-02 & 0.24 & 2.10e-02 & 0.17 \\
			\hline
			1440 & 1.60e-02 & 0.48 & 1.90e-02 & 0.27 & 2.09e-02 & 0.19 \\
			\hline
			1472 & 1.58e-02 & 0.50 & 1.89e-02 & 0.30 & 2.08e-02 & 0.21 \\
			\hline
			1504 & 1.56e-02 & 0.53 & 1.87e-02 & 0.33 & 2.07e-02 & 0.24 \\
			\hline
			1536 & 1.54e-02 & 0.57 & 1.86e-02 & 0.37 & 2.06e-02 & 0.26 \\
			\hline
			1568 & 1.52e-02 & 0.61 & 1.84e-02 & 0.40 & 2.05e-02 & 0.30 \\
			\hline
			1600 & 1.50e-02 & 0.66 & 1.83e-02 & 0.44 & 2.03e-02 & 0.33 \\
			\hline
			1632 & 1.48e-02 & 0.71 & 1.81e-02 & 0.49 & 2.02e-02 & 0.36 \\
			\hline
			1664 & 1.46e-02 & 0.76 & 1.79e-02 & 0.53 & 2.00e-02 & 0.40 \\
			\hline
			1696 & 1.44e-02 & 0.81 & 1.77e-02 & 0.57 & 1.98e-02 & 0.44 \\
			\hline
			1728 & 1.42e-02 & 0.87 & 1.75e-02 & 0.62 & 1.97e-02 & 0.47 \\
			\hline
			1760 & 1.39e-02 & 0.93 & 1.73e-02 & 0.66 & 1.95e-02 & 0.51 \\
			\hline
			1792 & 1.37e-02 & 0.99 & 1.71e-02 & 0.71 & 1.93e-02 & 0.55 \\
			\hline
			1824 & 1.34e-02 & 1.05 & 1.69e-02 & 0.75 & 1.91e-02 & 0.59 \\
			\hline
			1856 & 1.32e-02 & 1.11 & 1.66e-02 & 0.80 & 1.89e-02 & 0.62 \\
			\hline
			1888 & 1.29e-02 & 1.18 & 1.64e-02 & 0.85 & 1.87e-02 & 0.66 \\
			\hline
			1920 & 1.26e-02 & 1.25 & 1.61e-02 & 0.90 & 1.85e-02 & 0.70 \\
			\hline
			1952 & 1.24e-02 & 1.32 & 1.59e-02 & 0.94 & 1.82e-02 & 0.73 \\
			\hline
			1984 & 1.21e-02 & 1.39 & 1.56e-02 & 0.99 & 1.80e-02 & 0.77 \\
			\hline
			2016 & 1.18e-02 & 1.46 & 1.54e-02 & 1.04 & 1.78e-02 & 0.80 \\
			\hline
			2048 & 1.15e-02 & 1.52 & 1.51e-02 & 1.08 & 1.76e-02 & 0.83 \\
			\hline
		\end{tabular}
	\end{center}
\end{table}

\newpage

\def\cprime{$'$} \def\ocirc#1{\ifmmode\setbox0=\hbox{$#1$}\dimen0=\ht0
  \advance\dimen0 by1pt\rlap{\hbox to\wd0{\hss\raise\dimen0
  \hbox{\hskip.2em$\scriptscriptstyle\circ$}\hss}}#1\else {\accent"17 #1}\fi}

\end{document}